\documentclass[final,3p,times]{elsarticle}
\usepackage{amssymb}
\usepackage{amsmath}
\usepackage{graphicx}
\usepackage{epstopdf}

\usepackage{pdfpages}
\usepackage{amssymb}
\usepackage{empheq}
\usepackage{cases}
\usepackage{amsthm,amsmath}
\usepackage{caption,lipsum}
\usepackage{stmaryrd}
\usepackage{tabularx}
\usepackage{booktabs}
\usepackage{empheq,float}
\usepackage{subfigure}
\DeclareGraphicsExtensions{.eps}
\usepackage{tikz}
\usepackage{algorithm}
\usepackage{algorithmicx}
\usepackage{algpseudocode}
\usepackage{manyfoot}
\usepackage{mathrsfs} %增加
\usepackage{verbatim}
\usepackage{amsopn}
\usetikzlibrary{arrows,chains,matrix,positioning,scopes}
\makeatletter
\tikzset{join/.code=\tikzset{after node path={%
\ifx\tikzchainprevious\pgfutil@empty\else(\tikzchainprevious)%
edge[every join]#1(\tikzchaincurrent)\fi}}}
\makeatother
\tikzset{>=stealth',every on chain/.append style={join},
         every join/.style={->}}
\tikzstyle{labeled}=[execute at begin node=$\scriptstyle,
   execute at end node=$]

\usepackage{hyperref}
\hypersetup{
    colorlinks=true,
    linkcolor=blue,
    citecolor=blue,
    urlcolor=blue
}

\numberwithin{equation}{section}

\newtheorem{theorem}{Theorem}[section]
\newtheorem{lemma}[theorem]{Lemma}
\newtheorem{proposition}[theorem]{Proposition}

\theoremstyle{definition}

\newtheorem{example}[theorem]{Example}

\newtheorem{remark}[theorem]{Remark}

\newcommand{\bx}{{\boldsymbol x}} %\mathbf直立体, \boldsymbol斜体
\newcommand{\bv}{{\boldsymbol v}}
\newcommand{\by}{{\boldsymbol y}}
\newcommand{\bu}{{\boldsymbol u}} 
\newcommand{\mh}{{\mathfrak h}}

\newcommand{\bB}{{\boldsymbol B}}
\newcommand{\bE}{{\boldsymbol E}}
\newcommand{\bH}{{\boldsymbol H}}
\newcommand{\bF}{{\boldsymbol F}}
\newcommand{\bK}{{\boldsymbol K}}

\newcommand{\eps}{\varepsilon}

\begin{document}

\begin{frontmatter}



\title{A multi-physics structure-preserving integrator with uniform error bounds for relativistic charged-particle dynamics under strong magnetic fields}


\author[1] {Mengting Hu} %% Author name
           \ead{humting@stu.xjtu.edu.cn}
           
\author[2] {Yifa Tang} %% Author name
           \ead{tyf@lsec.cc.ac.cn}
           
\author[1] {Bin Wang\corref{cor1}} %% Author name
           \ead{wangbinmaths@xjtu.edu.cn}

\cortext[cor1]{Corresponding author}

\affiliation[1]{organization={School of Mathematics and Statistics, Xi'an Jiaotong University},
            postcode={710049},
            city={Xi'an},
            country={China}}
            
\affiliation[2]{organization={State Key Laboratory of Mathematical Sciences, Academy of Mathematics and Systems Science, Chinese Academy of Sciences},
            postcode={100190},
            city={Beijing},
            country={China}}

\begin{abstract}
In this paper, we develop an explicit multi-physics structure-preserving Strang splitting scheme for a four-dimensional relativistic charged-particle dynamical system in the presence of a strong magnetic field under the maximal ordering scaling. The proposed scheme not only preserves volume, energy, and Lorentz invariance, but also yields second-order uniform error bounds for the position and the parallel velocity component. We present rigorous theoretical proofs for these geometric properties and the error convergence, which are subsequently validated by several numerical experiments.

\end{abstract}

%% Keywords
\begin{keyword}
Relativistic charged-particle dynamics \sep Strong magnetic fields \sep  Multi-physics structure-preserving scheme \sep Uniform error bounds



\end{keyword}

\end{frontmatter}
%{\bf AMS Subject Classification:} 65L05, 65L20, 65L70


\section{Introduction}
Relativistic charged-particle dynamics (RCPD) describes the motion of charged particles moving at velocities close to the speed of light in electromagnetic fields, which finds wide applications in high-energy accelerators, magnetically confined fusion, astrophysics and plasma numerical simulation \cite{RevModPhys1949,Boris1970,PhysRevE2008,Ripperda2018}. Different from the classical non-relativistic model, this system is formulated within the four-dimensional Minkowski spacetime. Its governing equations satisfy Lorentz covariance, and the system intrinsically admits several invariants including mass shell conservation, energy conservation and phase-space volume preservation \cite{Jackson1998,Morales2017,Hairer2023}.
Due to the multiple conservation laws and multiple time-scale nature of the system, long-time numerical simulation of this relativistic system poses great challenges \cite{Vay2008}. 
Particularly for problems in the strong magnetic field regime where a small parameter is introduced, traditional methods either fail to exactly preserve the geometric structures and physical invariants, or the error bounds of existing structure-preserving schemes generally depend on this small parameter. 
Therefore, it is of great research significance to construct explicit and efficient numerical schemes that maintain multiple physical properties and achieve uniform error bounds independent of the small parameter. 

The three-dimensional relativistic charged-particle dynamical system in a strong magnetic field takes the form
\begin{equation}
	\begin{cases}	
		\displaystyle \dot{\bx}(t) = \frac{\bv(t)}{\gamma}, 
		\vspace{0.5em} \\ \displaystyle 
		\dot{\bv}(t) = \frac{\bv(t)}{\gamma} \times \frac{\bB(\varepsilon \bx(t))}{\varepsilon} + \bE(\bx(t)), 
		\vspace{0.3em} \\
		\gamma = \sqrt{1 + \left\|\bv\right\|^2}.
	\end{cases}
	\label{Rcpd}
\end{equation}
Here $\bx(t), \bv(t) \in \mathbb{R}^3$ denote the position and momentum at time $t$, respectively, and $\gamma$ stands for the relativistic factor.
The electric field $\bE=(E_1,E_2,E_3)^{\intercal}$ is derived from $\bE(\bx)= - \nabla U(\bx)$.
Let $\bB=(B_1,B_2,B_3)^{\intercal}$ be the magnetic field. The quantity $\mathcal{B}(\bx) = \bB(\varepsilon \bx)/\varepsilon$ represents the strong magnetic field under the maximal ordering scaling (MOS), where the small parameter satisfies $0<\varepsilon\ll 1$ and characterizes the field strength.
The MOS for strong magnetic field $\mathcal{B}(\bx)$ \cite{Rev,Poss2018} is more restrictive than the ordinary strong magnetic field $\bB(\bx)/\varepsilon$. It additionally requires two conditions $\left\|\bE\right\|/(c\left\|\mathcal{B}\right\|)\sim\varepsilon\ll1$ and $\rho\left\|\nabla\mathcal{B}\right\|/\left\|\mathcal{B}\right\|\sim\varepsilon\ll1$, where $c$ is the speed of light and $\rho$ denotes the particle gyroradius.

The non-relativistic charged-particle dynamical system under the maximal ordering scaling (MOS) for strong magnetic fields serves as a fundamental model for simulating multi-scale particle motion in magnetically confined fusion devices \cite{Lee1983,Nicolas2018,Filbet2017,Filbet2018}. To accurately capture the long-term behavior of this system, various structure-preserving algorithms have been developed \cite{Emmanuel2015,Hairer2020,Huang2020}. However, due to the presence of the small parameter $\varepsilon$, most of these conventional methods suffer from accuracy deterioration as $\varepsilon \to 0$. To overcome this $\varepsilon$-dependence, significant efforts have been made to construct uniformly accurate (UA) schemes. Notably, UA schemes for Vlasov equations based on the particle-in-cell (PIC) method were introduced in \cite{Nicolas2017,Chartier2020}, and UA algorithms for general non-relativistic dynamical systems were further developed in \cite{Wang2023}. Additionally, strategies such as the filtered Boris algorithm and filtered variational integrators were investigated in \cite{filter1,filter2} to mitigate similar numerical stiffness.
In contrast to the non-relativistic setting, numerical studies of relativistic charged-particle dynamics (RCPD) initially focused on three-dimensional formulations \cite{Zhang2015,He2016,Higuera2017}, and were subsequently extended to the four-dimensional Minkowski spacetime \cite{Wang2016}, yielding a variety of structure-preserving schemes  \cite{Akinobu2017,Zhang2018,Jianyuan2019,Yulei2021}.
However, to the best of our knowledge, error estimation for RCPD under the MOS of strong magnetic fields remains quite limited.
The only relevant attempt is found in \cite{Ruili2024}, where a splitting method was employed to construct the structure-preserving scheme VELPA, achieving first-order uniform error bounds for the relativistic system under the MOS.
However, this uniform accuracy property does not extend to its second-order counterpart, VELPA2; specifically, its error bound fails to be second-order uniformly accurate and degrades as $\varepsilon$ decreases. This limitation highlights the critical need for developing higher-order schemes that can maintain uniform accuracy in the strongly magnetized relativistic regime.

To address this limitation, in this work we construct and analyze a novel explicit multi-physics structure-preserving numerical scheme for the four-dimensional relativistic system \eqref{Rcpd}. By designing an innovative splitting strategy, we develop a new class of explicit schemes based on the Strang splitting method. Rigorous error analysis demonstrates that the proposed scheme achieves second-order uniform error bounds under the given magnetic field condition for $\nabla \bB(0)$. The proposed approach exhibits the following key features:
\begin{itemize}
	\item The integrator is \emph{fully explicit}, requiring no implicit or iterative solvers. All matrix exponential operators admit exact, closed-form representations, avoiding numerical approximations of matrix actions and resulting in a highly efficient scheme with minimal computational cost per step.
	
	\item Constructed within the framework of geometric numerical integration, the scheme exactly preserves various physical properties, such as the phase-space volume, the Hamiltonian energy, and the Lorentz invariance. These \emph{intrinsic conservation laws} ensure that the numerical flow maintains the underlying geometric structure, suppressing artificial energy growth over long-time integrations.

	\item Under the assumption of a weak magnetic field gradient at the initial state, we rigorously prove that the scheme satisfies \emph{second-order uniform error bounds} for both the position and the parallel velocity component. Notably, these error bounds remain independent of the strong magnetic field strength, providing a robust theoretical guarantee of accuracy for the highly oscillatory dynamics.
\end{itemize}

The remainder of this paper is organized as follows. Section \ref{sec2} introduces the splitting scheme and presents the main theoretical results. Section \ref{sec3} provides the rigorous proof of the second-order uniform error bounds for the proposed scheme. Several numerical experiments are demonstrated in Section \ref{sec4}. Finally, Section \ref{sec5} concludes the paper.

\section{The multi-physics structure-preserving scheme}\label{sec2}

In this section, we construct the numerical scheme and establish its convergence rate. Theoretical analyses of volume preservation, energy conservation, and Lorentz invariance are carried out. Throughout this work, the symbol \(A \lesssim B\) stands for \(A \le C B\) with a positive constant \(C\) independent of the time step size \(h\), the time step index \(n\) and the parameter \(\varepsilon\).

\subsection{Construction of the scheme}\label{subsec2.1}

To formulate the four-dimensional relativistic charged-particle dynamical system, we introduce the proper time $\tau$ for relativistic system \eqref{Rcpd}, which yields $dt/d \tau = \gamma$.
We further define two imaginary variables $\bar{t} = it$ and $w = i \gamma$. Then system \eqref{Rcpd} can be rewritten as
\begin{equation}
	\begin{cases}	
		\displaystyle \dot{\bx}(\tau) = \bv(\tau), \\
		\dot{\bar{t}}(\tau) = w(\tau), \\ \displaystyle 
		\dot{\bv}(\tau) =  \frac{\widehat{\bB}(\varepsilon \bx(\tau))}{\varepsilon} \bv(\tau) - iw(\tau)\bE(\bx(\tau)), \\
		\dot{w}(\tau) = i\bE(\bx(\tau))^{\intercal} \cdot \bv(\tau),
		\quad 0 < \tau \leq T,
	\end{cases}
	\label{Rcpd2}
\end{equation}
where  
	$\widehat{\bB} = 
	\begin{pmatrix}
		0 & B_{3} & -B_{2} \\ 
		-B_{3} & 0 & B_{1} \\ 
		B_{2} & -B_{1} & 0
	\end{pmatrix}$
is the skew-symmetric matrix associated with $\bB$.
Set $\by = (\bx^{\intercal},\bar{t})^{\intercal}$ and $\bu = (\bv^{\intercal},w)^{\intercal} \in \mathbb{R}^4$, and define the $4\times 4$ skew-symmetric matrix
\( \bF(\bx) = \frac{\widehat{\widehat{\bB}}(\varepsilon \bx)}{\varepsilon} + \widehat{\bE}(\bx) \),
with
\begin{equation*}
	\widehat{\widehat{\bB}}(\varepsilon \bx) = 
	\begin{pmatrix}
		\widehat{\bB}(\varepsilon \bx) & \mathbf{0} \\
		\mathbf{0} & 0
	\end{pmatrix},
    \quad
	\widehat{\bE}(\bx) = 
	\begin{pmatrix}
		\mathbf{0} & -i\bE(\bx) \\
		i\bE(\bx)^{\intercal} & 0
	\end{pmatrix}.
\end{equation*}
Accordingly, system \eqref{Rcpd2} is equivalent to
\begin{equation}
	\begin{cases}	
		\displaystyle \dot{\by}(\tau) = \bu(\tau), \\
		\dot{\bu}(\tau) =  \bF(\bx(\tau)) \bu(\tau),
		\quad 0 < \tau \leq T.
	\end{cases}
	\label{Rcpd3}
\end{equation}
One can readily verify that the relativistic charged-particle dynamical system in the eight-dimensional space is volume-preserving, energy-preserving and Lorentz covariant.
For system \eqref{Rcpd2}, the energy functional is defined as
\begin{equation*}
	\bH(\bx,\bar{t},\bv,w) = \frac{1}{2}\left\|\bv\right\|^2 + \frac{1}{2}w^2.
\end{equation*}

Set \(h = \Delta \tau >0\) as the time step, and the grid points are given by \(\tau_n = nh\) for \(n\in\mathbb{N}\). Let \(\by(0) = (\bx(0)^{\intercal},\bar{t}(0))^{\intercal} = (\bx_{0}^{\intercal},\bar{t}_0)^{\intercal} = \by_0\) and \(\bu(0) = (\bv(0)^{\intercal},w(0))^{\intercal} = (\bv_{0}^{\intercal},w_0)^{\intercal} = \bu_0\) denote the initial values of system \eqref{Rcpd3}.
Here \(\by^n\) and \(\bu^n\) stand for the numerical approximations to the exact solution \(\big(\by(\tau_n), \bu(\tau_n)\big)\). By fixing the position at the current grid point, we decompose system \eqref{Rcpd3} into the following two subsystems:
\begin{equation*} 
	\mathscr{A}_{\bx^{n}}: \begin{pmatrix} \dot{\by}(\tau)\\ \dot{\bu}(\tau) \end{pmatrix} =
	\begin{pmatrix}
		\bu(\tau) \\
		\bF(\bx^{n})\bu(\tau)
	\end{pmatrix}, \quad	
	\mathscr{B}_{\bx^{n}}: \begin{pmatrix} \dot{\by}(\tau)\\ \dot{\bu}(\tau) \end{pmatrix}=
	\begin{pmatrix}
		0 \\
		\big( \bF(\bx) - \bF(\bx^{n}) \big) \bu(\tau)
	\end{pmatrix}.
\end{equation*}

Let $\varphi_{\tau}^{\mathscr{A}_{\bx^{n}}}$ and $\varphi_{\tau}^{\mathscr{B}_{\bx^{n}}}$ denote the exact flows corresponding to the two subsystems. 
The exact evolution operator of the first subsystem is given by
\begin{equation*}
	\varphi_{h}^{\mathscr{A}_{\bx^{n}}} \begin{pmatrix} \by(\tau_{n}) \\ \bu(\tau_{n}) \end{pmatrix} := 
	\begin{pmatrix}
		\by(\tau_{n}) + h\varphi_{1}\bigl( h\bF(\bx^{n}) \bigr) \bu(\tau_{n}) \\[2pt]
		\mathrm{e}^{h \bF(\bx^{n}) } \bu(\tau_{n})
	\end{pmatrix},
\end{equation*}
where \(\varphi_{1}(s)=(\mathrm{e}^{s}-1)/s\). 
Similarly, the exact evolution operator of the second subsystem reads
\begin{equation*}
	\varphi_{h}^{\mathscr{B}_{\bx^{n}}} \begin{pmatrix} \by(\tau_{n}) \\ \bu(\tau_{n}) \end{pmatrix} := 
	\begin{pmatrix}
		\by(\tau_{n}) \\
		\mathrm{e}^{h \big( \bF(\bx(\tau_n)) - \bF(\bx^{n}) \big) } \bu(\tau_{n})
	\end{pmatrix}.
\end{equation*}

\noindent\textbf{Algorithm 2.1.}  
Composing these exact evolution operators yields the Strang splitting scheme
$\Psi_{h} = \varphi_{h/2}^{\mathscr{A}_{\bx^{n}}} \circ \varphi_{h}^{\mathscr{B}_{\bx^{n}}} \circ \varphi_{h/2}^{\mathscr{A}_{\bx^{n}}}$ for system \eqref{Rcpd3}. 
The explicit formulation of this scheme is
\begin{equation}
	\begin{cases}
		\begin{aligned}
			\displaystyle \by^{n+1} &= \,\, \by^{n}+\frac{h}{2}\varphi_{1}\biggl( \frac{h}{2} \bF^{n} \biggr)
			\Bigl(I_4 + \mathrm{e}^{h (\bF_{\bar{\bx}^{n}} - \bF^{n}) }
			\mathrm{e}^{\frac{h}{2}\bF^{n}}\Bigr) \bu^{n},   \\
			\displaystyle \bu^{n+1} &= \,\, \mathrm{e}^{\frac{h}{2}\bF^{n}}
			\mathrm{e}^{h (\bF_{\bar{\bx}^{n}} - \bF^{n})}
			\mathrm{e}^{\frac{h}{2}\bF^{n}} \bu^{n},  
			\quad 0\leq n < \frac{T}{h},
		\end{aligned}
	\end{cases}
	\label{SS2}
\end{equation}
where the related quantities are defined as
\begin{equation*}
	\bar{\by}^{n}=\by^{n}+\frac{h}{2}\varphi_{1}\biggl(\frac{h}{2}\bF^{n}\biggr) \bu^{n}, \quad \bar{\bx}^{n} = \bar{\by}^{n}(1:3),
	\quad \bF^{n} = \bF(\bx^{n}), \quad \bF_{\bar{\bx}^{n}} = \bF(\bar{\bx}^{n}).
\end{equation*}
For simplicity, we refer to scheme \eqref{SS2} as SS2-xn.

We note that although the exponential and $\varphi_{1}$ functions are matrix-valued, they admit exact and closed-form representations. The results are stated as follows.

Based on the properties of the skew-symmetric matrix \(\bF(\bx)\), the explicit expression of the matrix exponential \(\mathrm{e}^{h\bF(\bx)}\) is provided in the appendix of \cite{Ruili2024}, which takes the form:
\begin{equation}
	\mathrm{e}^{h\bF(\bx)} = S_1(h)\bF(\bx) + S_2(h)\widehat{\bF}(\bx) 
	+ S_3(h)I_4 + S_4(h)\bF^2(\bx), \label{eF}
\end{equation}
with coefficients
\begin{align*}
	S_1(h) & = - \frac{\sqrt{l_1-\Delta}}{\sqrt{2}\Delta} \sinh \Big( \frac{h\sqrt{l_1-\Delta}}{\sqrt{2}} \Big) 
	+ \frac{\sqrt{l_1+\Delta}}{\sqrt{2}\Delta} \sinh \Big( \frac{h\sqrt{l_1+\Delta}}{\sqrt{2}} \Big), \\
	S_2(h) & = - \frac{\sqrt{2}l_2}{\sqrt{l_1-\Delta}\Delta} \sinh \Big( \frac{h\sqrt{l_1-\Delta}}{\sqrt{2}} \Big) 
	+ \frac{\sqrt{2}l_2}{\sqrt{l_1+\Delta}\Delta} \sinh \Big( \frac{h\sqrt{l_1+\Delta}}{\sqrt{2}} \Big), \\
	S_3(h) & = \Big( \frac{1}{2} + \frac{l_1}{2\Delta} \Big) \cosh \Big( \frac{h\sqrt{l_1-\Delta}}{\sqrt{2}} \Big)
	+ \Big( \frac{1}{2} - \frac{l_1}{2\Delta} \Big) \cosh \Big( \frac{h\sqrt{l_1+\Delta}}{\sqrt{2}} \Big), \\
	S_4(h) & = - \frac{1}{\Delta} \cosh \Big( \frac{h\sqrt{l_1-\Delta}}{\sqrt{2}} \Big) 
	+ \frac{1}{\Delta} \cosh \Big( \frac{h\sqrt{l_1+\Delta}}{\sqrt{2}} \Big),
\end{align*}
where \( l_1 = \left\|\bE(\bx)\right\|^2 - \left\|\frac{\bB(\eps\bx)}{\eps}\right\|^2 \), 
\( l_2 = - (\bE(\bx),\frac{\bB(\eps\bx)}{\eps}) (\ne 0) \), 
\( \Delta = \sqrt{l_1^2 + 4l_2^{2}} \),
and 
\begin{equation*}
	\widehat{\bF}(\bx) = \frac{1}{\eps}
	\begin{pmatrix}
		0 & 0 & 0 & iB_{1}(\eps\bx) \\
		0 & 0 & 0 & iB_{2}(\eps\bx) \\
		0 & 0 & 0 & iB_{3}(\eps\bx) \\
		-iB_{1}(\eps\bx) & -iB_{2}(\eps\bx) & -iB_{3}(\eps\bx) & 0 
	\end{pmatrix} 
	+ 
	\begin{pmatrix}
		0 & E_{3}(\bx) & -E_{2}(\bx) & 0 \\ 
		-E_{3}(\bx) & 0 & E_{1}(\bx) & 0 \\ 
		E_{2}(\bx) & -E_{1}(\bx) & 0 & 0 \\
		0 & 0 & 0 & 0
	\end{pmatrix}.
\end{equation*}
We now derive the explicit expression of the operator \(\varphi_1\big( h\bF(\bx) \big)\). To this end, we impose the following ansatz:
\begin{equation}
	\varphi_1\big( h\bF(\bx) \big) = T_1(h)\bF(\bx) + T_2(h)\widehat{\bF}(\bx) 
	+ T_3(h)I_4 + T_4(h)\bF^2(\bx).  \label{eT}
\end{equation}
Accordingly, the core problem reduces to solving for the coefficients \(T_1(h)\), \(T_2(h)\), \(T_3(h)\), \(T_4(h)\). 
Using the definition \( \varphi_1 \big( h\bF(\bx) \big) = \frac{\mathrm{e}^{h\bF(\bx)} - I_4}{h\bF(\bx)} \), together with \eqref{eT} and the identities
\( \bF^3(\bx) = l_1\bF(\bx) + l_2\widehat{\bF}(\bx) \), \( \bF(\bx) \widehat{\bF}(\bx) = l_2I_4 \), we have 
\begin{align*}
	\mathrm{e}^{h\bF(\bx)} & = h\bF(\bx) \varphi_1 \big( h\bF(\bx) \big) + I_4 
    = h\bF(\bx) \big( T_1(h)\bF(\bx) + T_2(h)\widehat{\bF}(\bx) 
	+ T_3(h)I_4 + T_4(h)\bF^2(\bx) \big) + I_4 \\
	& = \big( hT_3(h) + hl_1T_4(h) \big) \bF(\bx) + hl_2T_4(h) \widehat{\bF}(\bx) + \big( hl_2T_2(h) + 1 \big) I_4 + hT_1(h) \bF^2(\bx).
\end{align*}
Combining this with the coefficients in \eqref{eF}, we obtain the following system of coefficient relations:
\begin{equation*}
	\begin{cases}	
		hT_3(h) + hl_1T_4(h) = S_1(h), \ \
		hl_2T_4(h) = S_2(h), \\
		hl_2T_2(h) + 1 = S_3(h), \ \
		hT_1(h) = S_4(h),
	\end{cases}
\end{equation*}
from which we solve for the coefficients of $\varphi_1\big( h\bF(\bx) \big)$ as
\begin{equation*}
	T_1(h) = \frac{S_4(h)}{h}, \quad T_2(h) = \frac{S_3(h)-1}{hl_2}, \quad
	T_3(h) = \frac{S_1(h)}{h} - \frac{l_1S_2(h)}{hl_2}, \quad
	T_4(h) = \frac{S_2(h)}{hl_2}.
\end{equation*}

\subsection{Main results}\label{subsec2.2}

In this part, we first establish the multi-physics structure-preserving properties of the scheme SS2-xn \eqref{SS2}, including volume preservation, energy conservation, and Lorentz invariance.

\begin{proposition}\label{Volume} (\textbf{Volume preservation}.)
	The SS2-xn scheme \eqref{SS2} is volume-preserving.
\end{proposition}

\begin{proof}
	We work in the eight-dimensional phase space $(\by,\bu)$. The vector field associated with $\mathscr{A}_{\bx^n}$ reads
	$
	R_{\mathscr{A}}(\by,\bu) = \bigl(\bu,\bF(\bx^n)\, \bu\bigr)^{\intercal}.
	$
	Note that $\bx^n$ is fixed at the start of each time step, so $\bF(\bx^n)$ is a constant skew-symmetric matrix. The divergence of this vector field is
	\begin{equation*}
		\nabla_{(\by,\bu)} \cdot R_{\mathscr{A}} = \nabla_{\by} \cdot \bu + \nabla_{\bu} \cdot \bigl(\bF(\bx^n)\, \bu\bigr) = 0 + \operatorname{tr}\,\bigl(\bF(\bx^n)\bigr).
	\end{equation*}
	All diagonal entries of a skew-symmetric matrix vanish, so $\operatorname{tr}\,(\bF(\bx^n))=0$. Consequently, $\nabla \cdot R_{\mathscr{A}}=0$. By Liouville's theorem, the flow $\varphi_{\tau}^{\mathscr{A}_{\bx^n}}$ is volume-preserving.
	
	The vector field of $\mathscr{B}_{\bx^n}$ is given by
	$
	R_{\mathscr{B}}(\by,\bu) = \bigl(\boldsymbol{0}, \bigl(\bF(\bx)-\bF(\bx^n)\bigr)\, \bu\bigr)^{\intercal}.
	$
	Here $\bx$ is a component of $\by$ and hence depends only on $\by$, not on $\bu$. When computing the divergence with respect to $(\by,\bu)$, the matrix $\bF(\bx)-\bF(\bx^n)$ is independent of $\bu$ and can be treated as constant. We then compute
	\begin{equation*}
		\nabla_{(\by,\bu)} \cdot R_{\mathscr{B}} = \nabla_{\by} \cdot \boldsymbol{0} + \nabla_{\bu} \cdot \bigl(\bigl(\bF(\bx)-\bF(\bx^n)\bigr) \, \bu\bigr) = \operatorname{tr}\, \bigl(\bF(\bx)-\bF(\bx^n)\bigr).
	\end{equation*}
	The difference of two skew-symmetric matrices is still skew-symmetric, so its trace equals zero. This yields $\nabla \cdot R_{\mathscr{B}}=0$, which implies that $\varphi_{\tau}^{\mathscr{B}_{\bx^n}}$ also preserves volume.
	
	Since both $\varphi_{h/2}^{\mathscr{A}_{\bx^n}}$ and $\varphi_{h}^{\mathscr{B}_{\bx^n}}$ are volume-preserving, their composition $\Psi_h = \varphi_{h/2}^{\mathscr{A}_{\bx^{n}}} \circ \varphi_{h}^{\mathscr{B}_{\bx^{n}}} \circ \varphi_{h/2}^{\mathscr{A}_{\bx^{n}}}$ preserves volume as well. This completes the proof.
\end{proof}

\begin{proposition}\label{Energy} (\textbf{Energy preservation}.) 
	The SS2-xn scheme \eqref{SS2} is energy-preserving.
\end{proposition}

\begin{proof}
	Both $\bF^{n}$ and $\bF_{\bar{\bx}^n}$ are real skew-symmetric matrices, so their difference $\bF_{\bar{\bx}^n} - \bF^{n}$ is also real skew-symmetric. The matrix exponential of a real skew-symmetric matrix is an orthogonal matrix, which satisfies $Q^{\intercal} = Q^{-1}$ and preserves the Euclidean norm of real vectors. Thus $\mathrm{e}^{\frac{h}{2}\bF^{n}}$ and $\mathrm{e}^{h(\bF_{\bar{\bx}^n}-\bF^{n})}$ are norm-preserving orthogonal operators.
	
	According to the update rule of the SS2-xn scheme \eqref{SS2}, the vector $\bu = (\bv^{\intercal},w)^{\intercal}$ evolves via the product of the above matrix exponentials. Combining with the energy functional, we derive
	\begin{align*}
		& H(\bx^{n+1},\bar{t}^{\, n+1},\bv^{n+1},w^{n+1}) = \frac{1}{2}\left\|\bv^{n+1}\right\|^2 + \frac{1}{2} (w^{n+1})^2 
		= \frac{1}{2} 
		\begin{pmatrix}\bv^{n+1} \\ w^{n+1}\end{pmatrix}^{\intercal} \begin{pmatrix}\bv^{n+1} \\ w^{n+1}\end{pmatrix} \\
		&\quad = \frac{1}{2} 
		\begin{pmatrix}\bv^{n} \\ w^{n} \end{pmatrix}^{\intercal}
		\Big( \mathrm{e}^{\frac{h}{2}\bF^{n}}
		\mathrm{e}^{h (\bF_{\bar{\bx}^{n}} - \bF^{n})}
		\mathrm{e}^{\frac{h}{2}\bF^{n}} \Big)^{\intercal} 
		\Big( \mathrm{e}^{\frac{h}{2}\bF^{n}}
		\mathrm{e}^{h (\bF_{\bar{\bx}^{n}} - \bF^{n})}
		\mathrm{e}^{\frac{h}{2}\bF^{n}} \Big) 
		\begin{pmatrix}\bv^{n} \\ w^{n}\end{pmatrix} \\
		&\quad = \frac{1}{2} 
		\begin{pmatrix}\bv^{n} \\ w^{n}\end{pmatrix}^{\intercal} 
		\Big( \mathrm{e}^{\frac{h}{2}\bF^{n}}
		\mathrm{e}^{h (\bF_{\bar{\bx}^{n}} - \bF^{n})}
		\mathrm{e}^{\frac{h}{2}\bF^{n}} \Big)^{-1} 
		\Big( \mathrm{e}^{\frac{h}{2}\bF^{n}}
		\mathrm{e}^{h (\bF_{\bar{\bx}^{n}} - \bF^{n})}
		\mathrm{e}^{\frac{h}{2}\bF^{n}} \Big) 
		\begin{pmatrix}\bv^{n} \\ w^{n}\end{pmatrix} \\
		&\quad = \frac{1}{2} 
		\begin{pmatrix}\bv^{n} \\ w^{n}\end{pmatrix}^{\intercal} \begin{pmatrix}\bv^{n} \\ w^{n}\end{pmatrix}
		= H(\bx^{n},\bar{t}^{\, n},\bv^{n},w^{n}).
	\end{align*}
	This indicates that the SS2-xn scheme preserves the discrete energy exactly.
\end{proof}

\begin{proposition}\label{Lorentz}
	(\textbf{Lorentz  invariance}.) 
	The SS2-xn scheme \eqref{SS2} satisfies Lorentz invariance.
\end{proposition}

\begin{proof}
	Let $L$ be a $4\times 4$ Lorentz matrix belonging to the Lorentz group $\mathrm{O}(1,3)$, which satisfies $L^{\intercal} g L = g$, where $g$ denotes the Minkowski metric. This matrix transforms the state vectors from the inertial frame $\mathcal{O}$ to a new inertial frame $\mathcal{O}'$ via
	\begin{equation*}
		\by' = L\by,\quad \bu' = L\bu.
	\end{equation*}
	Recall that $\bF$ corresponds to the electromagnetic tensor, which obeys the standard similarity transformation rule
	\begin{equation*}
		\bF'(\bx') = L \bF(\bx) L^{-1}.
	\end{equation*}
	
	We first prove that the submap $\varphi_{h}^{\mathscr{A}_{\bx^n}}$ commutes with the Lorentz transformation $L$. In the transformed frame $\mathcal{O}'$, the spatial position satisfies ${\bx^{n}}' = L\bx^n$, so
	\begin{equation*}
		\bF'^{n} = \bF'({\bx^{n}}') = L\bF^{n} L^{-1}.
	\end{equation*}
	For any invertible matrix $L$ and analytic function $f$, if $A' = LAL^{-1}$, then $f(A') = Lf(A)L^{-1}$. Since both the matrix exponential and $\varphi_1(s) = (\mathrm{e}^s-1)/s$ are analytic functions, we have
	\begin{equation*}
		\varphi_1(h\bF'^{n}) = L\varphi_1(h\bF^{n})L^{-1},\quad
		\mathrm{e}^{h\bF'^{n}} = L \mathrm{e}^{h\bF^{n}} L^{-1}.
	\end{equation*}
	For any numerical solution $(\by^n,\bu^n)$, the updated variables in the transformed frame read
	\begin{align*}
		\by'^{n+1} & = \by'^n + h\varphi_1(h\bF'^{n})\bu'^n
		= L\by^n + h L\varphi_1(h\bF^{n})L^{-1} L\bu^n 
		= L\big(\by^n + h\varphi_1(h\bF^{n})\bu^n\big)
		= L\by^{n+1}, \\
		\bu'^{n+1} &= \mathrm{e}^{h\bF'^{n}}\bu'^n
		= L \mathrm{e}^{h\bF^{n}} L^{-1} L\bu^n
		= L \mathrm{e}^{h\bF^{n}}\bu^n
		= L\bu^{n+1}.
	\end{align*}
	Consequently,
	\begin{equation*}
		\varphi_{h}^{\mathscr{A}_{\bx^{n}}}(L\by^n,L\bu^n) = L\,\varphi_{h}^{\mathscr{A}_{\bx^{n}}}(\by^n,\bu^n).
	\end{equation*}
	The same reasoning applies to the half-step operator $\varphi_{h/2}^{\mathscr{A}_{\bx^n}}$, which also commutes with $L$.
	
	Next, we verify the commutativity between $\varphi_{h}^{\mathscr{B}_{\bx^{n}}}$ and $L$.
	Note that $\bx$ consists of the first three spatial components of $\by$. From the above results for the half-step submap, the intermediate state satisfies $\bar{\by}'^n = L\bar{\by}^n$, which implies $\bar{\bx}'^n = L\bar{\bx}^n$.
	In frame $\mathcal{O}'$, we have $\bF'(\bar{\bx}'^n) = L\bF(\bar{\bx}^n)L^{-1}$. Combining with $\bF'^{n} = L\bF^{n} L^{-1}$, we obtain
	\begin{equation*}
		\bF'(\bar{\bx}'^n) - \bF'^{n} = L\big(\bF(\bar{\bx}^n)-\bF^{n}\big)L^{-1}.
	\end{equation*}
	This further yields
	\begin{equation*}
		\mathrm{e}^{h(\bF'(\bar{\bx}'^n)-\bF'^{n})} = L \mathrm{e}^{h(\bF(\bar{\bx}^n)-\bF^{n})} L^{-1}.
	\end{equation*}
	The update rule of $\varphi_{h}^{\mathscr{B}_{\bx^{n}}}$ in the transformed frame is
	\begin{align*}
		\by'^{n+1} & = \by'^n = L\by^n, \ \ \
		\bu'^{n+1}  = \mathrm{e}^{h(\bF'(\bar{\bx}'^n)-\bF'^{n})}\bu'^n
		= L \mathrm{e}^{h(\bF(\bar{\bx}^n)-\bF^{n})} L^{-1} L\bu^n
		= L\bu^{n+1}.
	\end{align*}
	Thus
	\begin{equation*}
		\varphi_{h}^{\mathscr{B}_{\bx^{n}}}(L\by^n,L\bu^n) = L\,\varphi_{h}^{\mathscr{B}_{\bx^{n}}}(\by^n,\bu^n).
	\end{equation*}
	
	The SS2-xn scheme is defined as the composition
	$
	\Psi_{h} = \varphi_{h/2}^{\mathscr{A}_{\bx^{n}}} \circ \varphi_{h}^{\mathscr{B}_{\bx^{n}}} \circ \varphi_{h/2}^{\mathscr{A}_{\bx^{n}}}.
	$
	Since each submap commutes with the Lorentz transformation $L$, their composition also commutes with $L$:
	\begin{align*}
		\Psi_h(L\by^n,L\bu^n)
		& = \varphi_{h/2}^{\mathscr{A}_{\bx^{n}}}\Bigl(\varphi_{h}^{\mathscr{B}_{\bx^{n}}}\bigl(\varphi_{h/2}^{\mathscr{A}_{\bx^{n}}}(L\by^n,L\bu^n)\bigr)\Bigr) = \varphi_{h/2}^{\mathscr{A}_{\bx^{n}}}\Bigl(\varphi_{h}^{\mathscr{B}_{\bx^{n}}}\bigl(L\,\varphi_{h/2}^{\mathscr{A}_{\bx^{n}}}(\by^n,\bu^n)\bigr)\Bigr) \\
		&= \varphi_{h/2}^{\mathscr{A}_{\bx^{n}}}\Bigl(L\,\varphi_{h}^{\mathscr{B}_{\bx^{n}}}\bigl(\varphi_{h/2}^{\mathscr{A}_{\bx^{n}}}(\by^n,\bu^n)\bigr)\Bigr) = L\,\varphi_{h/2}^{\mathscr{A}_{\bx^{n}}}\Bigl(\varphi_{h}^{\mathscr{B}_{\bx^{n}}}\bigl(\varphi_{h/2}^{\mathscr{A}_{\bx^{n}}}(\by^n,\bu^n)\bigr)\Bigr) = L\,\Psi_h(\by^n,\bu^n).
	\end{align*}
	In conclusion, the SS2-xn scheme possesses Lorentz invariance.
\end{proof}

In what follows, we establish the convergence of the scheme SS2-xn \eqref{SS2}, deferring its rigorous proof to Section \ref{sec3}.

\begin{theorem}\label{th2.1}
	(\textbf{Uniform second-order accuracy}.) 
	Suppose \(\bE(\cdot), \bB(\cdot) \in C^2(\mathbb{R}^3)\) and $\left\|\nabla \bB(0)\right\| \lesssim \varepsilon$. 
	Denote by $T_0>0$ the period of the flow generated by $\mathrm{e}^{\tau\widehat{\widehat{\bB}}(0)}$. When applying SS2-xn to system \eqref{Rcpd3} across the time interval $[0,T]$, we obtain numerical solutions $\by^n$ and $\bu^n$.
	Then we can find a constant $N_0>0$ independent of $\varepsilon$. For any integer $N\ge N_0$ and step size $h = \varepsilon T_0/N$, the uniform error bounds below hold for all $0 \leq n \leq T/h$:
	\begin{equation*}
		\left\|\by^n - \by(\tau_n)\right\| \lesssim N^{-m_{0}} +  h^2, 
		\quad \left\|\bu^n_{\parallel} - \bu_{\parallel}(\tau_n)\right\| \lesssim  N^{-m_{0}} + h^2, 
		\quad \varepsilon \left\|\bu^n - \bu(\tau_n)\right\| \lesssim h^2.
	\end{equation*}
	Here $m_0>0$ may be chosen arbitrarily large. The symbol $\nabla \bB$ represents the gradient of $\bB$. We define the component of $\bu$ parallel to the magnetic field as
	\begin{equation*}
		\bu_{\parallel}(\tau_n) := \frac{\widetilde{\bB}(\varepsilon \bx(\tau_n))}{\left\|\widetilde{\bB}(\varepsilon \bx(\tau_n))\right\|} \left( \frac{\widetilde{\bB}(\varepsilon \bx(\tau_n))}{\left\|\widetilde{\bB}(\varepsilon \bx(\tau_n))\right\|} \cdot \bu(\tau_n) \right),  \quad
		\bu_{\parallel}^n := \frac{\widetilde{\bB}(\varepsilon \bx^n)}{\left\|\widetilde{\bB}(\varepsilon \bx^n)\right\|} \left( \frac{\widetilde{\bB}(\varepsilon \bx^n)}{\left\|\widetilde{\bB}(\varepsilon \bx^n)\right\|} \cdot \bu^n \right),
	\end{equation*}
	with \( \widetilde{\bB} = (\bB^{\intercal},0)^{\intercal} \).
\end{theorem}

\begin{remark}
	This paper proposes an explicit Strang splitting integrator for RCPD systems that preserves volume, energy, and Lorentz invariance (see Propositions \ref{Volume}-\ref{Lorentz}), and for which we prove uniform second-order convergence in both $\by$ and $\bu_{\parallel}$. 
	Even under the condition \(\|\nabla \bB(0)\| \lesssim \varepsilon\), other structure-preserving algorithms, such as VELPA2 in \cite{Ruili2024}, still produce error bounds that scale with \(\varepsilon\).
	Subsequent numerical experiments will further confirm the practical advantage of our scheme. 
\end{remark}

\begin{remark}
	For systems subject to intense relativistic magnetic fields, the magnetic field gradient at the initial particle position satisfies \( \left\|\nabla \bB(0) \right\| \lesssim \varepsilon \), indicating that the external field is locally nearly uniform and slowly varying in space. This assumption is widely adopted in strong-field asymptotic analysis. Under this condition, we refine the local truncation error bounds for the SS2-xn scheme (see Lemma \ref{lemmalocal}) and derive the uniform second-order error estimates stated in Theorem \ref{th2.1}.
\end{remark}

\begin{remark}
	It should be pointed out that the step-size bound $h = \varepsilon T_0/N$ introduced within the theorem is not a mandatory requirement for practical numerical simulations. In fact, even when adopting a step size significantly larger than that permitted by the condition, e.g., \(h = 1/2^2\), the numerical results presented in Section \ref{sec4} still clearly exhibit second-order uniform convergence. We therefore conclude that such a step-size bound merely serves as an overly strict premise during theoretical deduction, instead of an inherent bottleneck of the proposed splitting scheme.
\end{remark}

\section{Error estimates: the proof of Theorem \ref{th2.1}}\label{sec3}

This section is devoted to proving Theorem~\ref{th2.1}. The overall proof strategy is outlined as follows.
\begin{itemize}
	\item In Section~\ref{subsec3.1}, we first carry out the time rescaling for the original system, then introduce its truncated approximate system under the long-time scale, and provide the error estimate between the two systems in Lemma \ref{approximate}.
	\item We then investigate the local truncation error and the standard global error of the SS2-xn scheme in Lemmas \ref{lemmalocal} and \ref{global}, respectively.
	Finally, we derive the second-order uniform error bound in Section \ref{subsec3.3}.
\end{itemize}

\subsection{The approximate truncated system under time rescaling}\label{subsec3.1}

Before starting the proof, we first recall the original relativistic system \eqref{Rcpd2}.
Under the strong magnetic field with the MOS, the magnetic field $\bB(\eps\bx)$ satisfies condition $\left\| \bB(\eps\bx) - \bB(0) \right\| \lesssim \eps$, and additionally needs to satisfy condition $\left\| \nabla \bB(0) \right\| \lesssim \eps$. For a fixed $T$ (independent of $\varepsilon$), to prove the second-order uniform error bound of the SS2-xn scheme in Theorem \ref{th2.1}, we rescale the time variable of the original system from $\tau$ to $\tau/\varepsilon$. For simplicity, we keep the previous variables and only use the new time step $\mh$ to distinguish, yielding the following long-time system: 
\begin{equation}
	\begin{cases}	
		\displaystyle \dot{\bx}(\tau) = \eps \bv(\tau), \\
		\dot{\bar{t}}(\tau) = \eps w(\tau), \\ \displaystyle 
		\dot{\bv}(\tau) =  \widehat{\bB}(\varepsilon \bx(\tau)) \bv(\tau) - i\eps w(\tau)\bE(\bx(\tau)),  \\[4pt]
		\dot{w}(\tau) = i\eps \bE(\bx(\tau))^{\intercal} \cdot \bv(\tau),
		\displaystyle \quad 0 < \tau \leq \frac{T}{\eps}.
	\end{cases}
	\label{Rcpd2s}
\end{equation}
Similarly, after introducing $\by = (\bx^{\intercal},\bar{t})^{\intercal}$ and $\bu = (\bv^{\intercal},w)^{\intercal}$, the above long-time system \eqref{Rcpd2s} can be equivalently written as 
\begin{equation}
	\begin{cases}	
		\displaystyle \dot{\by}(\tau) = \eps \bu(\tau), \\
		\dot{\bu}(\tau) =  \bK(\bx(\tau)) \bu(\tau),
		\displaystyle \quad 0 < \tau \leq \frac{T}{\eps},
	\end{cases}
	\label{Rcpd3s}
\end{equation}
where 
$
\bK(\bx) = \widehat{\widehat{\bB}}(\varepsilon \bx) + \eps \widehat{\bE}(\bx) 
$
is a $4\times 4$ skew-symmetric matrix.
Since \(\bE(\cdot), \bB(\cdot) \in C^2(\mathbb{R}^3)\), we obtain $\left\|\by\right\|_{L^{\infty}(0,T/\varepsilon)} + \left\|\bu\right\|_{L^{\infty}(0,T/\varepsilon)} \lesssim 1$.
On the time grid $\tau_n = n \mh$ ($n\in\mathbb{N}$), given the initial conditions $\by_0$ and $\bu_0$, we solve the above scaled long-time system \eqref{Rcpd3s} by means of the SS2-xn scheme and obtain
\begin{equation}
	\begin{cases}
		\begin{aligned}
			\displaystyle \by^{n+1} &= \by^{n}+\frac{\varepsilon \mh}{2}\varphi_{1}\biggl( \frac{\mh}{2} \bK^{n} \biggr)
			\Bigl(I_4 + \mathrm{e}^{\mh (\bK_{\bar{\bx}^{n}} - \bK^{n}) }
			\mathrm{e}^{\frac{\mh}{2}\bK^{n}}\Bigr) \bu^{n}, \\
			\displaystyle \bu^{n+1} &= \mathrm{e}^{\frac{\mh}{2}\bK^{n}}
			\mathrm{e}^{\mh (\bK_{\bar{\bx}^{n}} - \bK^{n})}
			\mathrm{e}^{\frac{\mh}{2}\bK^{n}} \bu^{n},
			\quad  0\leq n < \frac{T}{\varepsilon \mh},
		\end{aligned}
	\end{cases}
	\label{SS2s}
\end{equation}
in which the relevant notations are defined as
\begin{align*}
	\bar{\by}^{n}  = \by^{n}+\frac{\varepsilon\mh}{2}\varphi_{1}\biggl(\frac{\mh}{2}\bK^{n}\biggr) \bu^{n}, 
	\quad \bar{\bx}^{n} = \bar{\by}^{n}(1:3), 
	\quad \bK^{n}  = \bK(\bx^{n}), \quad \bK_{\bar{\bx}^{n}} = \bK(\bar{\bx}^{n}).
\end{align*}

We now introduce the approximate truncated system for the scaled long-time system \eqref{Rcpd2s} evaluated at $\tau = \tau_n + s$, as given below:
\begin{equation}
	\begin{cases}
		\dot{\widetilde{\bx}}^{\, n}(s) = \varepsilon \widetilde{\bv}^{\, n}(s), 
		\quad 0 < s \leq \mathfrak{h}, 
		\displaystyle \quad 0 \leq n < \frac{T}{\varepsilon \mathfrak{h}}, \\
		{\dot{\widetilde{\bar{t}}}}^{\,\,\,\, n}(s) = \varepsilon \widetilde{w}^{\, n}(s),  \quad
		\widetilde{\bx}^{\, n}(0) = \bx(\tau_n), \quad {\widetilde{\bar{t}}}^{\,\, n}(0) = \bar{t}(\tau_n), \\
		\displaystyle \dot{\widetilde{\bv}}^{\, n}(s) = \widetilde{\bv}^{\, n}(s) \times \big[ \bB_{mid} + \varepsilon \nabla \bB_{mid} \cdot \big( \widetilde{\bx}^{\, n}(s)-\bx_{mid} \big) \big]
		- i\varepsilon \widetilde{w}^{\, n}(s)\bE(\widetilde{\bx}^{\, n}(s)),  \\
		\dot{\widetilde{w}}^{\, n}(s) = i\eps \bE(\widetilde{\bx}^{\, n}(s))^{\intercal} \cdot \widetilde{\bv}^{\, n}(s), 
		\quad   \widetilde{\bv}^{\, n}(0) = \bv(\tau_n),
		\quad \widetilde{w}^{\, n}(0) = w(\tau_n), 
	\end{cases}
	\label{ts}
\end{equation}
where $\nabla \bB$ stands for the gradient of $\bB$, and 
\begin{equation*}
	\bx_{mid}=\bx(\tau_{n}+\frac{\mathfrak{h}}{2}), 
	\quad \bB_{mid}=\bB(\varepsilon \bx_{mid}), 
	\quad \nabla \bB_{mid}=\nabla \bB(\varepsilon \bx_{mid}).
\end{equation*}
The key ingredient in constructing this truncated system is the additional correction term $\varepsilon \nabla \bB_{mid} \cdot \big( \widetilde{\bx}^{\, n}(s)-\bx_{mid} \big)$. This term lowers the discrepancy between the two systems and improves the accuracy of the local truncation error of the SS2-xn scheme, and is therefore essential for proving the second-order uniform error bounds in the following analysis. 
Based on the boundedness of the electric field $\bE$, the magnetic field $\bB$, and the solutions $\left\|\by\right\|_{L^{\infty}(0,T/\varepsilon)}, \left\|\bu\right\|_{L^{\infty}(0,T/\varepsilon)}$ of the original system, it can also be deduced that the solution of the truncated system is bounded, i.e., $\left\|\widetilde{\by}^{\, n}\right\|_{L^{\infty}(0,\mathfrak{h})} + \left\|\widetilde{\bu}^{\, n}\right\|_{L^{\infty}(0,\mathfrak{h})} \lesssim 1$.

In order to analyze the error between the two systems, we first present the following error quantities: 
\begin{align}
	\zeta^n_{\bx}(s) & := \bx(\tau_n + s) - \widetilde{\bx}^{\, n}(s), \quad 
	\zeta^n_{\bar{t}}(s) := \bar{t}(\tau_n + s) - {\widetilde{\bar{t}}}^{\, \,n}(s),  \label{Z1} \\
	\zeta^n_{\bv}(s) & := \bv(\tau_n + s) - \widetilde{\bv}^{\, n}(s),
	\quad \zeta^n_{w}(s) := w(\tau_{n} + s) - \widetilde{w}^{\, n}(s),
    \quad 0 < s \leq \mh,
	\quad 0 \leq n < \frac{T}{\varepsilon \mathfrak{h}},
	\label{Z2}
\end{align}
and $ \zeta^{n}_{\by}(s) = (\zeta^n_{\bx}(s)^{\intercal},\zeta^n_{\bar{t}}(s))^{\intercal}$, 
$ \zeta^{n}_{\bu}(s) = (\zeta^n_{\bv}(s)^{\intercal},\zeta^n_{w}(s))^{\intercal}$.

\begin{lemma}\label{approximate}
	Let $\zeta^n_{\boldsymbol{y}}$ and $\zeta^n_{\boldsymbol{u}}$ denote the errors associated with the long-time-scale system \eqref{Rcpd2s} and its corresponding truncated approximate system \eqref{ts}, whose precise definitions are given in \eqref{Z1}-\eqref{Z2}. 
	Under the assumptions that the electric and magnetic fields satisfy $\boldsymbol{E}, \boldsymbol{B} \in C^2(\mathbb{R}^3)$ and the initial gradient condition $\|\nabla \boldsymbol{B}(0)\| \lesssim \varepsilon$ holds, the error between the two systems admits the following estimate: 
	\begin{equation}
		\left\|\zeta_{\by}^{n}(\mh)\right\| \lesssim \eps^5 \mh^4, 
		\quad \left\|\zeta_{\bu}^{n}(\mh)\right\| \lesssim \eps^4 \mh^3,
		\quad 0 \leq n < \dfrac{T}{\varepsilon \mathfrak{h}}. \label{M}
	\end{equation}
\end{lemma}

\begin{proof}
	{The proof is presented in \ref{appe}.} 
\end{proof}

\subsection{Uniform error estimates}\label{subsec3.3}

In Lemma \ref{lemmalocal}, we introduce the definition of the local truncation error for the SS2-xn scheme and conduct a thorough error estimation.

\begin{lemma}\label{lemmalocal}
	Let $\bE, \bB \in C^2(\mathbb{R}^3)$ and $\left\|\nabla \bB(0)\right\| \lesssim \varepsilon$.
	If the numerical solution at step $n-1$ is exact, denote by $\xi_{\by}^{n} = \big({\xi_{\bx}^{n}}^{\intercal},\xi_{\bar{t}}^n\big)^{\intercal}$ and $\xi_{\bu}^{n} = \big({\xi_{\bv}^{n}}^{\intercal},\xi_{w}^n\big)^{\intercal}$ the local truncation errors of the SS2-xn scheme for system \eqref{ts} at step $n$.
	There exists a constant $\mathfrak{h}_0$, independent of $\varepsilon$, satisfying the following: for any step size $\mathfrak{h}$ with $0 < \mathfrak{h} < \mathfrak{h}_0$ and integer $0 \leq n < \dfrac{T}{\varepsilon \mathfrak{h}}$, we have
	\begin{equation}
		\left\|\xi_{\by}^{n}\right\| \lesssim \varepsilon^{3} \mathfrak{h}^{3}, \quad 
		\left\|\xi_{\bu}^{n}\right\| \lesssim \varepsilon^{2} \mathfrak{h}^{3}.
		\label{local}
	\end{equation}
\end{lemma}

\begin{proof}
	{The proof is given in \ref{appe}.} 
\end{proof}

Before analyzing the global error, we first define the errors of the SS2-xn scheme, denoted by $e_{\by}^{n+1} = \big({e_{\bx}^{n+1}}^{\intercal},e_{\bar{t}}^{n+1}\big)^{\intercal} := \by(\tau_{n+1}) - \by^{n+1}$ and $e_{\bu}^{n+1} = \big({e_{\bv}^{n+1}}^{\intercal},e_{w}^{n+1}\big)^{\intercal} := \bu(\tau_{n+1}) - \bu^{n+1}$. According to the definition of systematic errors in \eqref{Z1}-\eqref{Z2}, the total error of the scheme can be formulated as 
\begin{align}
	e_{\by}^{n+1} & = \zeta_{\by}^{n}(\mathfrak{h}) + \widetilde{e}_{\by}^{\, n}, \quad e_{\bu}^{n+1}  = \zeta_{\bu}^{n}(\mathfrak{h}) + \widetilde{e}_{\bu}^{\, n}, 
	\quad 0 \leq n < \frac{T}{\varepsilon \mathfrak{h}},  \label{eyu}
\end{align}  
where
$
\widetilde{e}_{\by}^{\, n} := \widetilde{\by}^{\, n}(\mathfrak{h}) - \by^{n+1},
$
$\widetilde{e}_{\bu}^{\, n} := \widetilde{\bu}^{\, n}(\mathfrak{h}) - \bu^{n+1}.
$

\begin{lemma}\label{global}	
	Let $\by^n$ and $\bu^n$ denote the numerical solutions generated by the SS2-xn scheme for the long-time system \eqref{Rcpd3s} over the time interval $[0,T/\varepsilon]$. There exists a positive constant $\mathfrak{h}_0$ that is independent of the small parameter $\varepsilon$, such that for all step sizes satisfying $0<\mathfrak{h}<\mathfrak{h}_0$, the following error bounds are valid:
	\begin{equation}
		\left\|\by^n - \by(\tau_n)\right\| \lesssim \varepsilon \mathfrak{h}^2, \quad \left\|\bu^n - \bu(\tau_n)\right\| \lesssim \varepsilon \mathfrak{h}^2,
		\quad 0 \leq n \leq \frac{T}{\varepsilon \mathfrak{h}},  \label{gyu}
	\end{equation}
	and moreover, 
	$$
	\left\|\by^n\right\| \leq \left\|\by\right\|_{L^\infty(0,T/\varepsilon)} + 1, \quad  \left\|\bu^n\right\| \leq \left\|\bu\right\|_{L^\infty(0,T/\varepsilon)} + 1.
	$$
\end{lemma}

\begin{proof}
	{The proof is also given in \ref{appe}.} 
\end{proof}

According to Lemma \ref{global}, the global error estimate of the SS2-xn scheme depends on $\varepsilon^{-1}$ when returning to the relativistic dynamical system under the original scale. To further improve the error bound, we utilize the exponential periodicity of the skew-symmetric magnetic field and conduct a refined error analysis within a single period. As a result, the error bound is improved to $\mathcal{O}(\varepsilon^2 \mathfrak{h}^2)$.

\begin{lemma}\label{uniform}
	Consider the same hypotheses as Lemma \ref{lemmalocal} and Lemma \ref{global}. Denote by $T_0>0$ the period corresponding to the periodic flow of $\mathrm{e}^{\tau\widehat{\widehat{\bB}}(0)}$. One can find a positive constant $N_0$ that does not depend on $\varepsilon$. If an integer $N$ satisfies $N > N_0$ and we set $\mathfrak{h} = T_0 / N$, then the estimates below are valid for $0 \leq n \leq \dfrac{T}{\varepsilon \mathfrak{h}}$:
	\begin{equation*}
		\left\| \by^n - \by(\tau_n) \right\| \lesssim N^{-m_{0}} + \varepsilon^2 \mathfrak{h}^{2}, 
		\quad \left\| \bu_{\parallel}^n - \bu_{\parallel}(\tau_n) \right\| \lesssim  N^{-m_{0}} + \varepsilon^2 \mathfrak{h}^{2},
	\end{equation*}
	where $m_0>0$ can be taken arbitrarily large.
\end{lemma}

\begin{proof}
	For any fixed $T$, we have 
	\begin{equation*}
		\frac{T}{\varepsilon} = M T_0  + \tau_r, \quad 0 \leq \tau_r < T_0,
	\end{equation*}
	where $T_0$ denotes the minimal positive period of the orthogonal matrix flow $\mathrm{e}^{\tau\widehat{\widehat{\bB}}(0)}$ and 
	\begin{equation*}
		M = \bigg\lfloor  \frac{T}{\varepsilon T_0} \bigg\rfloor  = \mathcal{O}(1/\varepsilon).
	\end{equation*}
	Without loss of generality, we only consider the case $\tau_r=0$.
	Suppose there exists a positive integer $N_0$ such that the assumptions in Lemma \ref{global} are satisfied for all integers $N \geq N_0$ and step sizes $\mathfrak{h} = T_0 / N \leq T_0 / N_0 = \mathfrak{h}_0$. Then the previously derived global error estimates as well as the boundedness of numerical solutions remain valid.
	
	Based on the above preparations, we introduce a refined temporal mesh and use $\tau_{n}^{m} = m T_0 + n\mh$ $(0 \leq n \leq N, 0 \leq m < N)$ to stand for the $n$-th grid point within the $m$-th periodic interval. The associated numerical solutions and approximation errors at this grid are denoted by $\by_{n}^{m},\bu_{n}^{m}$ and $e_{\by}^{n,m},e_{\bu}^{n,m}$, respectively. By following the same derivation strategy as for the error recursions \eqref{eya}-\eqref{eub}, we obtain the updated error equations at the newly defined grid points as given by 
	\begin{subequations}
		\begin{align}
			e_{\by}^{n+1,m} = &\,\, e_{\by}^{n,m} + \frac{\varepsilon \mh}{2}\varphi_{1}\biggl( \frac{\mh}{2} \bK_{\bx(\tau_{n}^{m})} \biggr)
			\Bigl(I_4 + \mathrm{e}^{\mh (\bK_{\bar{\bx}(\tau_{n}^{m})} - \bK_{\bx(\tau_{n}^{m})}) }
			\mathrm{e}^{\frac{\mh}{2}\bK_{\bx(\tau_{n}^{m})}} \Bigr)  e_{\bu}^{n,m}  
			+ \xi_{\by}^{n,m} + \zeta_{\by}^{n,m}(\mathfrak{h}) + \eta_{\by}^{n,m},    \label{ebya}  \\
			e_{\bu}^{n+1,m} = &\,\, \mathrm{e}^{\frac{\mh}{2}\bK_{\bx(\tau_{n}^{m})}}
			\mathrm{e}^{\mh (\bK_{\bar{\bx}(\tau_{n}^{m})} - \bK_{\bx(\tau_{n}^{m})})}
			\mathrm{e}^{\frac{\mh}{2}\bK_{\bx(\tau_{n}^{m})}} e_{\bu}^{n,m}
			+ \xi_{\bu}^{n,m} + \zeta_{\bu}^{n,m}(\mathfrak{h}) + \eta_{\bu}^{n,m},
			\quad 0 \leq n \leq N-1,
			\quad 0 \leq m < M. \label{ebub}
		\end{align}
	\end{subequations}
	
	We now perform a refined error analysis over one full period. Fixing the index $m$, we sum equation \eqref{ebya} for $n$ from $0$ to $N-1$ to obtain
	\begin{equation}
		e_{\by}^{N,m} = e_{\by}^{0,m} + \varepsilon \mathfrak{h} \sum_{n=0}^{N-1} \mathrm{e}^{\frac{\mathfrak{h}}{2} \bK(0)} e_{\bu}^{n,m}
		+ \sum_{n=0}^{N-1} \big( \xi_{\by}^{n,m} + \zeta_{\by}^{n,m}(\mathfrak{h}) + \eta_{\by}^{n,m} \big) +\delta_{\by}^{n,m},
		\quad 0 \leq m < M,  \label{eby2}
	\end{equation}
	where 
	\begin{equation*}
		\delta_{\by}^{n,m} = \varepsilon \mathfrak{h} \sum_{n=0}^{N-1}
		\left[
		\frac{1}{2} \varphi_{1}\biggl( \frac{\mh}{2} \bK_{\bx(\tau_{n}^{m})} \biggr)
		\Bigl(I_4 + \mathrm{e}^{\mh (\bK_{\bar{\bx}(\tau_{n}^{m})} - \bK_{\bx(\tau_{n}^{m})}) }
		\mathrm{e}^{\frac{\mh}{2}\bK_{\bx(\tau_{n}^{m})}} \Bigr) - \mathrm{e}^{\frac{\mathfrak{h}}{2} \bK(0)} \right] e_{\bu}^{n,m}.
	\end{equation*}
	Using Taylor expansions along with estimate \eqref{gyu} and 
	\begin{equation}
		\left\|  \bK_{\bar{\bx}(\tau_{n}^{m})} - \bK(0) \right\| \lesssim 
		\left\| 
		\begin{pmatrix}
			\widehat{\bB}(\eps \bar{\bx}(\tau_n^m)) - \widehat{\bB}(0) & -i\eps \big( \bE(\bar{\bx}(\tau_n^m)) - \bE(0) \big) \\[2pt]
			i\eps \big( \bE(\bar{\bx}(\tau_n^m))^{\intercal} - \bE(0)^{\intercal} \big) & 0
		\end{pmatrix} 
		\right\| \lesssim \eps,  \label{K0}
	\end{equation}
	we further derive
	\begin{equation}
		\left\| \delta_{\by}^{n,m} \right\| \lesssim \eps\mh \sum_{n=0}^{N-1} \eps\mh \cdot \eps\mh^2 \lesssim \eps^3\mh^3, \quad 0 \leq m < M.  \label{deby}
	\end{equation}
	The first summation term on the right-hand side of \eqref{eby2} can be constructed and analyzed by virtue of \eqref{ebub}. 
	We then rearrange \eqref{ebub} as
	\begin{align}
		e_{\bu}^{n+1,m} = \mathrm{e}^{\mathfrak{h} \bK(0)}  e_{\bu}^{n,m}
		+ \xi_{\bu}^{n,m} + \zeta_{\bu}^{n,m}(\mathfrak{h}) + \eta_{\bu}^{n,m} + \delta_{\bu}^{n,m}, 
		\quad 0 \leq n \leq N-1,
		\quad 0 \leq m < M.
		\label{ebu2}
	\end{align}
	Here the additional term satisfies 
	\begin{equation*}
		\delta_{\bu}^{n,m} = \big( \mathrm{e}^{\frac{\mh}{2}\bK_{\bx(\tau_{n}^{m})}}
		\mathrm{e}^{\mh (\bK_{\bar{\bx}(\tau_{n}^{m})} - \bK_{\bx(\tau_{n}^{m})})}
		\mathrm{e}^{\frac{\mh}{2}\bK_{\bx(\tau_{n}^{m})}} - \mathrm{e}^{\mathfrak{h} \bK(0)} \big) e_{\bu}^{n,m}.
	\end{equation*}
	Combining \eqref{gyu} and \eqref{K0}, 
	we arrive at the bound 
	\begin{equation}
		\left\| \delta_{\bu}^{n,m} \right\| \lesssim \eps^2\mh^3, 
		\quad 0 \leq m < M.  \label{debu}
	\end{equation}  
	Recursively expanding the error term on the right-hand side of \eqref{ebu2} down to $n=0$ yields 
	\begin{equation*}
		e_{\bu}^{n,m} = \mathrm{e}^{n \mathfrak{h} \bK(0)}  e_{\bu}^{0,m}
		+ \sum_{j=0}^{n-1} \mathrm{e}^{(n-1-j) \mathfrak{h} \bK(0)} 
		\big(  \xi_{\bu}^{j,m} + \zeta_{\bu}^{j,m}(\mathfrak{h}) + \eta_{\bu}^{j,m} + \delta_{\bu}^{j,m} \big).   
	\end{equation*}
	Multiplying both sides by $\eps\mathfrak{h} \mathrm{e}^{\frac{\mathfrak{h}}{2}\bK(0)}$ from the left and summing over $n$ from $0$ to $N-1$, we arrive at
	\begin{equation}
		\varepsilon \mathfrak{h} \sum_{n=0}^{N-1} \mathrm{e}^{\frac{\mathfrak{h}}{2} \bK(0)} e_{\bu}^{n,m} 
		= \varepsilon \mathfrak{h} \sum_{n=0}^{N-1}
		\mathrm{e}^{(n+\frac{1}{2}) \mathfrak{h} \bK(0)} e_{\bu}^{0,m} 
		+ \varepsilon \mathfrak{h} \sum_{n=0}^{N-1} \sum_{j=0}^{n-1} 
		\mathrm{e}^{(n-\frac{1}{2}-j) \mathfrak{h} \bK(0)}
		\big(  \xi_{\bu}^{j,m} + \zeta_{\bu}^{j,m}(\mathfrak{h}) + \eta_{\bu}^{j,m} + \delta_{\bu}^{j,m} \big).
		\label{ehk}
	\end{equation}
	Substituting \eqref{ehk} into the first summation on the right-hand side of \eqref{eby2} and eliminating the corresponding term, we get
	\begin{equation}
		e_{\by}^{N,m} = e_{\by}^{0,m} + \varepsilon \mathfrak{h} \sum_{n=0}^{N-1}
		\mathrm{e}^{(n+\frac{1}{2}) \mathfrak{h} \bK(0)} e_{\by}^{0,m} + \Gamma^{m},
		\quad 0 \leq m < M,
		\label{eby3}
	\end{equation}
	where the small quantity satisfies 
	\begin{equation*}
		\Gamma^{m} = \sum_{n=0}^{N-1}
		\big( \xi_{\by}^{n,m} + \zeta_{\by}^{n,m}(\mathfrak{h}) + \eta_{\by}^{n,m} \big) +\delta_{\by}^{n,m}   + \varepsilon \mathfrak{h} \sum_{n=0}^{N-1} \sum_{j=0}^{n-1} 
		\mathrm{e}^{(n-\frac{1}{2}-j) \mathfrak{h} \bK(0)}
		\big(  \xi_{\bu}^{j,m} + \zeta_{\bu}^{j,m}(\mathfrak{h}) + \eta_{\bu}^{j,m} + \delta_{\bu}^{j,m} \big).  
	\end{equation*}
	Using the estimates \eqref{M}, \eqref{local}, \eqref{etaby}, \eqref{etabu}, \eqref{gyu}, \eqref{deby}, \eqref{debu}, together with $N \mathfrak{h} = T_{0} \lesssim 1$ and $n^2\mathfrak{h}^2 \leq N^2 \mathfrak{h}^2 = T_{0}^{2} \lesssim 1$, we deduce
	\begin{equation*}
		\left\| \Gamma^{m} \right\| \lesssim \varepsilon^{3} \mathfrak{h}^{2}, 
		\quad 0 \leq m < M. 
	\end{equation*}
	We next apply the trapezoidal quadrature rule to the summation in \eqref{eby3} to obtain 
	\begin{equation*}
		\varepsilon \mathfrak{h} \sum_{n=0}^{N-1}
		\mathrm{e}^{(n+\frac{1}{2}) \mathfrak{h} \bK(0)} e_{\by}^{0,m} = 
		\eps \int_{0}^{T_0} \mathrm{e}^{s \bK(0)} ds \cdot e_{\by}^{0,m} + \mathcal{X}^{m},
	\end{equation*}
	with 
	$
		\left\| \mathcal{X}^{m} \right\| \lesssim \eps N^{-m_{0}},  
	$
	where $m_0>0$ can be chosen arbitrarily large.
	We now approximate the matrix exponential by expanding $\bK(0)$:
	\begin{equation*}
		\mathrm{e}^{s \bK(0)} = \mathrm{e}^{s \big( \widehat{\widehat{\bB}}(0) - \eps \widehat{\bE}(0) \big) } = \mathrm{e}^{\eps s \widehat{\bE}(0)} \mathrm{e}^{s \widehat{\widehat{\bB}}(0)} + \mathcal{O}(\eps s^2)
		= \mathrm{e}^{s \widehat{\widehat{\bB}}(0)} + \mathcal{O}(\eps s) \mathrm{e}^{s \widehat{\widehat{\bB}}(0)}.
	\end{equation*}
	Substituting this expression into the integral yields
	\begin{equation*}
		\eps \int_{0}^{T_0} \mathrm{e}^{s \bK(0)} ds \cdot e_{\by}^{0,m} 
		=  \eps \int_{0}^{T_0} \big( \mathrm{e}^{s \widehat{\widehat{\bB}}(0)} + \mathcal{O}(\eps s) \mathrm{e}^{s \widehat{\widehat{\bB}}(0)} \big) ds \cdot e_{\by}^{0,m}  \lesssim  \eps \int_{0}^{T_0} \mathrm{e}^{s \widehat{\widehat{\bB}}(0)} ds \cdot e_{\by}^{0,m}
		= \eps T_0 
		\big( \widetilde{\widetilde{\bB}}(0) \cdot e_{\bu}^{0,m} \big) \widetilde{\widetilde{\bB}}(0).
	\end{equation*}
	Here $\widetilde{\bB} = (\bB^{\intercal},0)^{\intercal}$, and $\widetilde{\widetilde{\bB}} = \widetilde{\bB} / | \widetilde{\bB} |$ denotes the unit vector in the direction of $\widetilde{\bB}$. 
	Based on the above analysis, taking the norm on both sides of \eqref{eby3} and using the relation \(e_{\by}^{N,m}=e_{\by}^{0,m+1}\), we get 
	\begin{equation}
		\left\| e_{\by}^{0,m+1} \right\| - \left\| e_{\by}^{0,m} \right\| 
		\lesssim \varepsilon \left\| e_{\bu,\parallel}^{0,m} \right\|
		+ \eps N^{-m_{0}}
		+ \varepsilon^{3} \mathfrak{h}^{2},
		\quad 0 \leq m < M,  \label{eby4}
	\end{equation}
	where $e_{\bu,\parallel}^{n,m}$ stands for the component of $e_{\bu}^{n,m}$ parallel to $\widetilde{\widetilde{\bB}}\big(\varepsilon \bx(\tau_{n}^{m})\big)$, defined as
	\begin{equation*}
		e_{\bu,\parallel}^{n,m} := \big( \widetilde{\widetilde{\bB}}^{n,m} \cdot e_{\bu}^{n,m} \big) \widetilde{\widetilde{\bB}}^{n,m}, \quad
		\widetilde{\widetilde{\bB}}^{n,m} := \frac{\widetilde{\bB} ( \varepsilon \bx(\tau_{n}^{m}))}{\left\| \widetilde{\bB} ( \varepsilon \bx(\tau_{n}^{m})) \right\|}.
	\end{equation*}

	Taking the inner product of \eqref{ebub} with $\widetilde{\widetilde{\bB}}^{n+1,m}$ and then computing the norm, we obtain 
	\begin{equation}
		\left\| e_{\bu,\parallel}^{n+1,m} \right\| \lesssim  \left\| \widetilde{\widetilde{\bB}}^{n+1,m} \cdot 
		\big( \mathrm{e}^{\frac{\mh}{2}\bK_{\bx(\tau_{n}^{m})}}
		\mathrm{e}^{\mh (\bK_{\bar{\bx}(\tau_{n}^{m})} - \bK_{\bx(\tau_{n}^{m})})}
		\mathrm{e}^{\frac{\mh}{2}\bK_{\bx(\tau_{n}^{m})}} e_{\bu}^{n,m} \big) \right\| 
		+ \left\| \widetilde{\widetilde{\bB}}^{n+1,m} \cdot \xi_{\bu}^{n,m} \right\|
		+ \left\| \widetilde{\widetilde{\bB}}^{n+1,m} \cdot \zeta_{\bu}^{n,m}(\mathfrak{h}) \right\| + \left\| \widetilde{\widetilde{\bB}}^{n+1,m} \cdot \eta_{\bu}^{n,m} \right\|.
		\label{ebu3}
	\end{equation}
	Using $\widetilde{\widetilde{\bB}}^{n+1,m} = \widetilde{\widetilde{\bB}}^{n,m} + O(\varepsilon^{3} \mathfrak{h})$ and the global error estimate \eqref{gyu}, we have
	\begin{equation*}
		\left\| \widetilde{\widetilde{\bB}}^{n+1,m} \cdot 
		\big( \mathrm{e}^{\frac{\mh}{2}\bK_{\bx(\tau_{n}^{m})}}
		\mathrm{e}^{\mh (\bK_{\bar{\bx}(\tau_{n}^{m})} - \bK_{\bx(\tau_{n}^{m})})}
		\mathrm{e}^{\frac{\mh}{2}\bK_{\bx(\tau_{n}^{m})}} e_{\bu}^{n,m} \big) \right\|
		\lesssim \left\| e_{\bu,\parallel}^{n,m} \right\| + \varepsilon^{4} \mathfrak{h}^{3}.
	\end{equation*} 
	Recalling the analysis of $\xi_{\bu}^n$, its bound is controlled by 
	\begin{equation*}
		- i\varepsilon \mathfrak{h} \int_{0}^{1} \mathrm{e}^{\rho \mathfrak{h} \widehat{\bB}_{mid}} w(\tau_{n}) \int_{0}^{1}\nabla \bE(s_{\sigma}) \, d\sigma \Big[ (\frac{1}{2} - \rho)\mathfrak{h} \varepsilon \bv(\tau_{\rho}^{n}) - \zeta_{\bx}^{n}((1-\rho) \mathfrak{h}) \Big] \, d\rho.
	\end{equation*}
	After taking the inner product with $\widetilde{\widetilde{\bB}}^{n+1,m}$ and applying the Rodrigues rotation formula, we obtain the improved bound 
	\begin{equation*}
		\left\| \widetilde{\widetilde{\bB}}^{n+1,m} \cdot \xi_{\bu}^{n,m} \right\| \lesssim \varepsilon^{3} \mathfrak{h}^{3}.
	\end{equation*}
	From \eqref{M} and \eqref{etabu}, it follows that
	\begin{equation*}
		\left\| \widetilde{\widetilde{\bB}}^{n+1,m} \cdot \zeta_{\bu}^{n,m}(\mathfrak{h}) \right\|  \lesssim \left\| \zeta_{\bu}^{n,m}(\mathfrak{h}) \right\| 
		\lesssim \varepsilon^{4} \mathfrak{h}^{3},
		\quad
		\left\| \widetilde{\widetilde{\bB}}^{n+1,m} \cdot \eta_{\bu}^{n,m} \right\| 
		\lesssim \left\| \eta_{\bu}^{n,m} \right\|
		\lesssim \varepsilon \mathfrak{h} \left\| e_{\by}^{n,m} \right\| 
		+ \varepsilon^{2} \mathfrak{h}^{2} \left\| e_{\bu}^{n,m} \right\|.
	\end{equation*}
	Substituting the above estimates into \eqref{ebu3} yields
	\begin{equation}
		\left\| e_{\bu,\parallel}^{n+1,m} \right\| - \left\| e_{\bu,\parallel}^{n,m} \right\| \lesssim
		\varepsilon \mathfrak{h} \left\| e_{\by}^{n,m} \right\| + \varepsilon^{3} \mathfrak{h}^{3},
		\quad 0 \leq n \leq N-1, \quad 0 \leq m <M.
		\label{ebu4}
	\end{equation} 
	From the recurrence relation in \eqref{ebya}, we readily obtain the estimate for 
	\begin{equation}
		\left\| e_{\by}^{n,m} \right\| \lesssim \varepsilon^2 \mathfrak{h}^2 + \left\| e_{\by}^{0,m} \right\|, 
		\quad 1 \leq n \leq N, \quad 0 \leq m <M. \label{Eby}
	\end{equation}
	Substitute \eqref{Eby} into \eqref{ebu4} and sum over $n$ from $0$ to $N-1$. Using the relation $e_{\bu}^{N,m}=e_{\bu}^{0,m+1}$, we finally arrive at
	\begin{equation}
		\left\| e_{\bu,\parallel}^{0,m+1} \right\| - \left\| e_{\bu,\parallel}^{0,m} \right\|  \lesssim
		\varepsilon \mathfrak{h} \sum_{n=0}^{N-1} \left\| e_{\by}^{0,m} \right\| + \sum_{n=0}^{N-1} \varepsilon^{3} \mathfrak{h}^{3} 
		\lesssim
		\varepsilon  \left\| e_{\by}^{0,m} \right\| + \varepsilon^{3} \mathfrak{h}^{2},
		\quad 0 \leq m <M. \label{ebu5}
	\end{equation}
	
	Summing the inequalities \eqref{eby4} and \eqref{ebu5} leads to
	\begin{equation*}
		\left\| e_{\by}^{0,m+1} \right\| + \left\| e_{\bu,\parallel}^{0,m+1} \right\| - \left\| e_{\by}^{0,m} \right\| - \left\| e_{\bu,\parallel}^{0,m} \right\| 
		\lesssim 
		\varepsilon \big( \left\| e_{\by}^{0,m} \right\| 
		+  \left\| e_{\bu,\parallel}^{0,m} \right\| \big) 
		+ \eps N^{-m_{0}}
		+ \varepsilon^{3} \mathfrak{h}^{2}, 
		\quad 0 \leq m < M.
	\end{equation*}
	Combined with the initial conditions $e_{\by}^{0,0} = e_{\bu,\parallel}^{0,0} = 0$, we apply Gronwall's inequality and conclude
	\begin{equation*}
		\left\| e_{\by}^{0,m} \right\| + \left\| e_{\bu,\parallel}^{0,m} \right\| \lesssim N^{-m_{0}} +
		\varepsilon^2 \mathfrak{h}^{2}, 
		\quad  0 \leq m \leq M. 
	\end{equation*}
	Substituting this bound into \eqref{Eby}, we obtain for interior nodes $1\le n\le N-1$ that
	\begin{equation*}
		\left\| e_{\by}^{n,m} \right\| \lesssim N^{-m_{0}} + \varepsilon^{2} \mathfrak{h}^{2},
		\quad 1 \leq n \leq N-1, \quad 0 \leq m < M.
	\end{equation*}
	Further substituting into \eqref{ebu4} yields
	\begin{equation*}
		\left\| e_{\bu,\parallel}^{n,m} \right\| \lesssim N^{-m_{0}} + \varepsilon^{2} \mathfrak{h}^{2},
		\quad 1 \leq n \leq N-1, \quad 0 \leq m < M.
	\end{equation*}
	Thus the proof of Lemma \ref{uniform} is finished.
\end{proof} 

Returning to the original system \eqref{Rcpd2}, we introduce the substitution $h = \varepsilon \mathfrak{h}$. This yields the second-order uniform error bounds stated in Theorem \ref{th2.1}.
The proof of Theorem \ref{th2.1} is complete.

\section{Numerical experiment}\label{sec4}

Under the assumptions of Theorem \ref{th2.1}, three different strong magnetic field cases under the MOS are chosen for numerical experiments. The condition $\left\|\nabla \bB(0)\right\| \lesssim \varepsilon$ holds for Example \ref{p1}, and $\left\|\nabla \bB(0)\right\| = 0$ for Examples \ref{p2} and \ref{p3}. 
In order to study the motion characteristics of charged particles under these strong magnetic fields and to verify the error convergence results and the energy-preserving property of the SS2-xn scheme \eqref{SS2}, we define the physical velocity of the particles as \(\bv_{\text{phys}} = (v_{\text{phys}}^{1},v_{\text{phys}}^{2},v_{\text{phys}}^{3})^{\intercal} \) and introduce the following error indicators:
\begin{align}
	& erry := \frac{\left\|\by^N - \by(\tau_N)\right\|}{\left\|\by(\tau_N)\right\|},
	\quad erru := \frac{\left\|\bu^N - \bu(\tau_N)\right\|}{\left\|\bu(\tau_N)\right\|}, 
	\quad erru_{\parallel} := \frac{\left\|\bu_{\parallel}^N - \bu_{\parallel}(\tau_N)\right\|}{\left\|\bu_{\parallel}(\tau_N)\right\|},  \nonumber  \\
	&  error = erry + erru_{\parallel},  
	\label{error} \\
	& \varepsilon \text{-} erru = \varepsilon \cdot erru,  \label{erru} \\
	& err_H := \frac{\left\|\bH(\bx^n,\bar{t}^{\, n},\bv^n,w^n) - \bH(\bx^0,\bar{t}^{\,0},\bv^0,w^0)\right\|}{\left\|\bH(\bx^0,\bar{t}^{\,0},\bv^0,w^0)\right\|}, \quad 0 \leq n \leq N. \label{energy}
\end{align}
The reference solutions of the RCPD system are obtained via ‘ode45’, and we solve the system using the SS2-xn scheme until $T=\tau_N=1$ to confirm its second-order uniform error bounds.

As a comparative experiment, we employ the VELPA2 scheme from \cite{Ruili2024}. 
The relativistic system \eqref{Rcpd3} is split into two subsystems
\begin{equation*}
	\mathbf{S}_1 : \begin{pmatrix} \dot{\by}(\tau)\\ \dot{\bu}(\tau) \end{pmatrix} =
	\begin{pmatrix}
		0 \\
		\bF(\bx)\bu(\tau)
	\end{pmatrix}, \quad	
	\mathbf{S}_2: \begin{pmatrix} \dot{\by}(\tau)\\ \dot{\bu}(\tau) \end{pmatrix}=
	\begin{pmatrix}
		\by(\tau) \\
		0
	\end{pmatrix}.
\end{equation*}
By solving these two subsystems exactly, we obtain the solution flows \(\Phi_{\tau}^{\mathbf{S}_1}\) and \(\Phi_{\tau}^{\mathbf{S}_2}\), with their explicit expressions given below: 
\begin{equation*}
	\Phi_{h}^{\mathbf{S}_1} \begin{pmatrix} \by(\tau_{n}) \\ \bu(\tau_{n}) \end{pmatrix} := 
	\begin{pmatrix}
		\by(\tau_{n}) \\
		\mathrm{e}^{h \bF(\bx^n) } \bu(\tau_{n})
	\end{pmatrix}, \quad
	\Phi_{h}^{\mathbf{S}_2} \begin{pmatrix} \by(\tau_{n}) \\ \bu(\tau_{n}) \end{pmatrix} := 
	\begin{pmatrix}
		\by(\tau_{n}) + h \bu(\tau_{n}) \\
		\bu(\tau_{n})
	\end{pmatrix}.
\end{equation*}
From this, the second-order splitting scheme $ \Phi_{h/2}^{\mathbf{S}_1} \circ \Phi_{h}^{\mathbf{S}_2} \circ \Phi_{h/2}^{\mathbf{S}_1} $ is denoted as
\begin{equation} 
	\begin{cases}
		\begin{aligned}
			\displaystyle \by^{n+1}= & \,\, \displaystyle \by^{n} + h \mathrm{e}^{\frac{h}{2}\bF(\bx^n)} \bu^{n},   \\
			\displaystyle \bu^{n+1}= &\,\, \mathrm{e}^{\frac{h}{2}\bF(\bx^{n+1})}
			\mathrm{e}^{\frac{h}{2}\bF(\bx^{n})}
			\bu^{n},  
			\quad 0\leq n < \frac{T}{h}.
		\end{aligned}
	\end{cases} \label{V2}
\end{equation}

\begin{example}\label{p1}
	To begin with, we study the relativistic motion of charged particles in a strong magnetic field obeying the MOS, namely 
	$
	\frac{1}{\varepsilon}\bB(\varepsilon \bx)=\frac{1}{\varepsilon}(1 + \varepsilon \sin(\varepsilon x_{1}),1 + \cos(\varepsilon x_{2}),1 - (\varepsilon \sin(\varepsilon x_{3}))/2)^{\intercal}.
	$ 
	The electric field is defined via $\bE(\bx)=-\nabla U(\bx)$ with $U(\bx)=\dfrac{1}{\sqrt{x_{1}^{2}+x_{2}^{2}+x_{3}^{2}}}$.
	The initial data are set to \( \bx_0 = (1/6,1/8,1/4)^{\intercal} \), \(\bar{t}_0 = 0\), \( \bv_0 = (1/5,1/3,1/2)^{\intercal} \) and \(w_0 = i\sqrt{1+\left\|\bv_0\right\|^2}\).
\end{example}

Figs.~\ref{spatial_1}-\ref{velocity_1} display the two-dimensional (2D) and three-dimensional (3D) phase portraits of position and physical velocity for long-time particle motion simulated by the SS2-xn scheme.
The relative errors of the SS2-xn scheme \eqref{SS2} and the VELPA2 scheme \eqref{V2} are presented in Figs.~\ref{S_1}-\ref{V_1}, respectively, while the energy errors are shown in Fig.~\ref{E_1}. A comparison reveals the following:
\begin{enumerate}
	\item Since the Coulomb electric field is a radial field centered at the origin, its electric field vector is centrally symmetric about the origin. When the particle passes through the vicinity of the origin, the direction of the electric field reverses, and the drift direction reverses accordingly, causing the trajectory to turn around at the origin and oscillate back and forth. As shown in Fig.~\ref{spatial_1}, as $\eps$ decreases (from $1/2^3$ to $1/2^5$), the magnetic field strength increases, the guiding-center motion becomes more stable, and the trajectory eventually evolves into an approximately S-shaped symmetric structure.
	Within the relativistic framework (with the speed of light $c=1$), the physical velocity $\bv_{\text{phys}}$ and momentum $\bv$ satisfy the relation
	$\bv_{\text{phys}} = \frac{\bv}{\gamma} 
	= \frac{\bv}{\sqrt{1+\left\|\bv\right\|^2}}.$ 
	The physical velocity approaches the speed of light $1$ as the momentum tends to infinity. Fig.~\ref{velocity_1} demonstrates that the physical velocity of the particle never exceeds $1$ at different simulation times $T$. 
	Collectively, these numerical results verify the validity and reliability of the SS2-xn scheme in long-time simulations.

	\item For $\varepsilon\in(0,1)$, the SS2-xn scheme exhibits second-order convergence in $h$ for both $\by$ and the parallel component $\bu_{\parallel}$, and the error bound is independent of the small parameter $\varepsilon$, i.e., a second-order uniform error bound. Moreover, the error bound of the SS2-xn scheme for $\bu$ is indeed $\mathcal{O}(\varepsilon^{-1}h^2)$ (see the middle panel of Fig.~\ref{S_1}). This verifies the theoretical error results of Theorem \ref{th2.1} and demonstrates that the obtained error bounds are optimal.
	
	\item From the first two panels of Fig.~\ref{V_1}, the VELPA2 scheme also achieves second-order convergence in $h$ for $\by$ and the parallel component $\bu_{\parallel}$, but it fails to yield a uniform error bound. The right panel of Fig.~\ref{V_1} more clearly illustrates the dependence of the error bound of VELPA2 on the small parameter $\varepsilon$: as $\varepsilon$ decreases $(\eps \to 0)$, the error increases with a slope of approximately $\varepsilon^{-1.5}$.
	
	\item To verify the energy-preserving property of the two schemes, we compute the energy error for different values of $\varepsilon$ with step sizes $h = 1/2^8$ and $1/2^{10}$ up to $T = 1000$. The results in Fig.~\ref{E_1} indicate that both schemes maintain energy stability over long-time integration.
\end{enumerate}

%%空间相图
\begin{figure}[t!] 
	\centering
	\subfigure{\includegraphics[height=4cm]{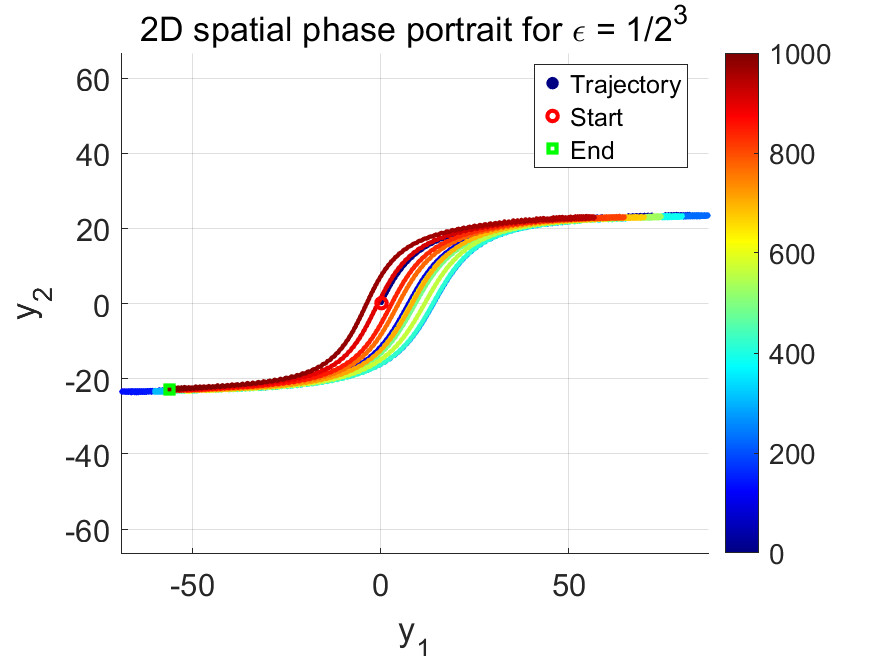}}
	\subfigure{\includegraphics[height=4cm]{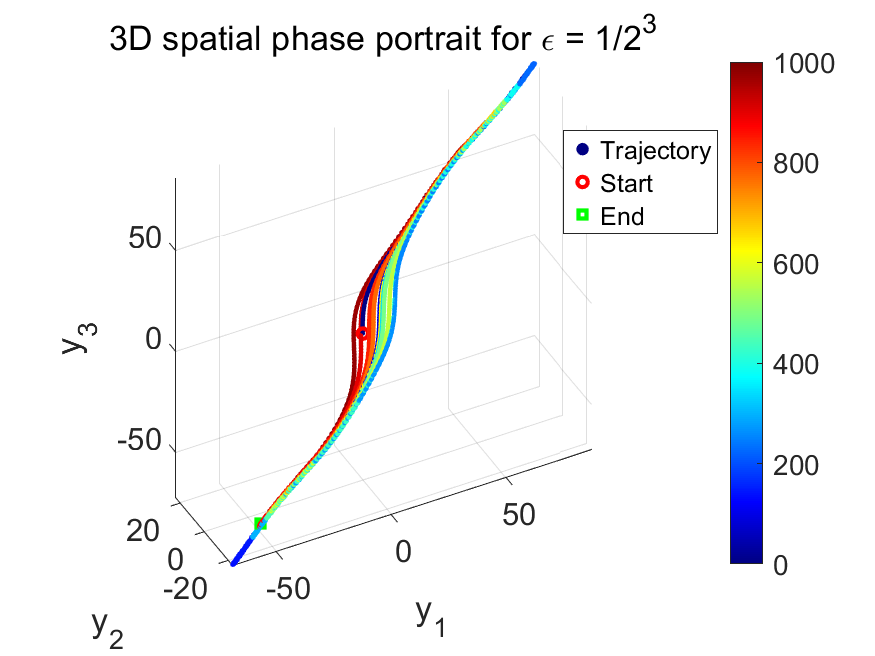}}\\
    \subfigure{\includegraphics[height=4cm]{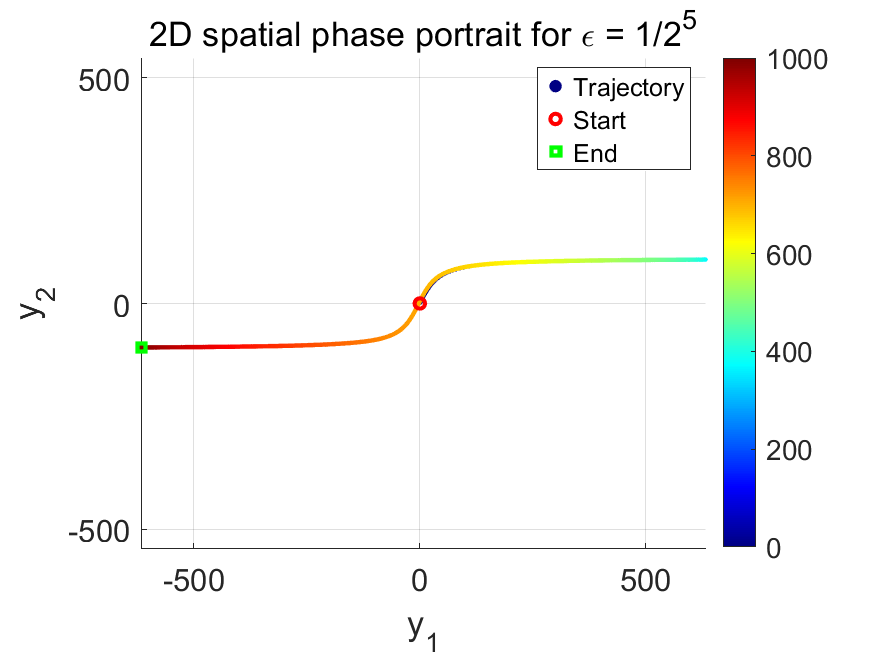}}
    \subfigure{\includegraphics[height=4cm]{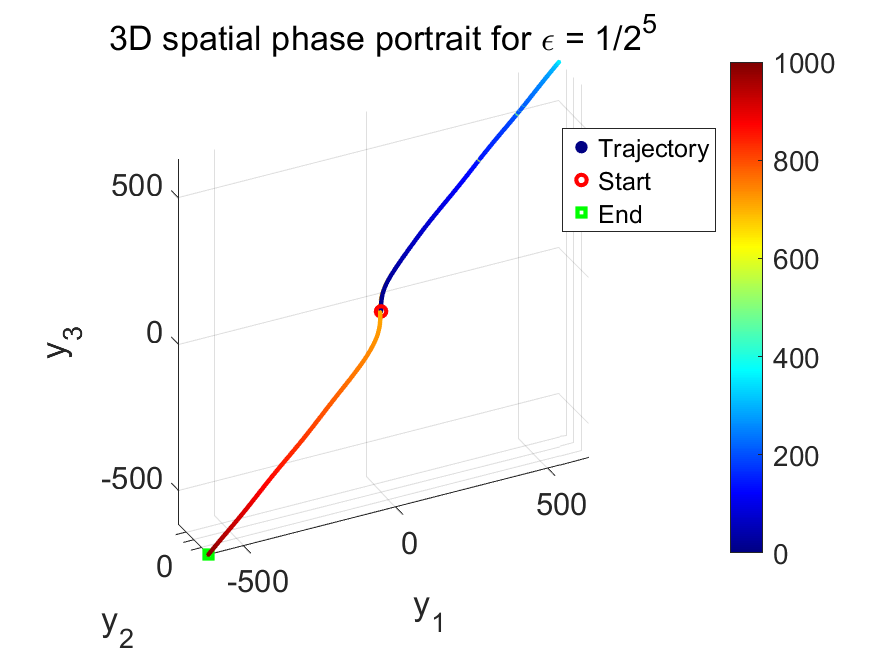}}
	\caption{Example \ref{p1}. 2D (left panels) and 3D (right panels) spatial evolution plots of SS2-xn with \(h = 1/2^8\), \(T = 1000\), and \(\varepsilon\) varying.}
	\label{spatial_1}
\end{figure}

%%速度相图
\begin{figure}[t!] 
	\centering
	\subfigure{\includegraphics[height=4cm]{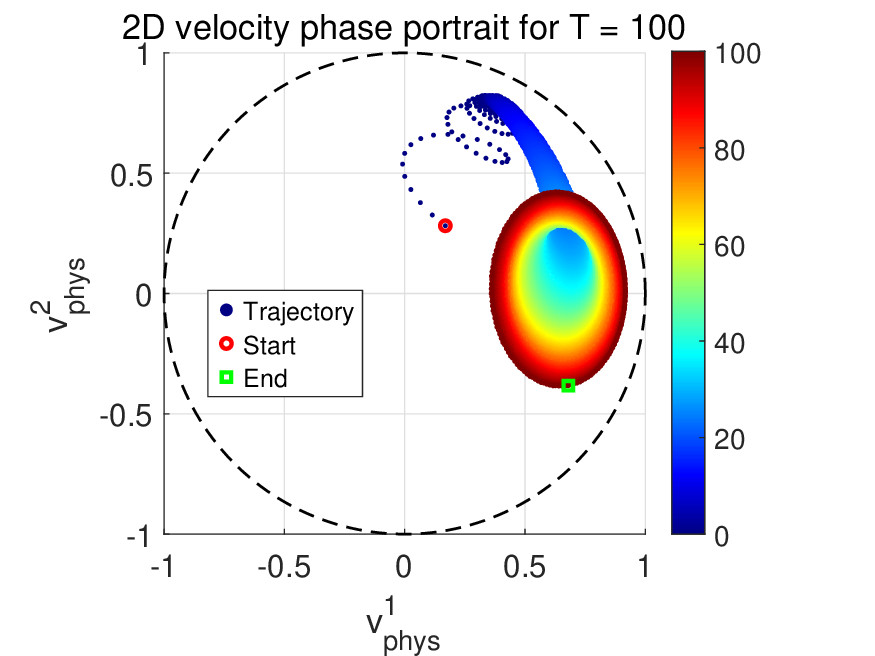}}
	\subfigure{\includegraphics[height=4cm]{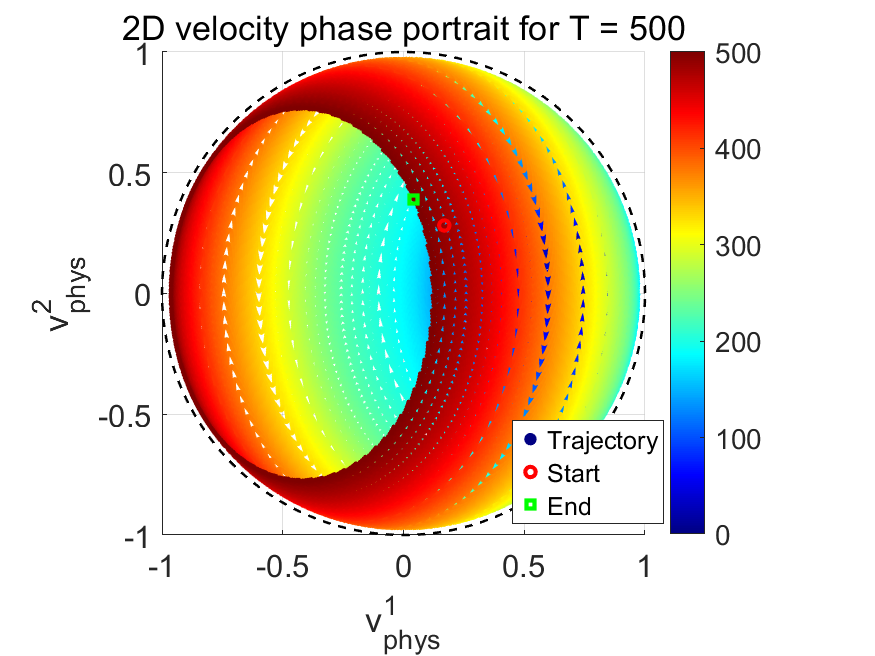}}
	\subfigure{\includegraphics[height=4cm]{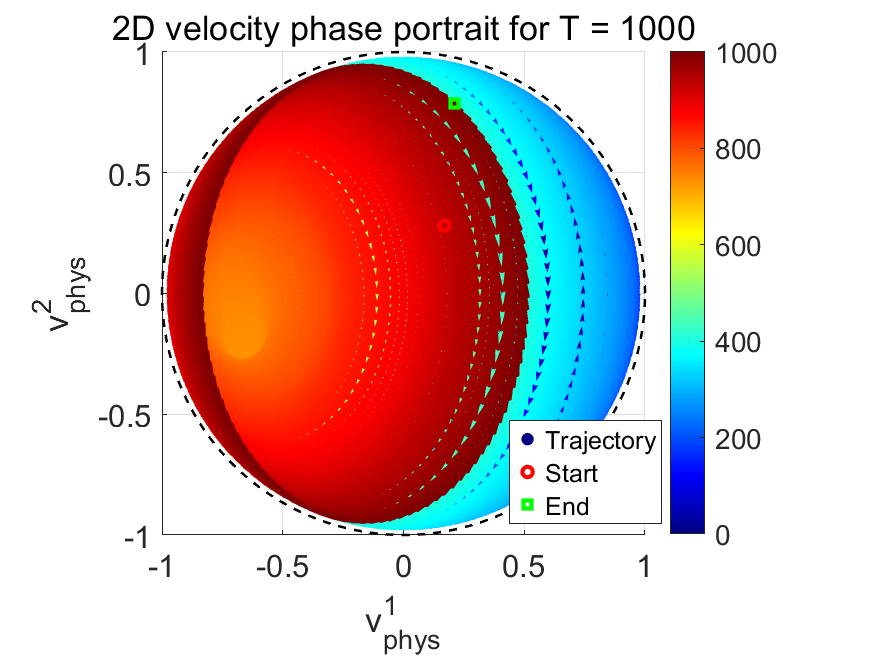}}
    \subfigure{\includegraphics[height=4cm]{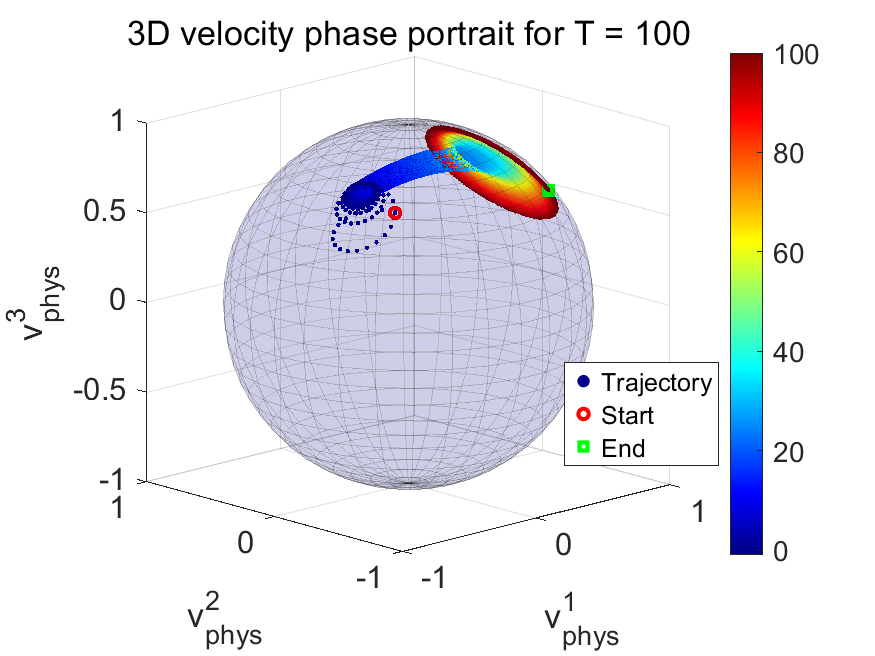}}
	\subfigure{\includegraphics[height=4cm]{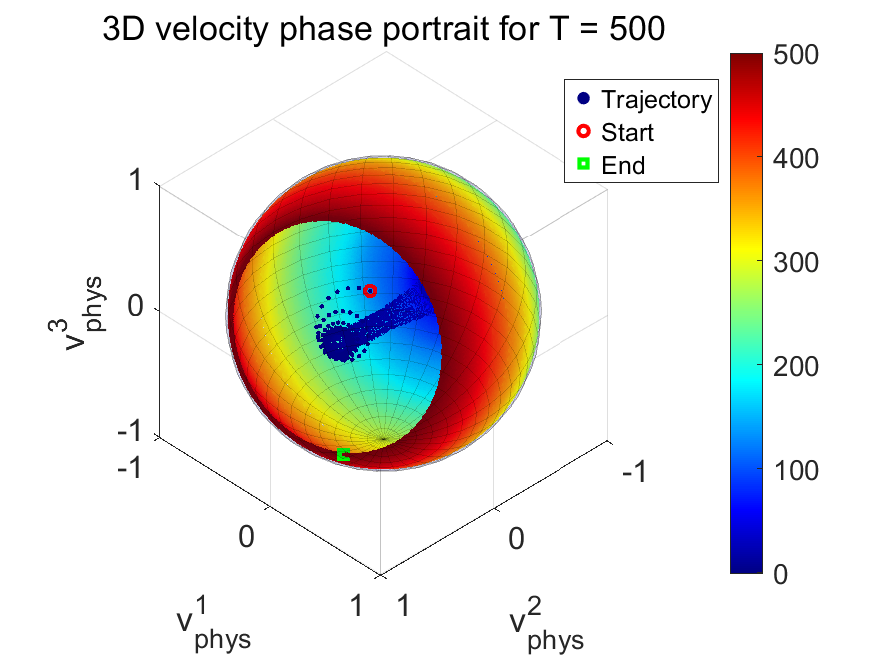}}
	\subfigure{\includegraphics[height=4cm]{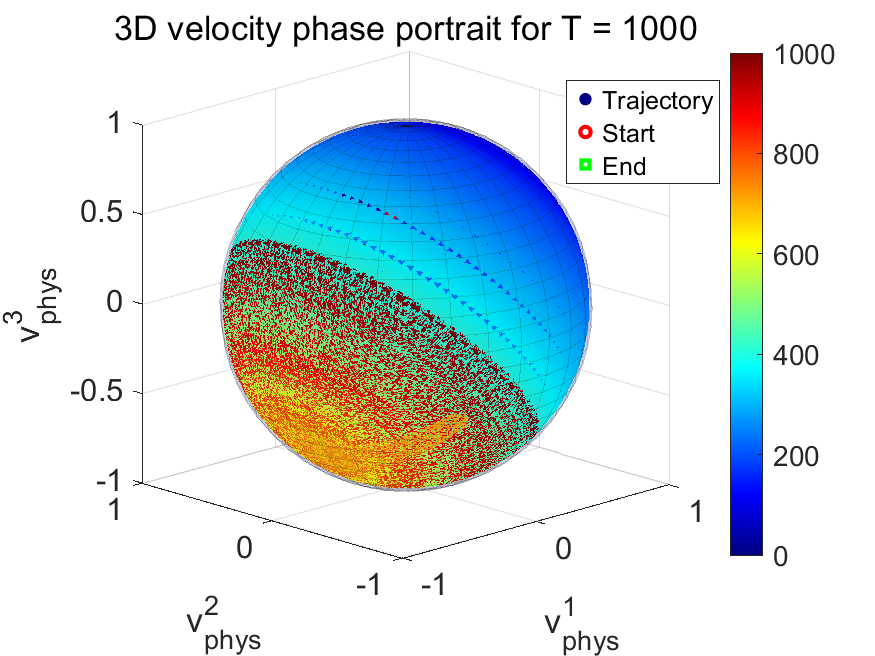}}
	\caption{Example \ref{p1}. Physical velocity evolution of SS2-xn in 2D (top panels) and 3D (bottom panels) for \(h = 1/2^8\), \(\varepsilon = 1/2^5\), and varying \(T\).}
	\label{velocity_1}
\end{figure}

%% SS2-xn 例1
\begin{figure}[t!] 
	\centering
	\subfigure{\includegraphics[height=4cm]{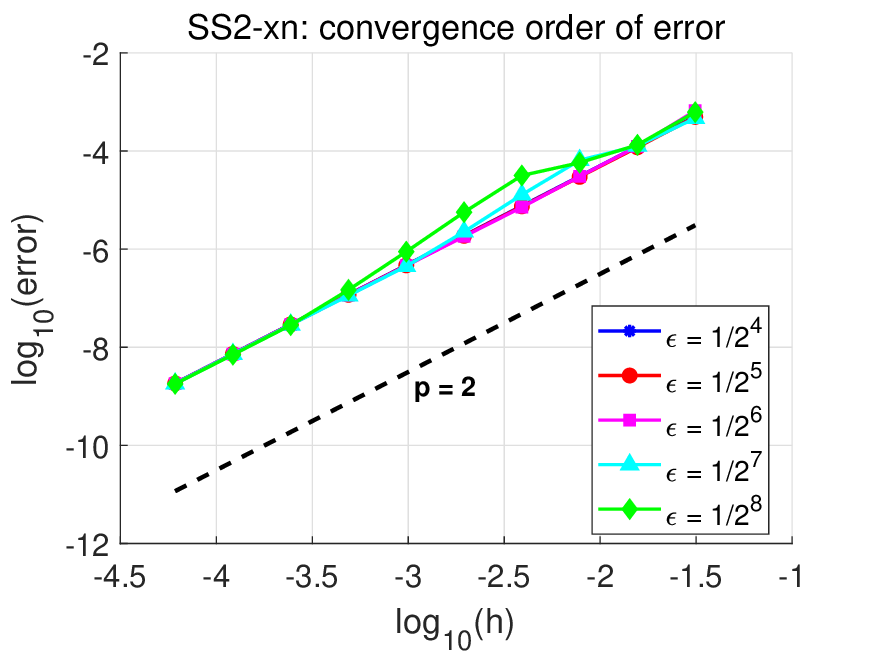}}
	\subfigure{\includegraphics[height=4cm]{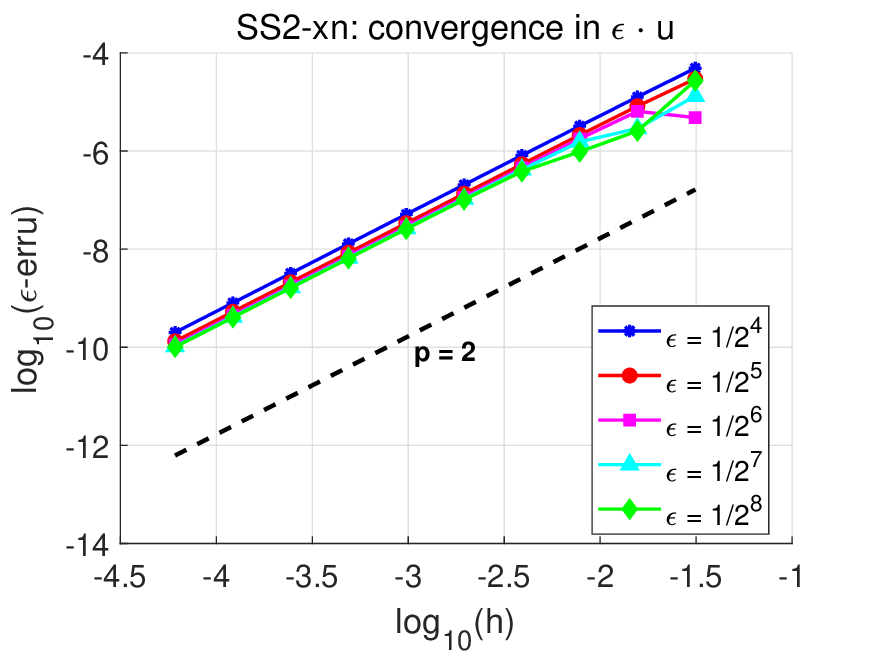}}
	\subfigure{\includegraphics[height=4cm]{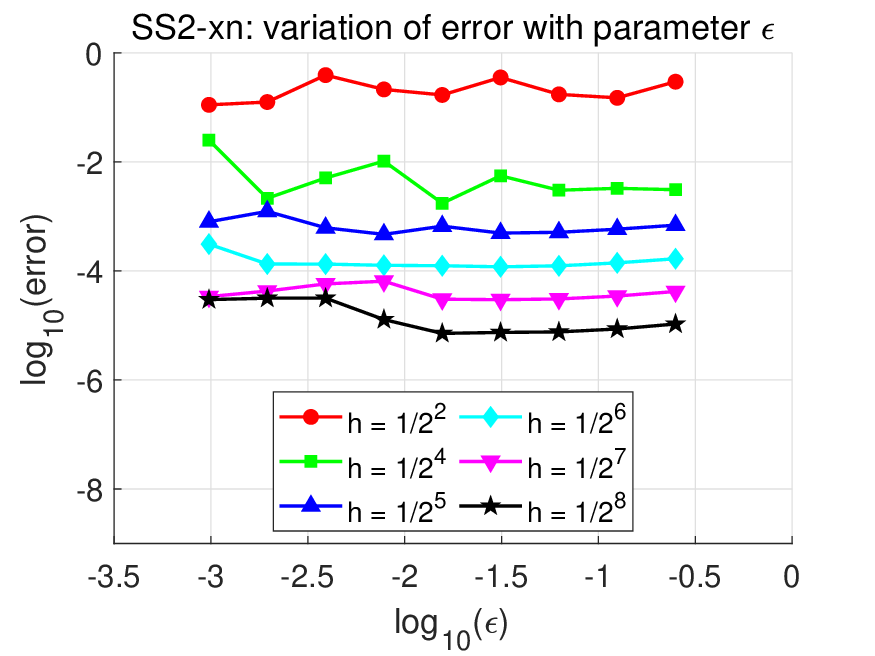}}
	\caption{Example \ref{p1}. SS2-xn errors: \eqref{error} and \eqref{erru} (left and middle panels) with \( h = 1/2^{k} \) for \( k=5,\dots,14 \) and varying \( \varepsilon \); \eqref{error} (right panel) with \( \varepsilon = 1/2^{k} \) for \( k=2,\dots,10 \) and varying \( h \).}
	\label{S_1}
\end{figure}

%% VELPA 例1
\begin{figure}[t!] 
	\centering
	\subfigure{\includegraphics[height=4cm]{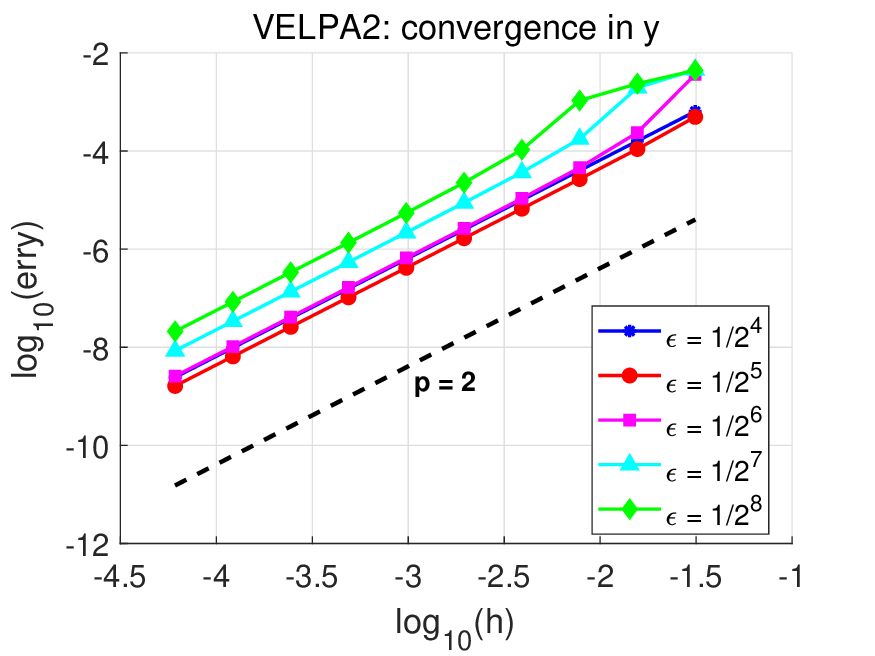}}
	\subfigure{\includegraphics[height=4cm]{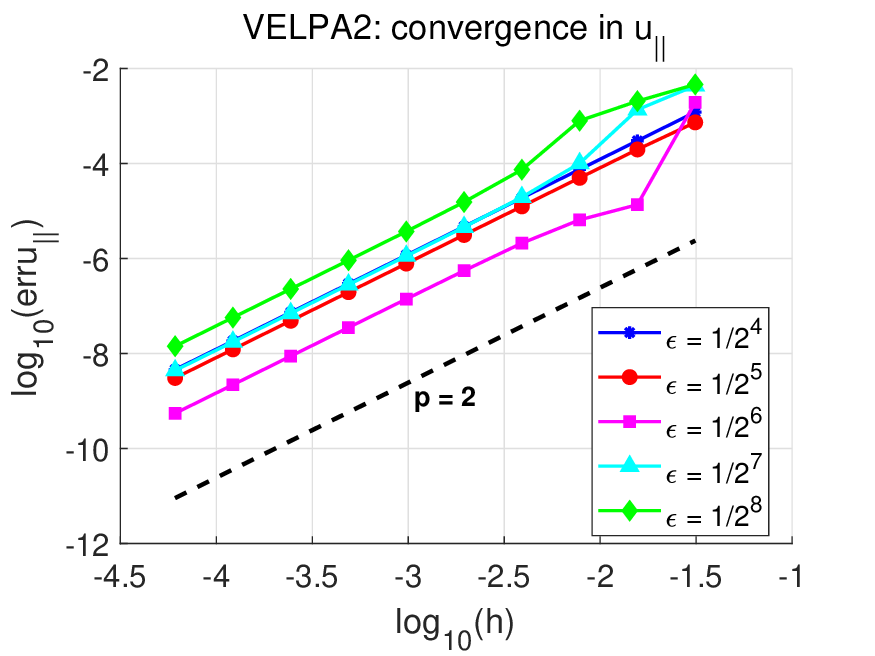}}
	\subfigure{\includegraphics[height=4cm]{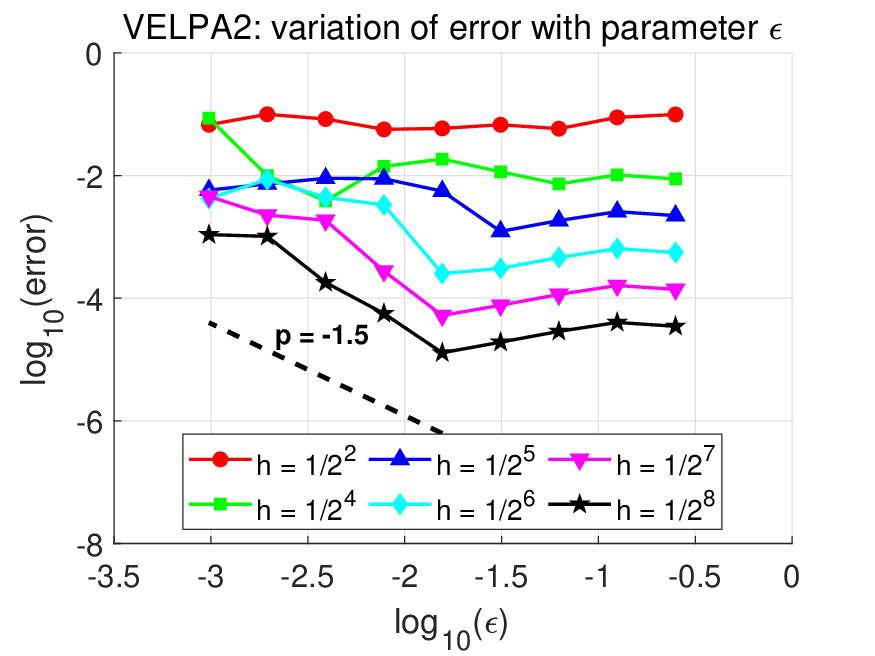}}
	\caption{Example \ref{p1}. VELPA2 errors: \(erry\) and \(erru_{\parallel}\) (left and middle panels) with \( h = 1/2^{k} \) for \( k=5,\dots,14 \) and varying \( \varepsilon \); \eqref{error} (right panel) with \( \varepsilon = 1/2^{k} \) for \( k=2,\dots,10 \) and varying \( h \).}
	\label{V_1}
\end{figure}

%%能量图
\begin{figure}[t!] 
	\centering
	\subfigure{\includegraphics[height=4cm]{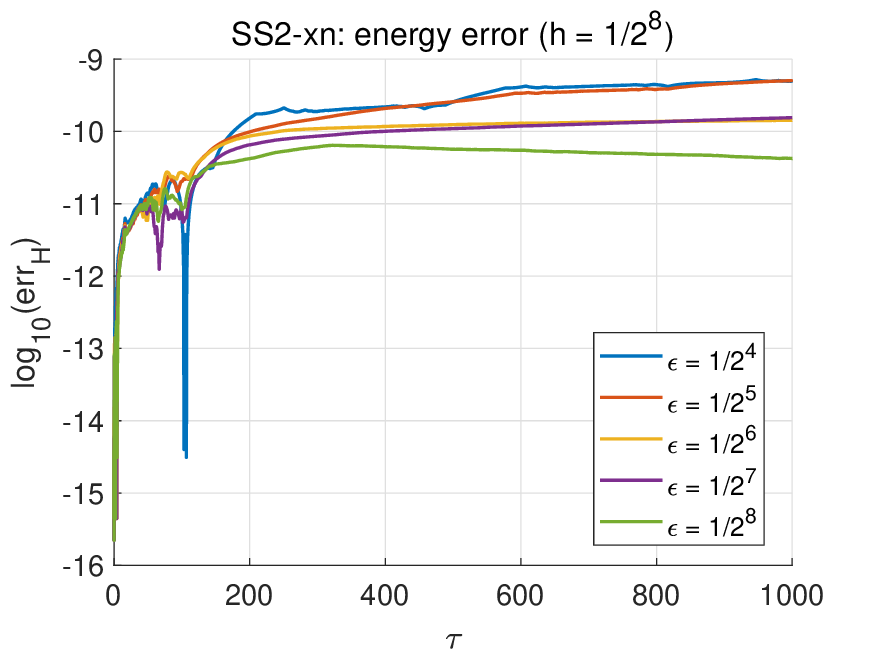}}
	\subfigure{\includegraphics[height=4cm]{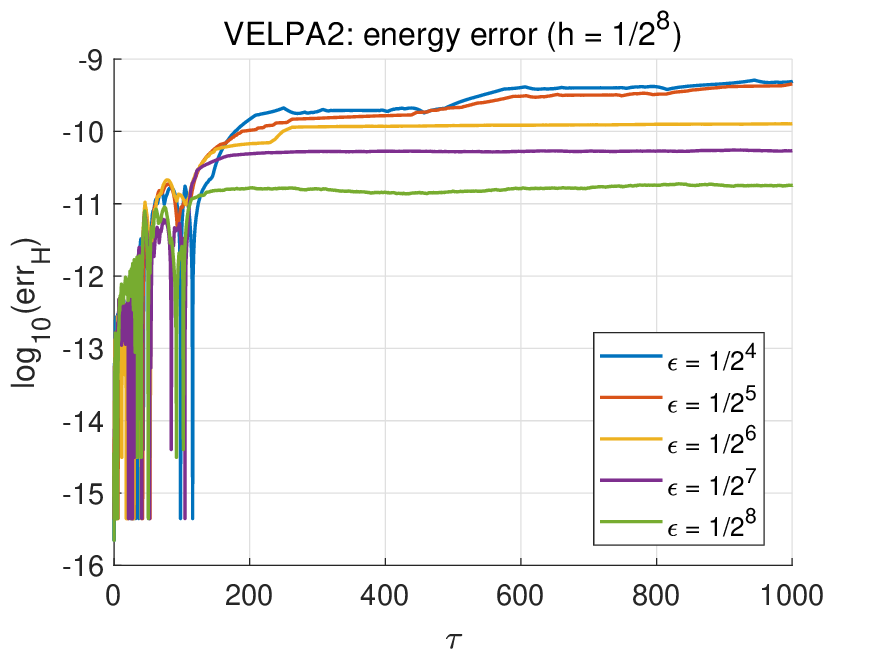}}\\
	\subfigure{\includegraphics[height=4cm]{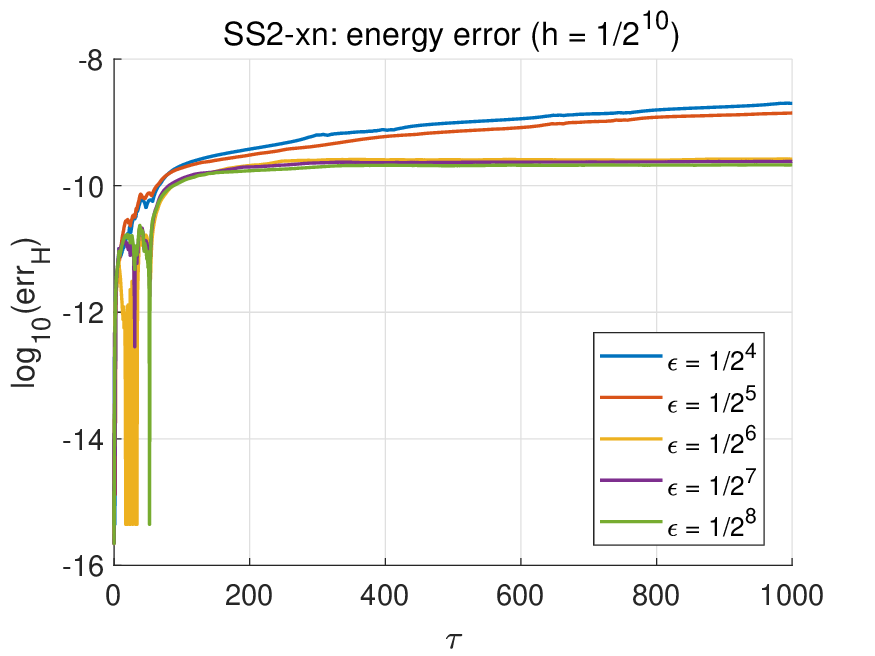}}
	\subfigure{\includegraphics[height=4cm]{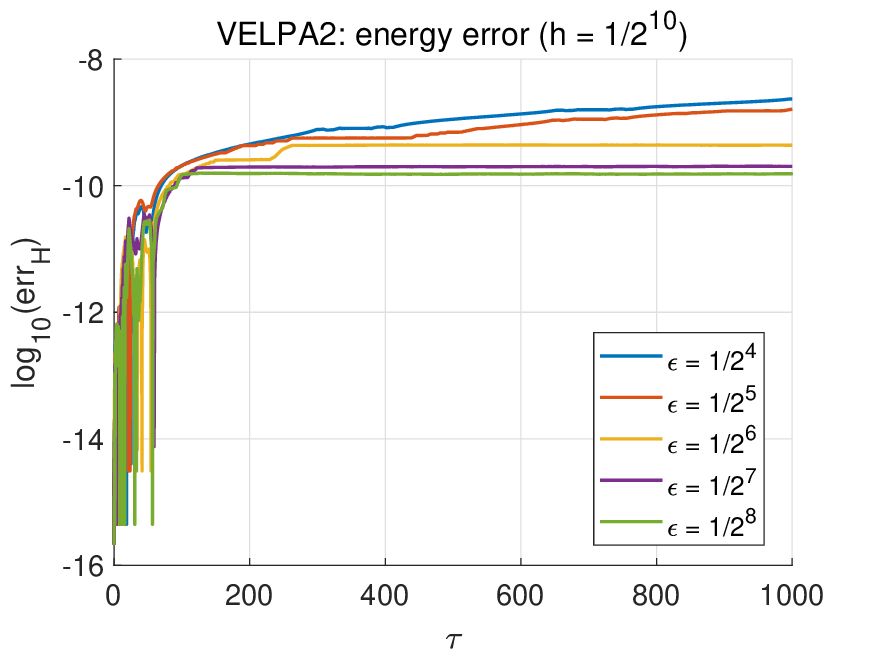}}
	\caption{Example \ref{p1}. The energy errors of SS2-xn (left panels) and VELPA2 (right panels).}
	\label{E_1}
\end{figure}

\begin{example}\label{p2}
	Next, we study the relativistic motion of a charged particle in a strong magnetic field with the MOS:
	$
	\frac{1}{\varepsilon}\bB(\varepsilon \bx)=\frac{1}{\varepsilon}(1 -  \cos(\varepsilon x_{2})  /2,	1 +  \cos(\varepsilon x_{3}) /2 ,1 + \cos(\varepsilon x_{1}) /2)^{\intercal}.
	$ 
	The potential \(U(\bx) = -\sin(x_{1}/2)\sin(x_{2})\sin(x_{3})\) yields the electric field \(\bE = -\nabla U\).
	The initial values are taken as before.
\end{example}

Figs.~\ref{spatial_2}-\ref{velocity_2} illustrate the time evolution of the position and physical velocity of the particle computed by the SS2-xn scheme.
Figs.~\ref{S_2}-\ref{V_2} show the relative errors of the two schemes for Example \ref{p2}, and Fig.~\ref{E_2} presents the energy errors. 
Example \ref{p2} adopts an asymmetric periodic electrostatic field. Under this electric field, the guiding center drifts continuously along a fixed direction determined by the averaged electric field gradient without turning back. Consequently, as $\varepsilon$ decreases (see Fig.~\ref{spatial_2}), the trajectory is gradually stretched and eventually forms a narrow strip-shaped structure. Fig.~\ref{velocity_2} shows that even after long-time evolution with $T=1000$, the physical velocity of the particle remains less than the speed of light $1$, which is consistent with the theoretical prediction. 
Fig.~\ref{S_2} further verifies that for the SS2-xn scheme, the error bound in the $\bu$ direction is $\mathcal{O}(\varepsilon^{-1}h^2)$, while for $\by$ and the parallel component of $\bu$, the scheme indeed exhibits a second-order uniform error bound in $h$ that is independent of the small parameter $\varepsilon$. 
Fig.~\ref{V_2} indicates that for the VEPLA2 scheme, the error bounds in $\by$ and $\bu_{\parallel}$ are of second order in $h$ but depend on $\varepsilon^{-1}$. 
Comparing the two error figures, the proposed scheme shows a clear advantage in terms of error behavior. 
Moreover, both schemes maintain good energy preservation over long-time numerical simulations (see Fig.~\ref{E_2}).

%%空间相图
\begin{figure}[t!] 
	\centering
	\subfigure{\includegraphics[height=4cm]{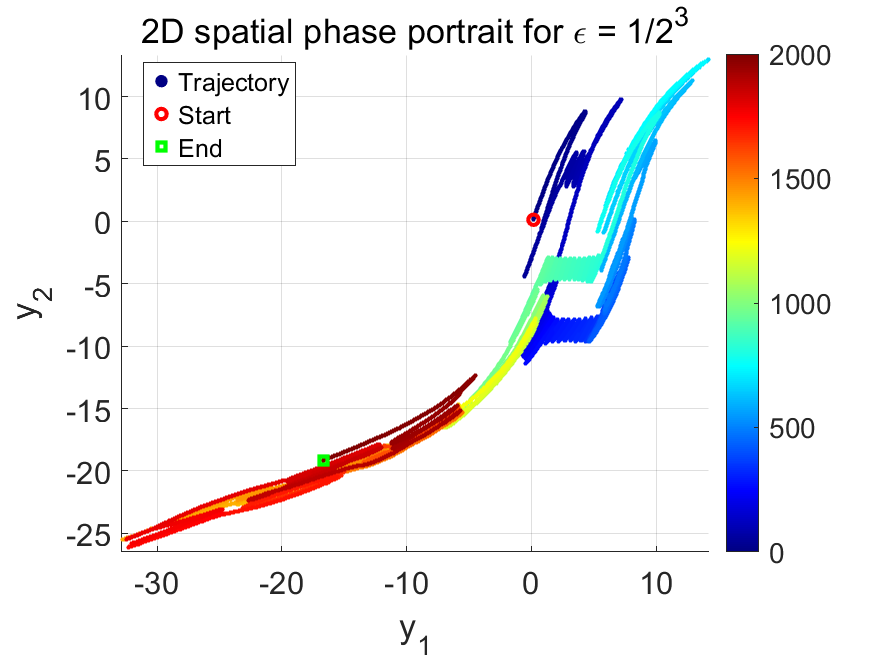}}
	\subfigure{\includegraphics[height=4cm]{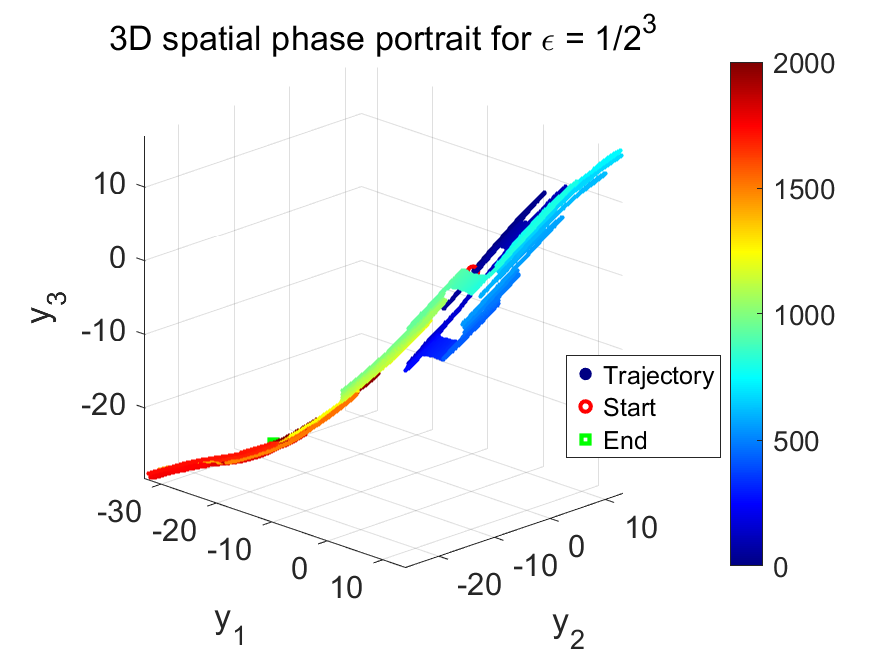}}\\
	\subfigure{\includegraphics[height=4cm]{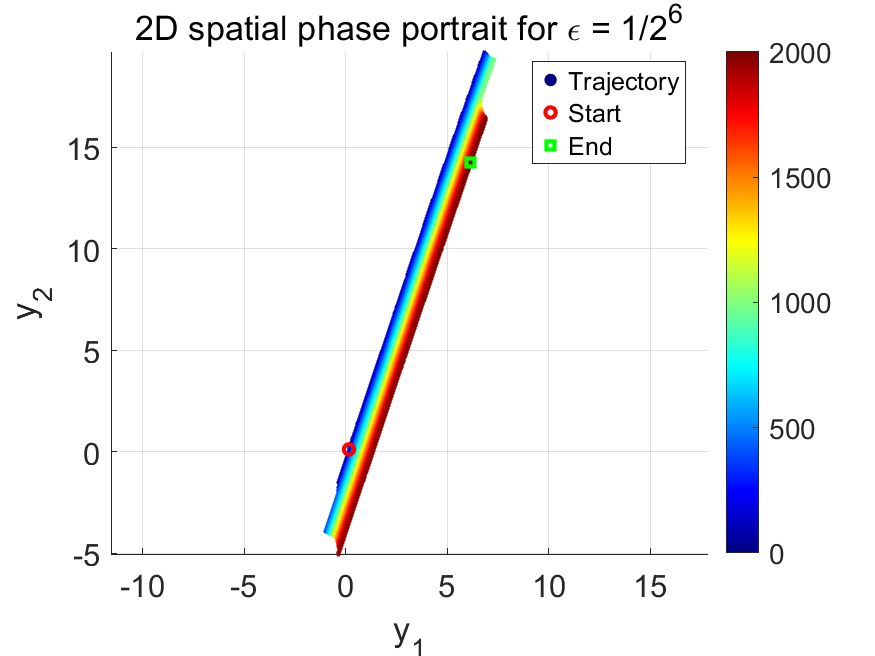}}
	\subfigure{\includegraphics[height=4cm]{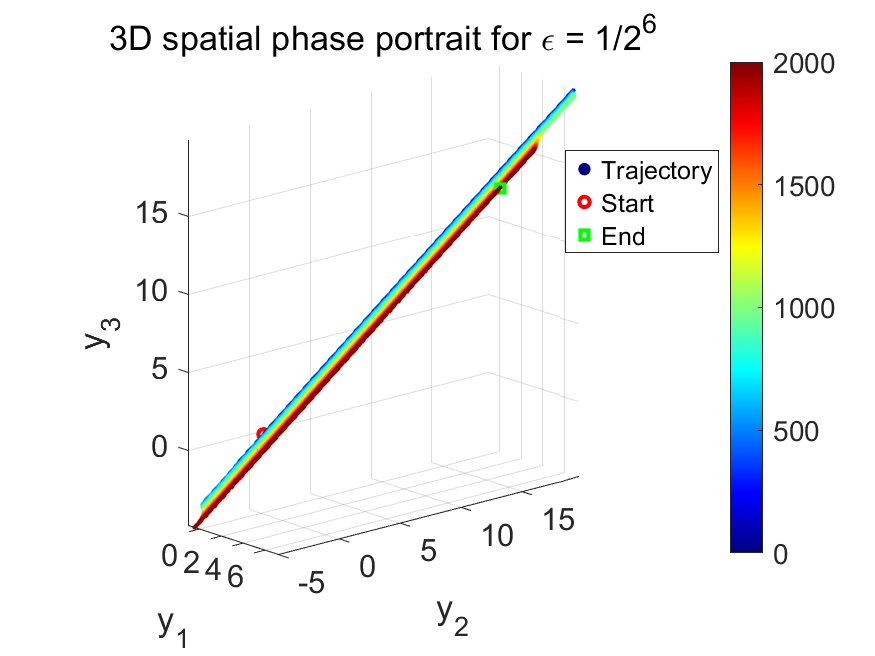}}
	\caption{Example \ref{p2}. 2D (left panels) and 3D (right panels) spatial evolution plots of SS2-xn with \(h = 1/2^8\), \(T = 2000\), and \(\varepsilon\) varying.}
	\label{spatial_2}
\end{figure}

%%速度相图
\begin{figure}[t!] 
	\centering
	\subfigure{\includegraphics[height=4cm]{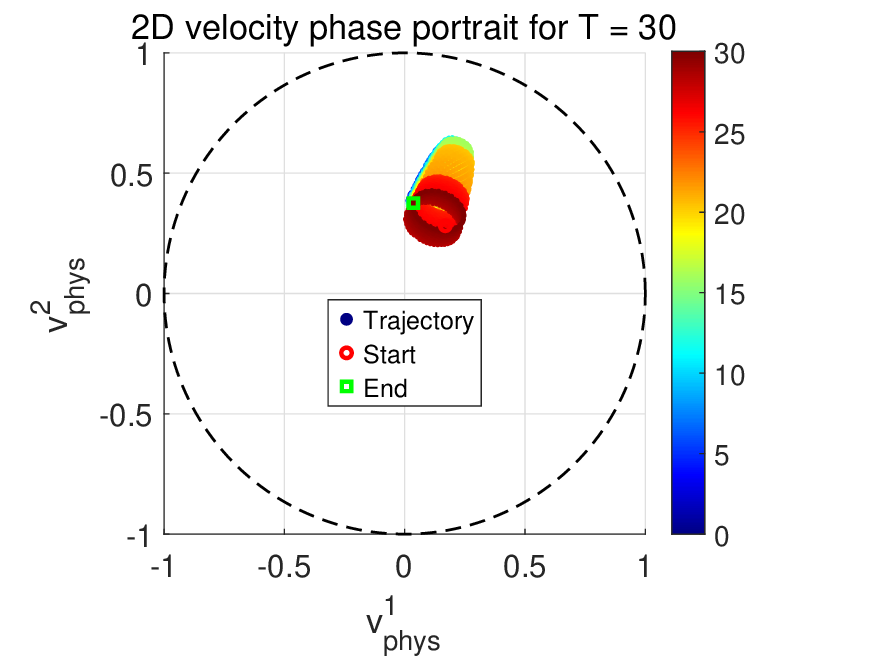}}
	\subfigure{\includegraphics[height=4cm]{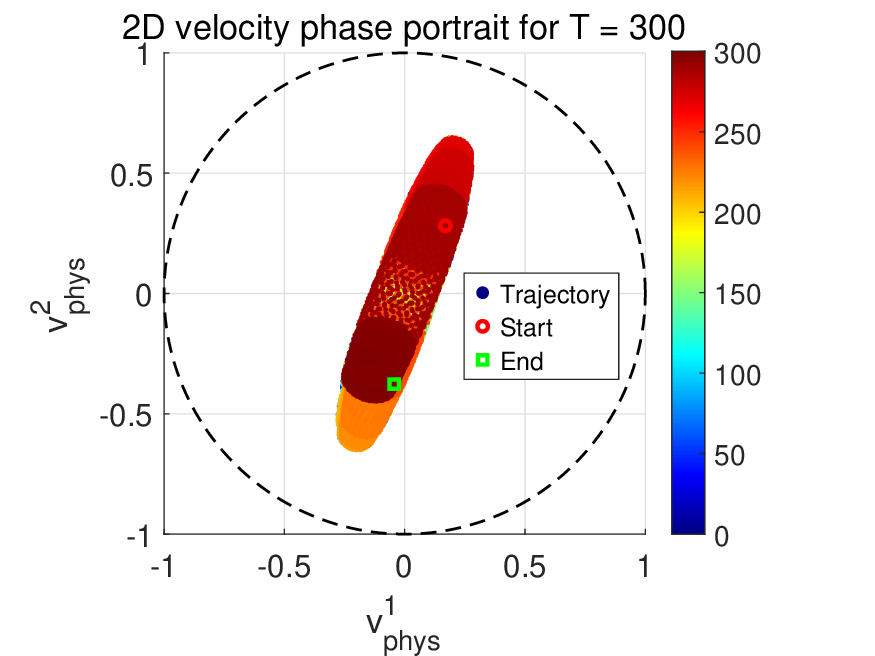}}
	\subfigure{\includegraphics[height=4cm]{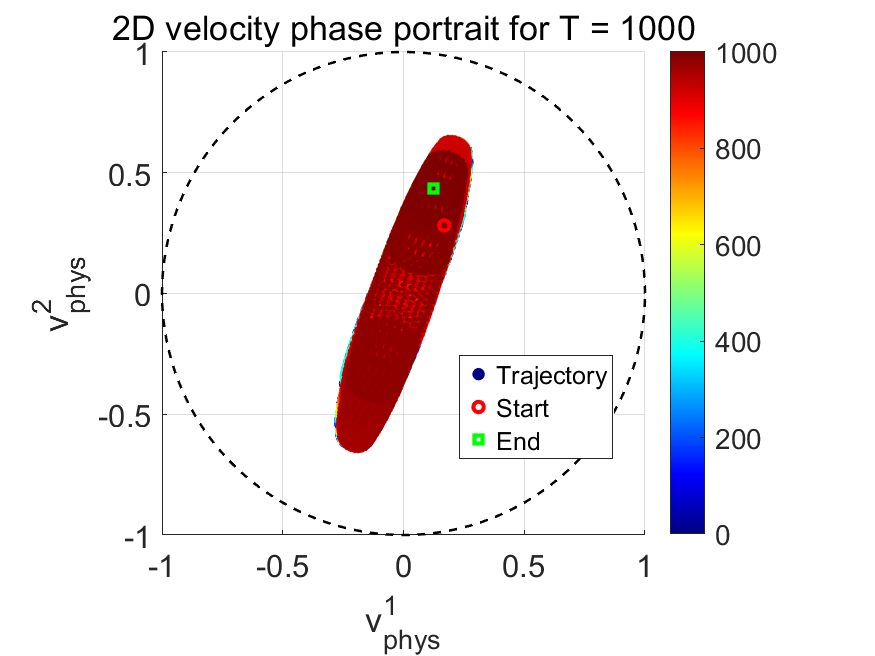}}
	\subfigure{\includegraphics[height=4cm]{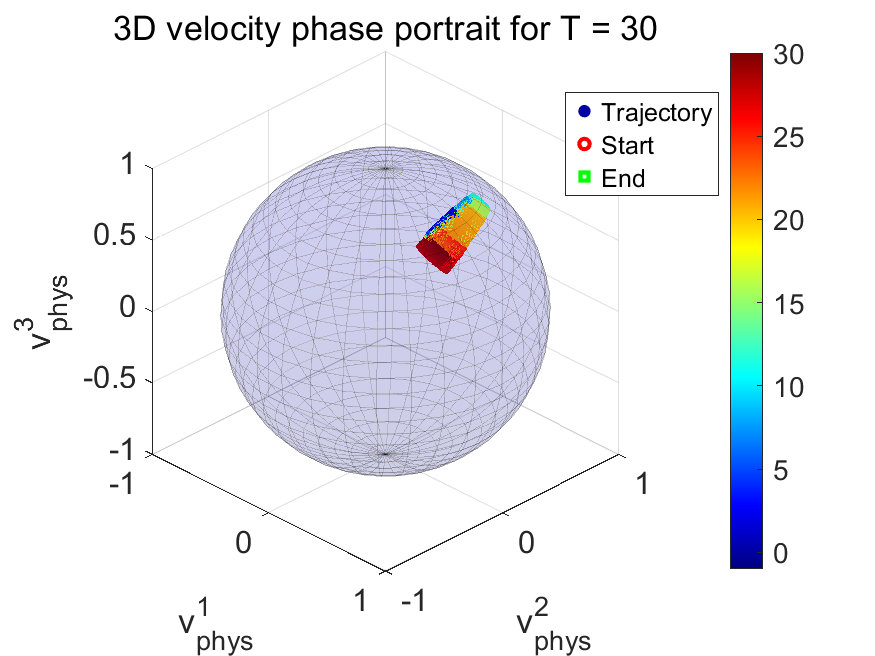}}
	\subfigure{\includegraphics[height=4cm]{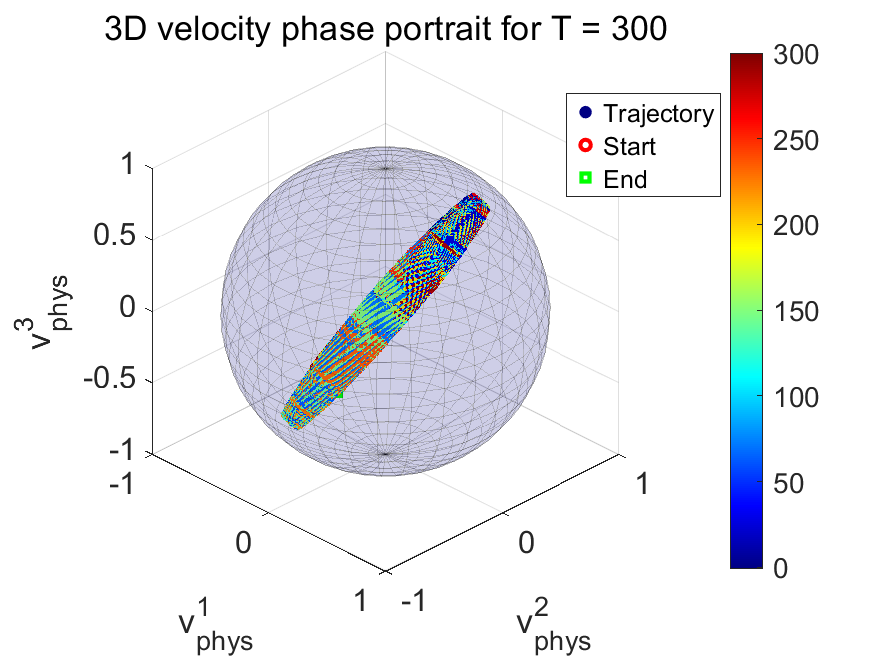}}
	\subfigure{\includegraphics[height=4cm]{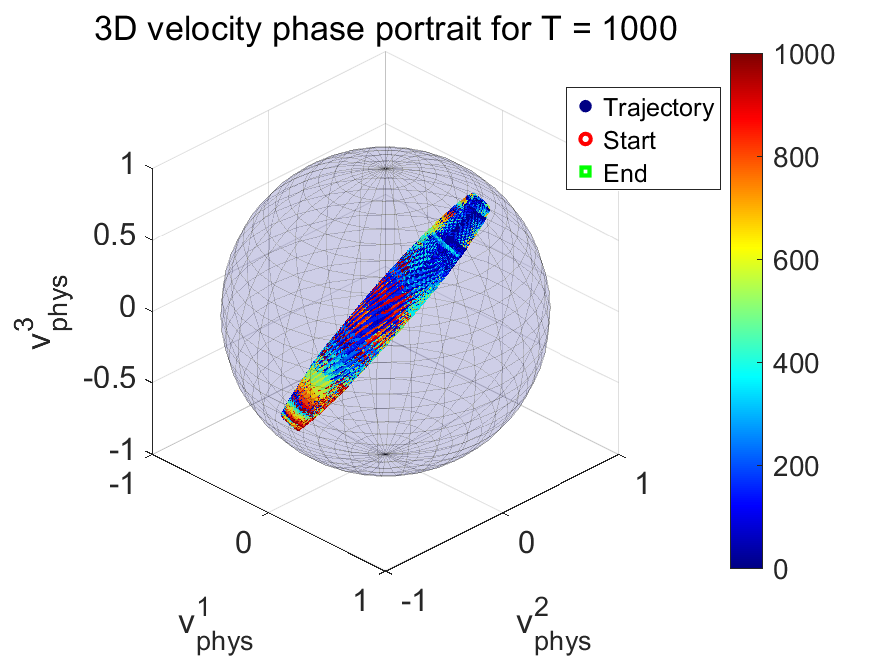}}
	\caption{Example \ref{p2}. Physical velocity evolution of SS2-xn in 2D (top panels) and 3D (bottom panels) for \(h = 1/2^8\), \(\varepsilon = 1/2^5\), and varying \(T\).}
	\label{velocity_2}
\end{figure}

%% SS2-xn 例2
\begin{figure}[t!] 
	\centering
	\subfigure{\includegraphics[height=4cm]{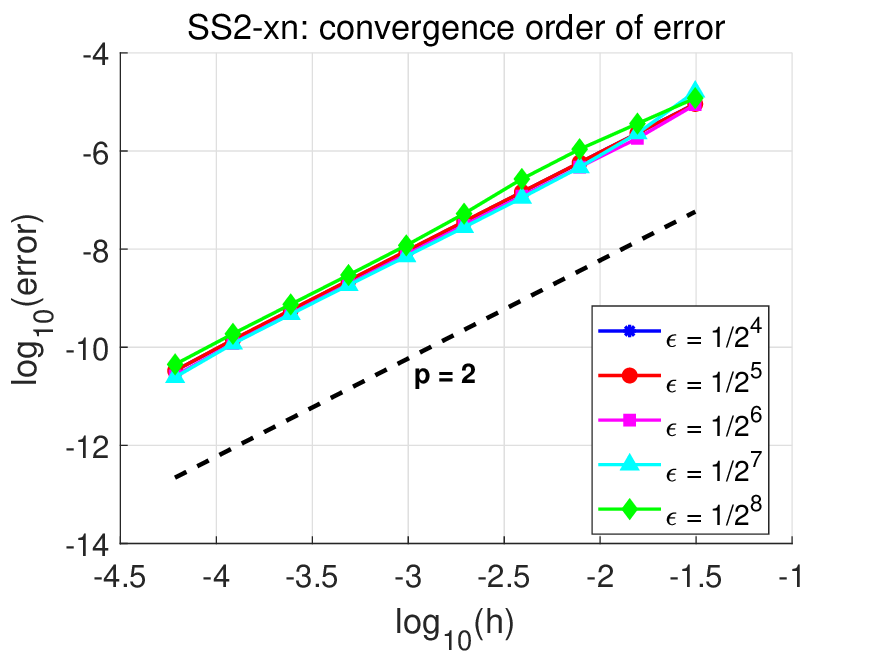}}
	\subfigure{\includegraphics[height=4cm]{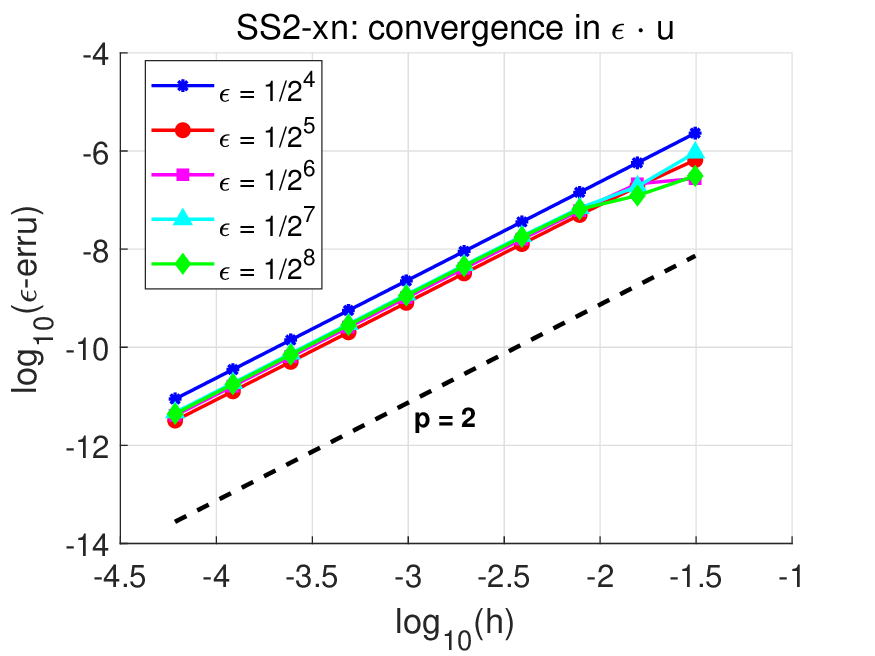}}
	\subfigure{\includegraphics[height=4cm]{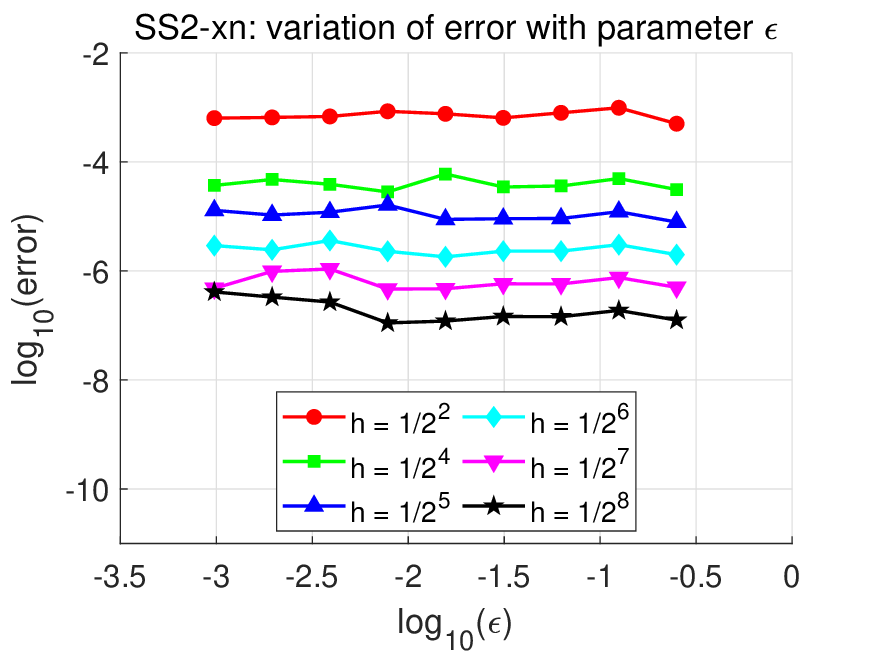}}
	\caption{Example \ref{p2}. SS2-xn errors: \eqref{error} and \eqref{erru} (left and middle panels) with \( h = 1/2^{k} \) for \( k=5,\dots,14 \) and varying \( \varepsilon \); \eqref{error} (right panel) with \( \varepsilon = 1/2^{k} \) for \( k=2,\dots,10 \) and varying \( h \).}
	\label{S_2}
\end{figure}

%% VELPA2 例2
\begin{figure}[t!] 
	\centering
	\subfigure{\includegraphics[height=4cm]{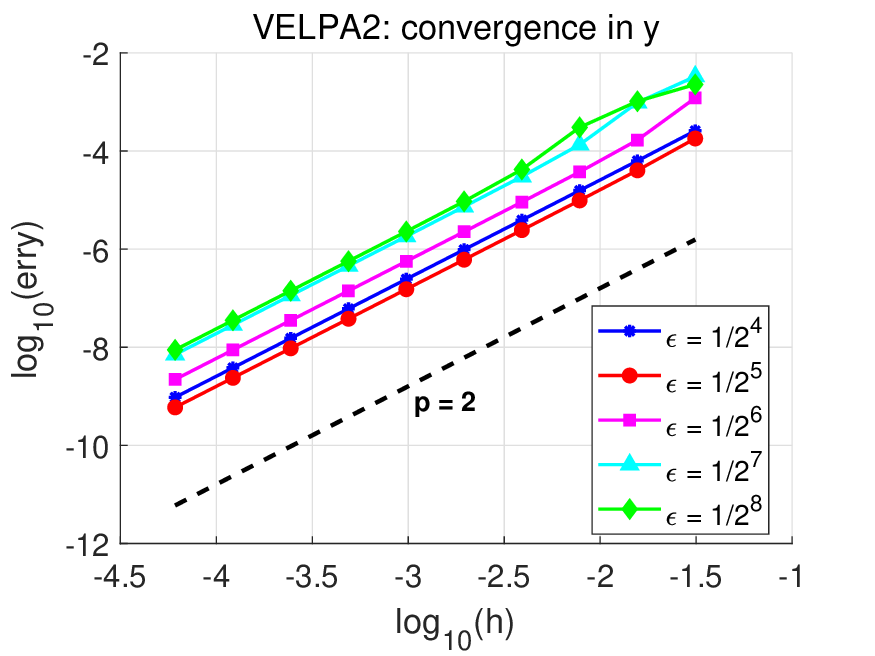}}
	\subfigure{\includegraphics[height=4cm]{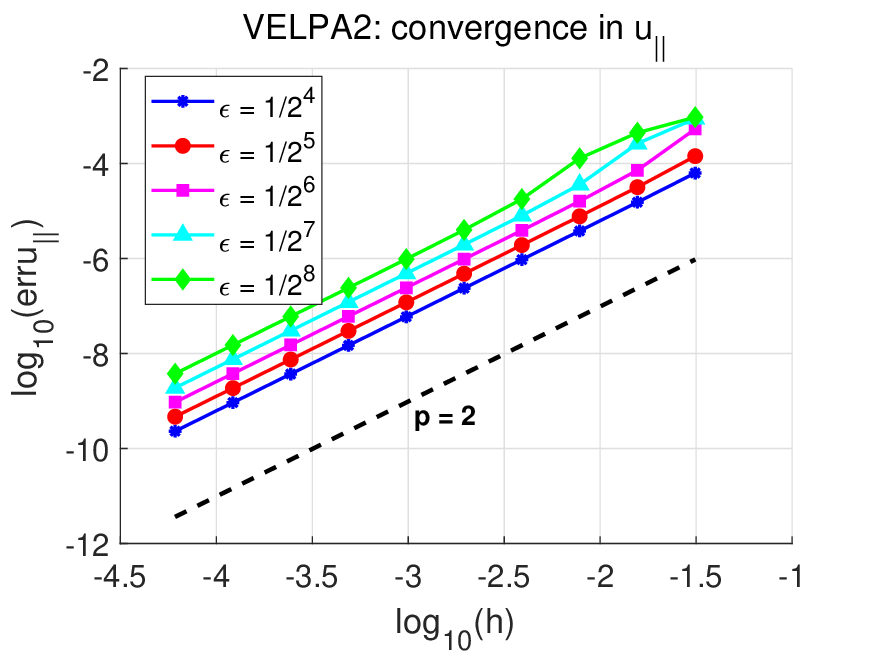}}
	\subfigure{\includegraphics[height=4cm]{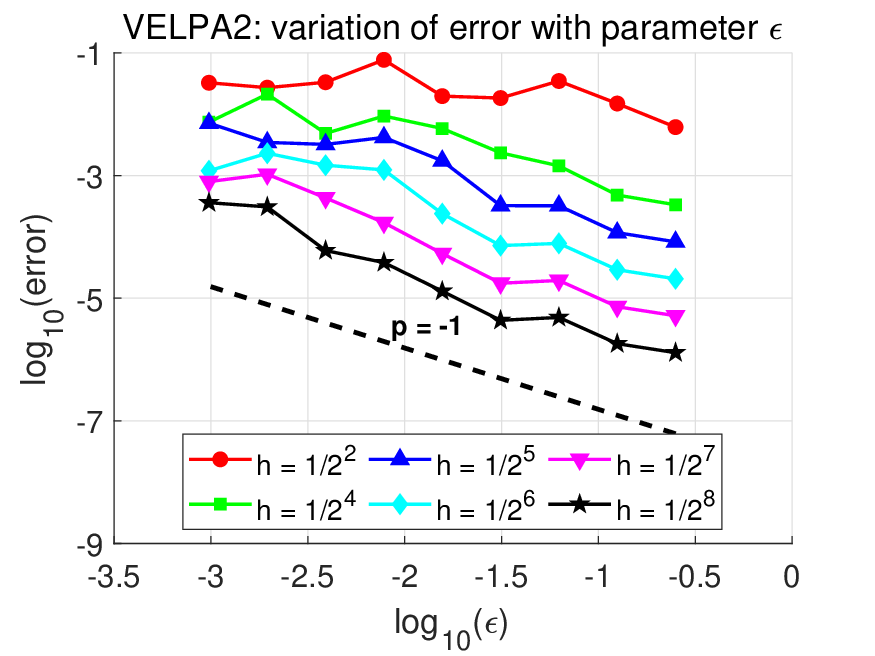}}
	\caption{Example \ref{p2}. VELPA2 errors: \(erry\) and \(erru_{\parallel}\) (left and middle panels) with \( h = 1/2^{k} \) for \( k=5,\dots,14 \) and varying \( \varepsilon \); \eqref{error} (right panel) with \( \varepsilon = 1/2^{k} \) for \( k=2,\dots,10 \) and varying \( h \).}
	\label{V_2}
\end{figure}

%%能量图
\begin{figure}[t!] 
	\centering
	\subfigure{\includegraphics[height=4cm]{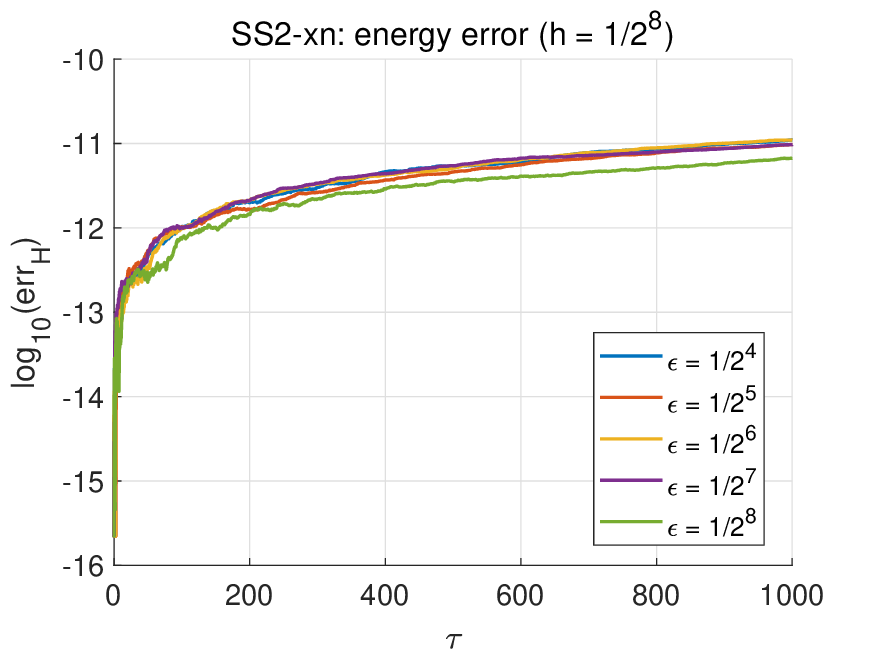}}
	\subfigure{\includegraphics[height=4cm]{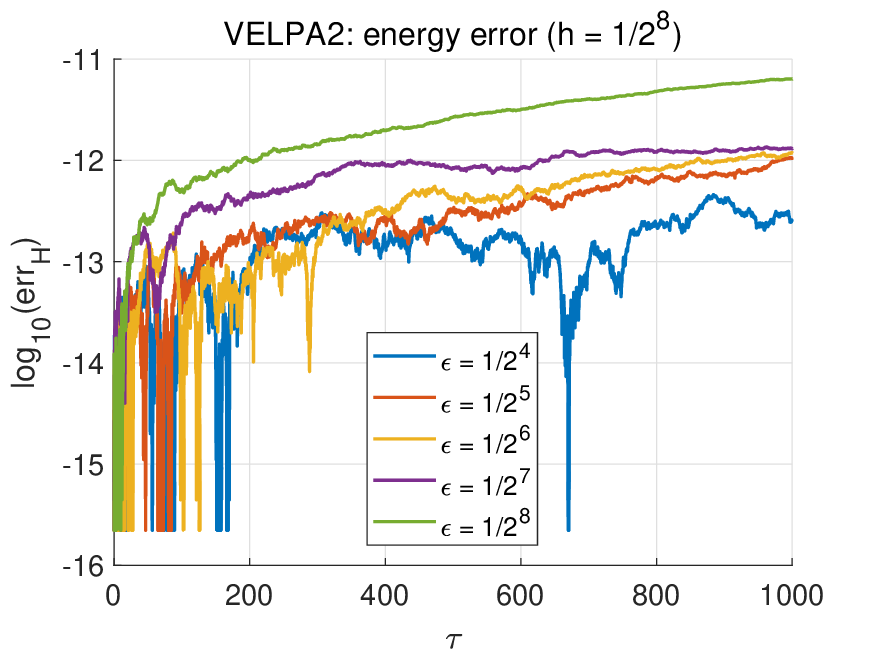}}\\
	\subfigure{\includegraphics[height=4cm]{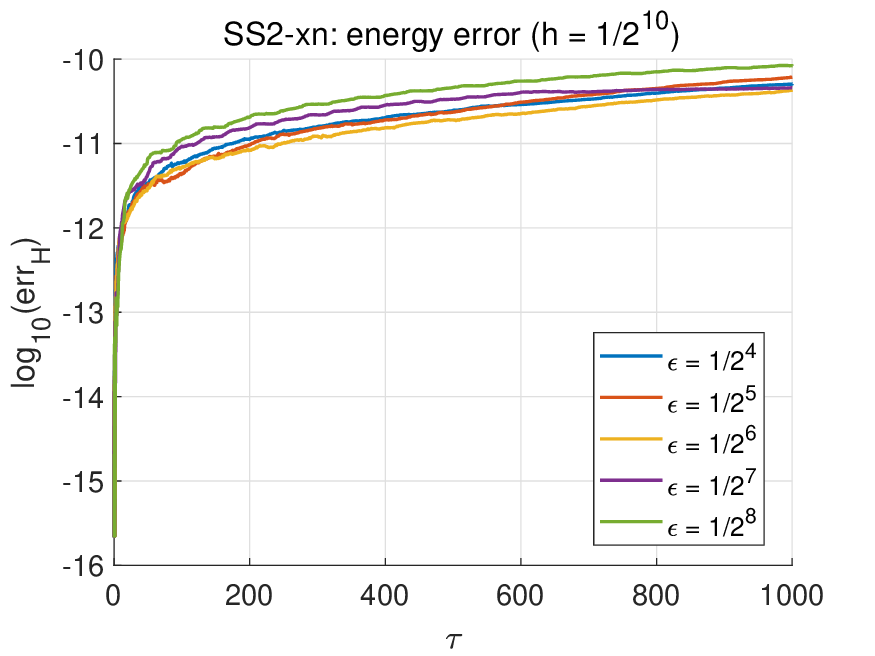}}
	\subfigure{\includegraphics[height=4cm]{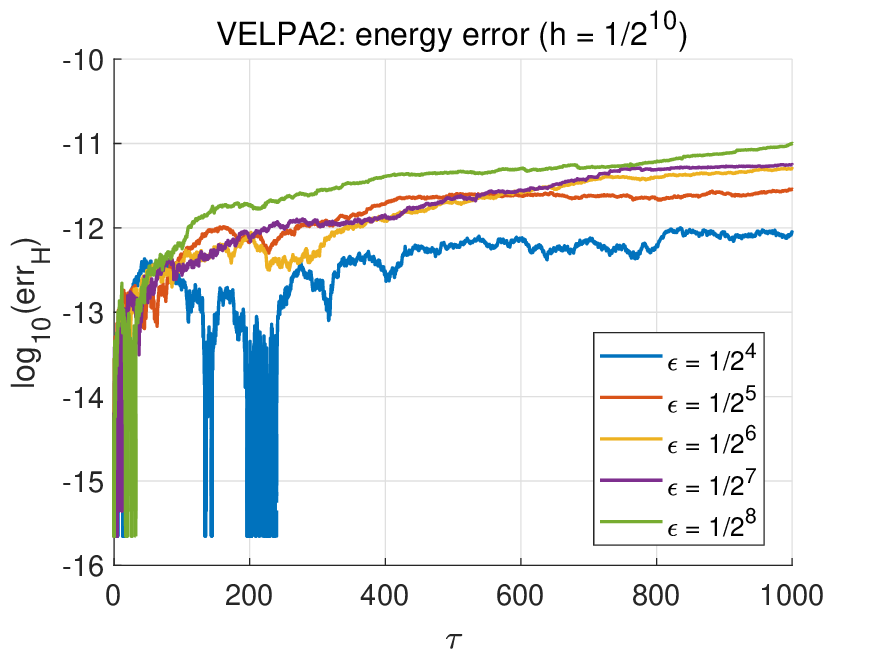}}
	\caption{Example \ref{p2}. The energy errors of SS2-xn (left panels) and VELPA2 (right panels).}
	\label{E_2}
\end{figure}

\begin{example}\label{p3}
	In the final test, the strong magnetic field under the MOS is taken as
	$ 
	\frac{1}{\varepsilon}\bB(\varepsilon \bx)=\frac{1}{\varepsilon}(1 - \cos(\varepsilon x_{1}),\sin(\varepsilon x_{3}) - \varepsilon x_{3},1 - (\cos(\varepsilon x_{2}))/2)^{\intercal}.
	$ 
	The potential is chosen as \(U(\bx) = x_{1}^{3} - x_{2}^{3} + \frac{1}{5}x_{1}^{4} + x_{2}^{4} + x_{3}^{4}\).
	The initial conditions are unchanged from the previous example. 
\end{example}

The particle evolution computed by the SS2-xn scheme and the error results of the two schemes for Example~\ref{p3} are presented in Figs.~\ref{spatial_3}-\ref{V_3}.
This numerical example employs a polynomial confining potential. The corresponding electric field generates a finite potential well, which causes the drift motion to reverse direction repeatedly. The guiding center can only oscillate back and forth inside the potential well; the particle is tightly confined within a bounded spatial region and cannot drift continuously in a single direction. As a result, compact quasi-periodic phase orbits are formed in the finite phase space. When $\varepsilon$ is sufficiently small, the strong magnetic field effectively suppresses transverse perturbations, and the trajectory evolves into the loop-shaped bundle structure shown in Fig.~\ref{spatial_3}. Fig.~\ref{velocity_3} further verifies the physical validity of the particle velocity. 
Fig.~\ref{S_3} shows that SS2-xn possesses a second-order uniform error bound, which is consistent with the theoretical results of Theorem \ref{th2.1}. 
As can be seen from the top panels of Fig.~\ref{V_3}, the error of VELPA2 in the $\by$ direction still depends on the small parameter $\varepsilon$. Since the difference in the dependence of the overall error bounds of the two schemes on $\varepsilon$ is not obvious, we additionally present a comparison of the errors of the two schemes in the $\by$ direction. It can be observed that the error of the VELPA2 scheme grows at the order of $\varepsilon^{-1}$ (see the bottom panels of Fig.~\ref{V_3}).
As shown in Fig.~\ref{E_3}, both schemes preserve energy conservation during long-time integration.

%%空间相图
\begin{figure}[t!] 
	\centering
	\subfigure{\includegraphics[height=4cm]{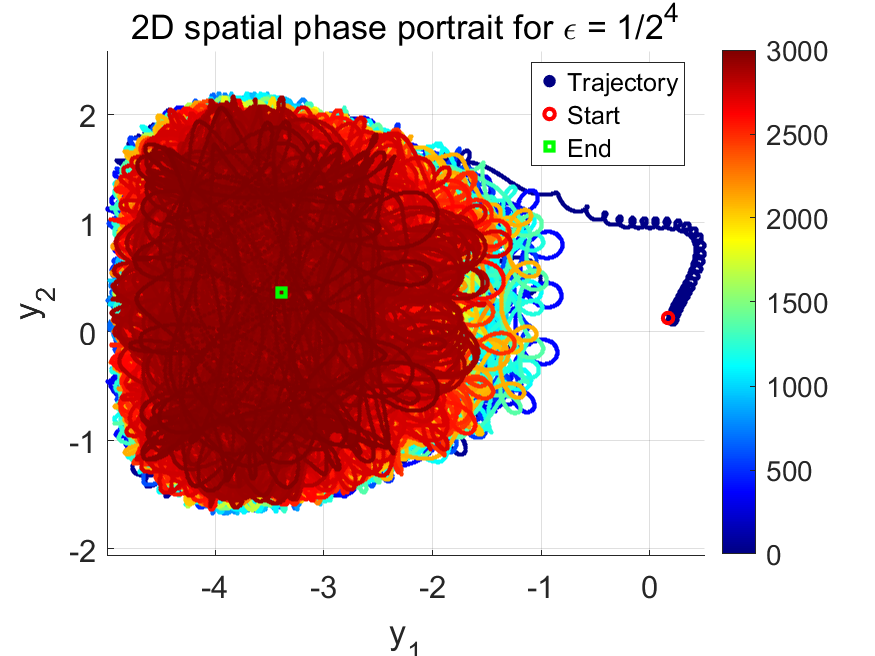}}
	\subfigure{\includegraphics[height=4cm]{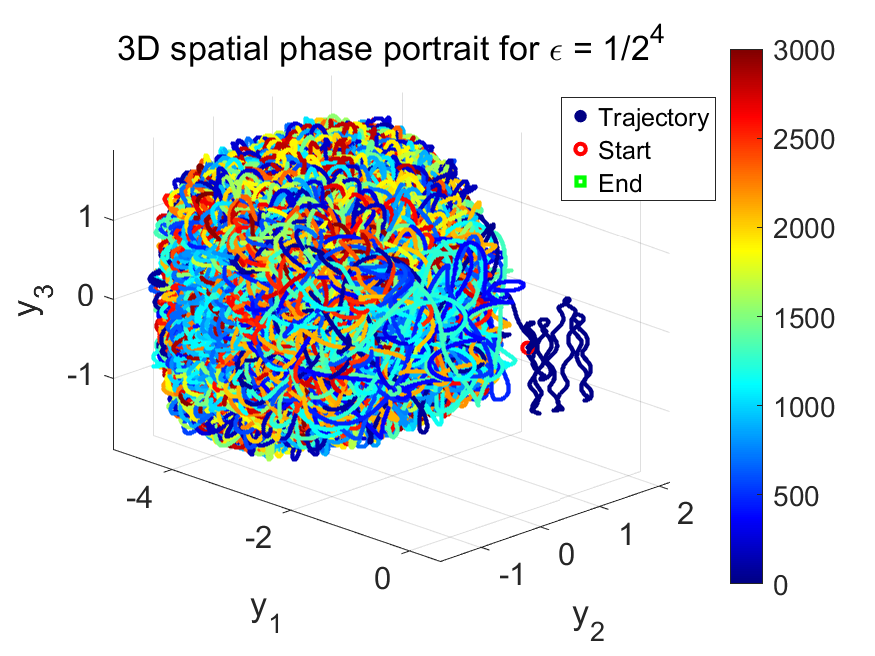}}\\
	\subfigure{\includegraphics[height=4cm]{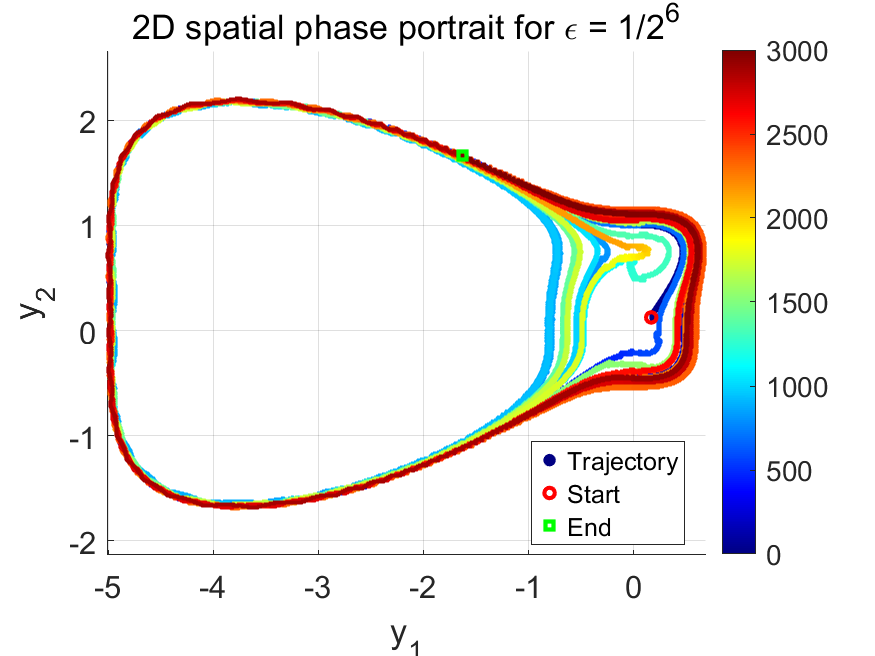}}
	\subfigure{\includegraphics[height=4cm]{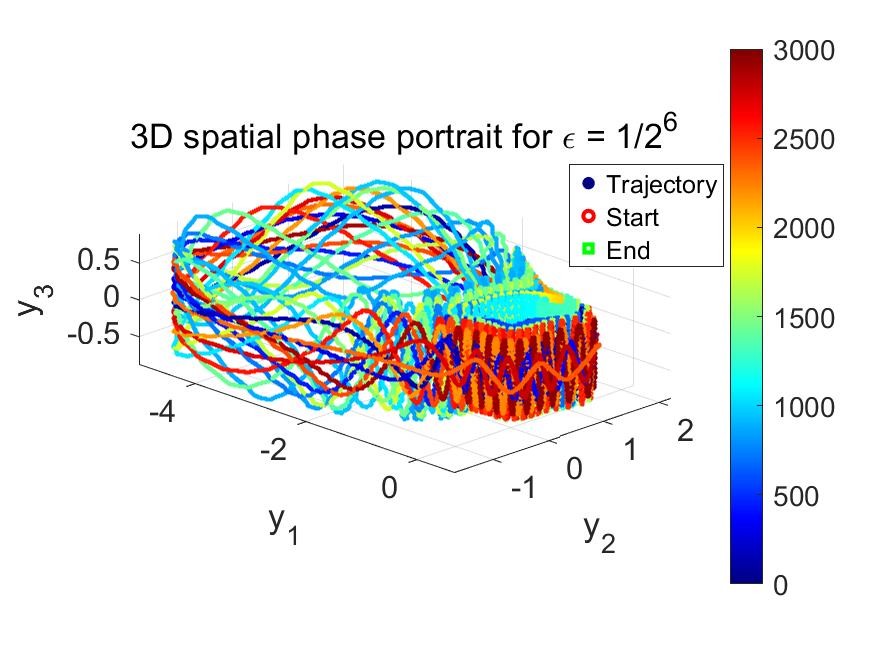}}
	\caption{Example \ref{p3}. 2D (left panels) and 3D (right panels) spatial evolution plots of SS2-xn with \(h = 1/2^8\), \(T = 3000\), and \(\varepsilon\) varying.}
	\label{spatial_3}
\end{figure}

%%速度相图
\begin{figure}[h!] 
	\centering
	\subfigure{\includegraphics[height=4cm]{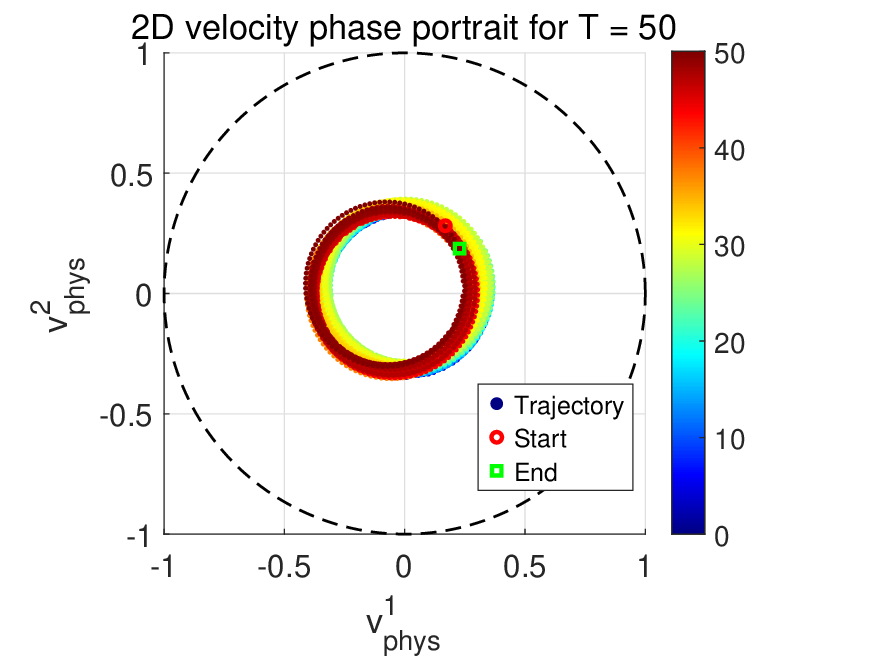}}
	\subfigure{\includegraphics[height=4cm]{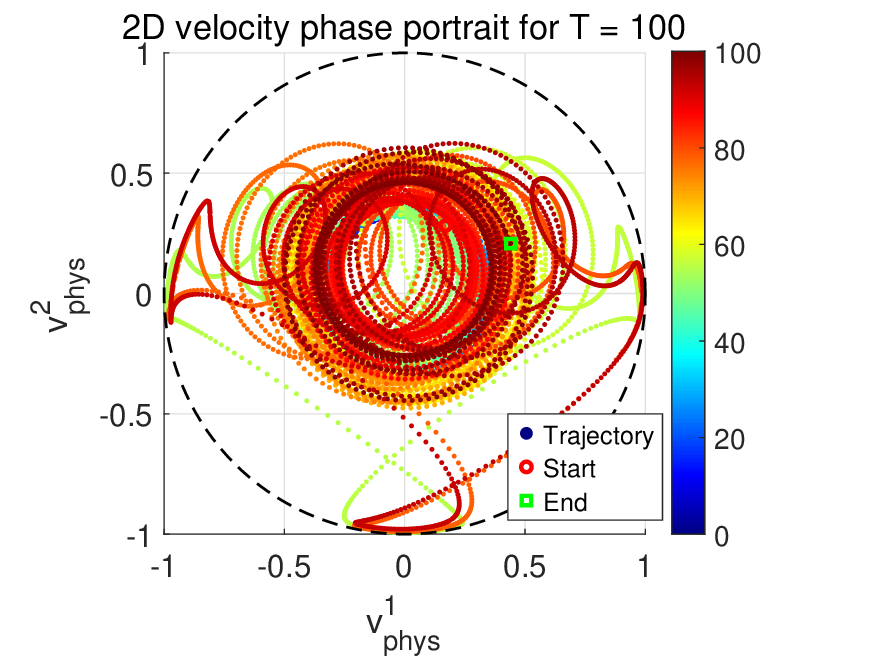}}
	\subfigure{\includegraphics[height=4cm]{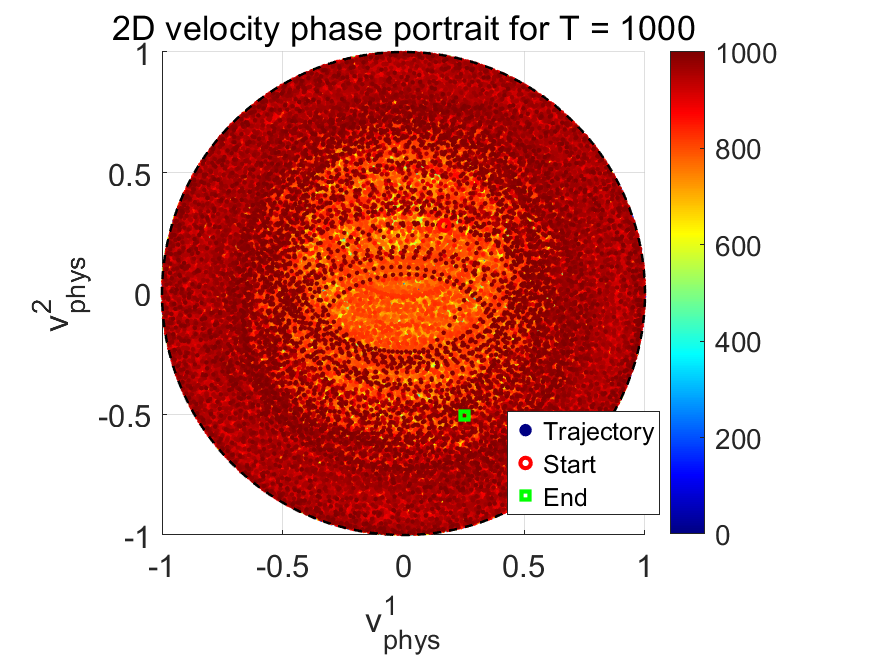}}
	\subfigure{\includegraphics[height=4cm]{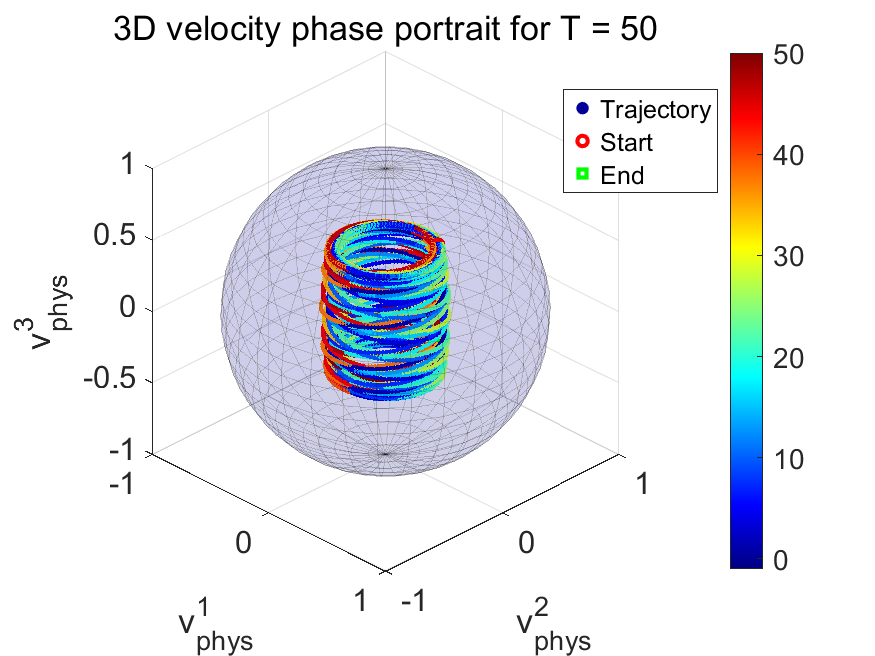}}
	\subfigure{\includegraphics[height=4cm]{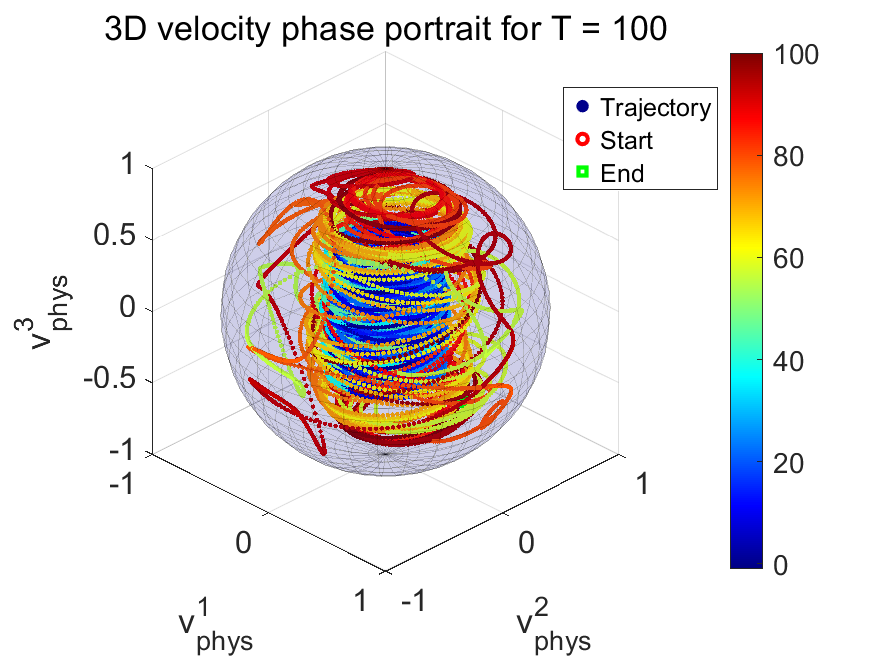}}
	\subfigure{\includegraphics[height=4cm]{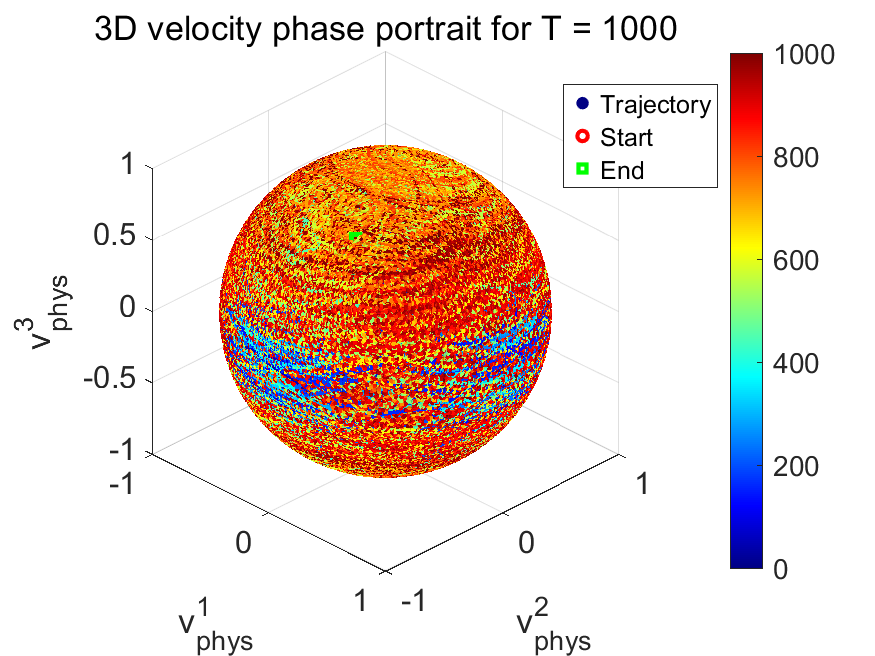}}
	\caption{Example \ref{p3}. Physical velocity evolution of SS2-xn in 2D (top panels) and 3D (bottom panels) for \(h = 1/2^8\), \(\varepsilon = 1/2^5\), and varying \(T\).}
	\label{velocity_3}
\end{figure}

%% SS2-xn 例3
\begin{figure}[t!] 
	\centering
	\subfigure{\includegraphics[height=4cm]{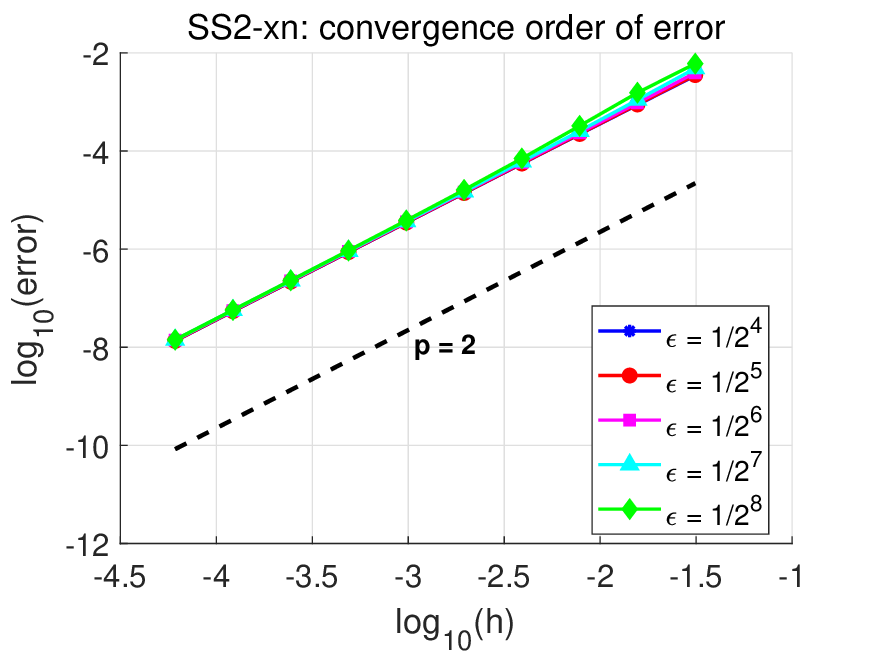}}
	\subfigure{\includegraphics[height=4cm]{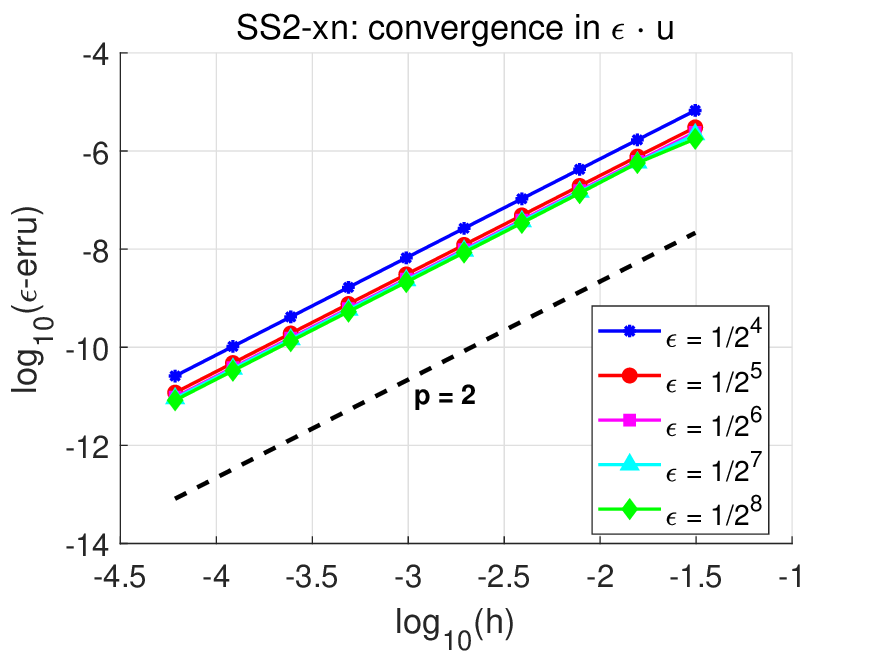}}\\
	\subfigure{\includegraphics[height=4cm]{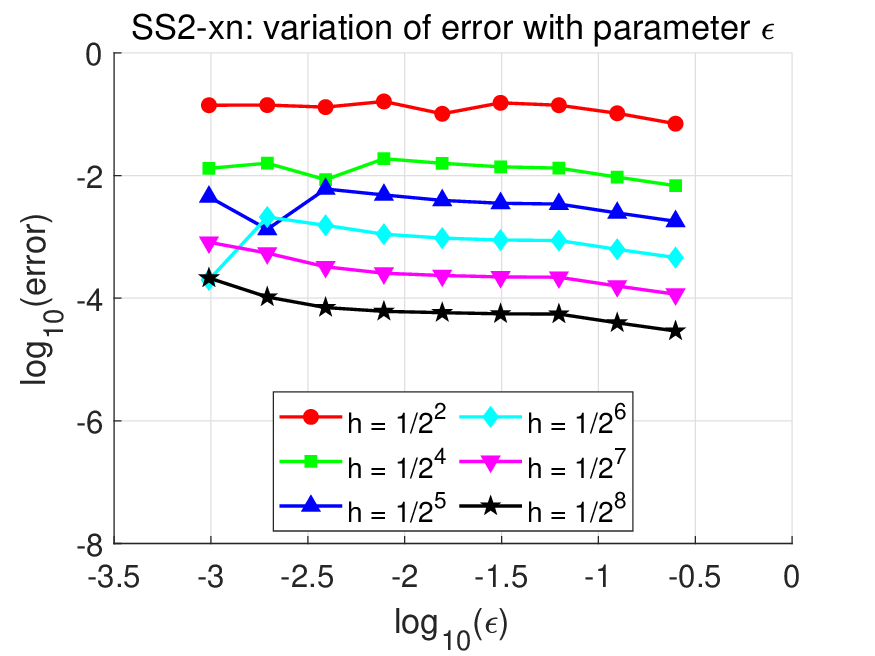}}
	\subfigure{\includegraphics[height=4cm]{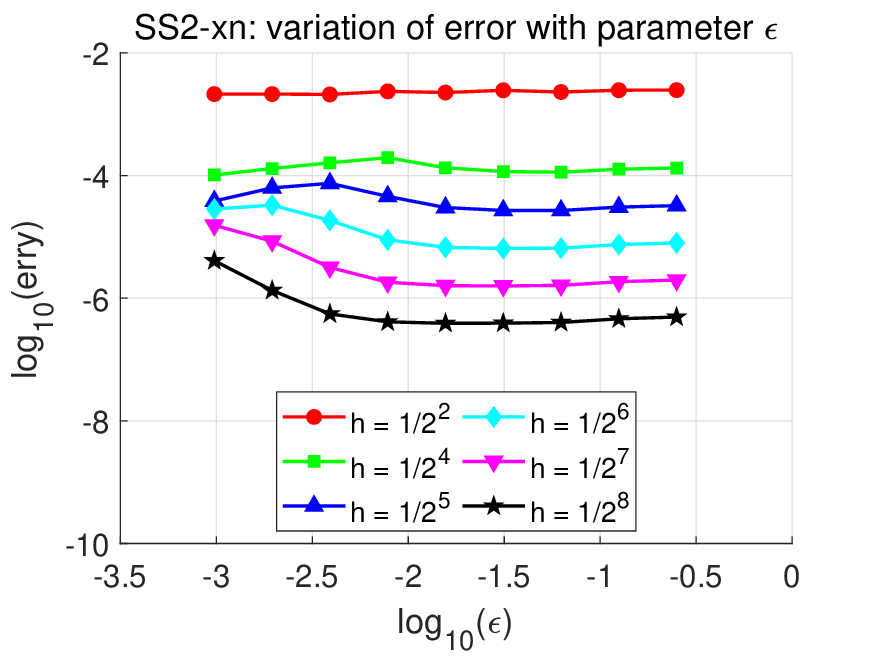}}
	\caption{Example \ref{p3}. SS2-xn errors: \eqref{error} and \eqref{erru} (top panels) with \( h = 1/2^{k} \) for \( k=5,\dots,14 \) and varying \( \varepsilon \); \eqref{error} and \(erry\) (bottom panels) with \( \varepsilon = 1/2^{k} \) for \( k=2,\dots,10 \) and varying \( h \).}
	\label{S_3}
\end{figure}

%% VELPA2 例3
\begin{figure}[t!] 
	\centering
	\subfigure{\includegraphics[height=4cm]{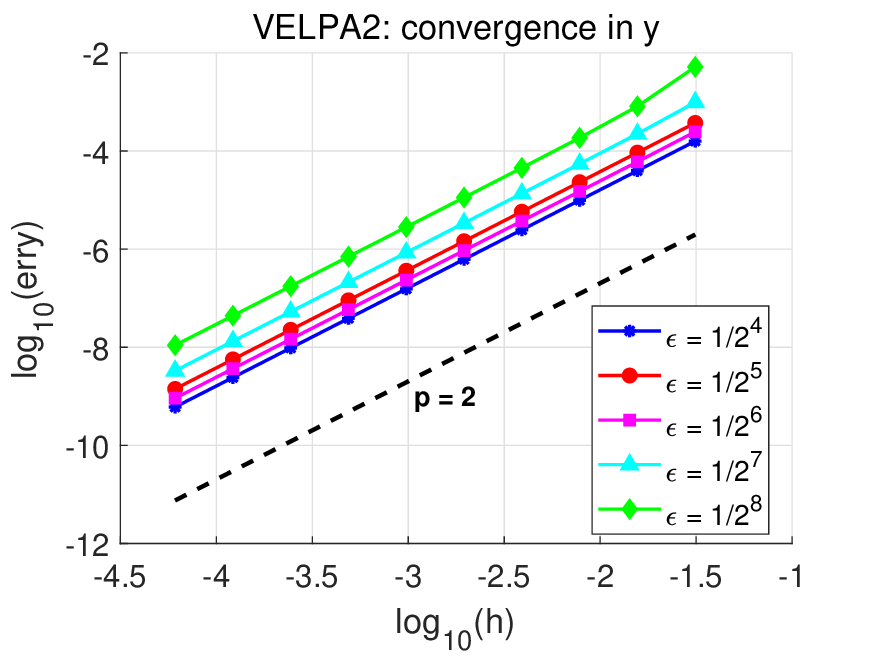}}
	\subfigure{\includegraphics[height=4cm]{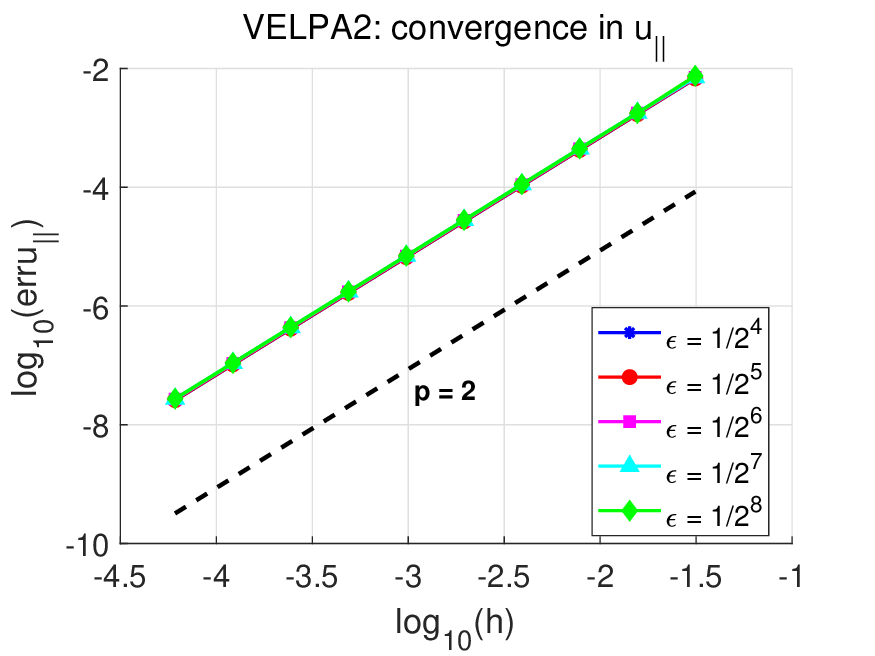}}\\
	\subfigure{\includegraphics[height=4cm]{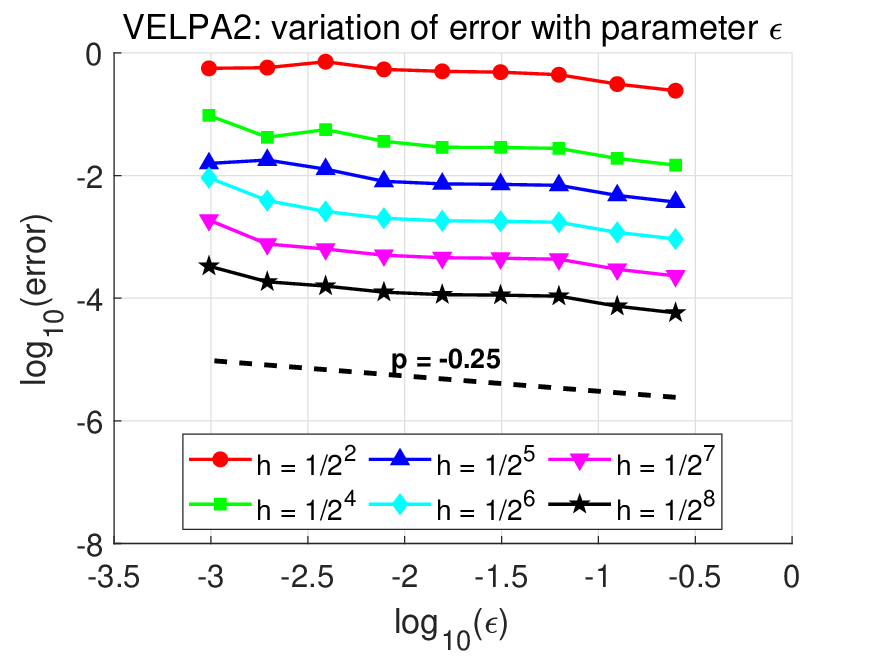}}
	\subfigure{\includegraphics[height=4cm]{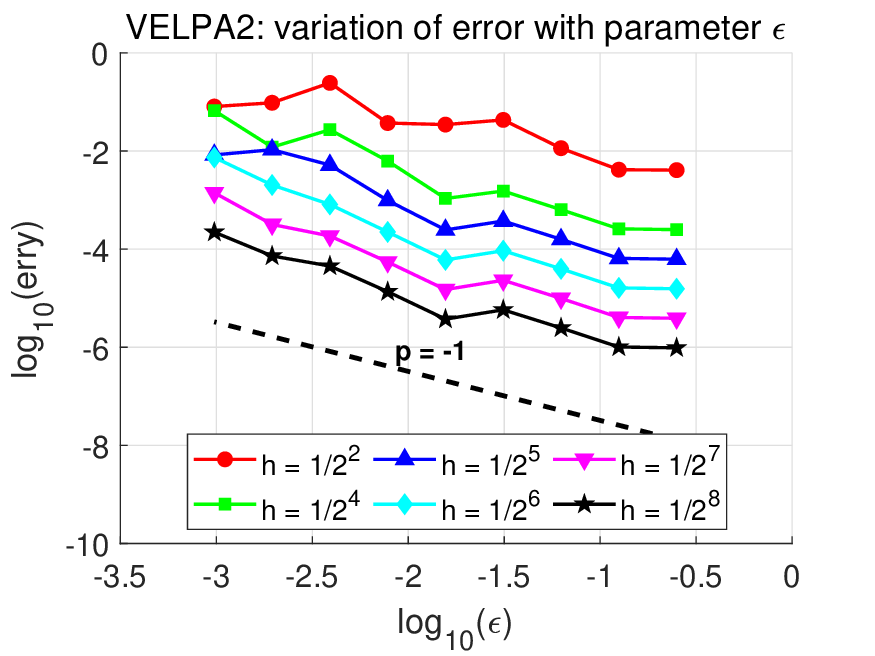}}
	\caption{Example \ref{p3}. VELPA2 errors: \(erry\) and \(erru_{\parallel}\) (top panels) with \( h = 1/2^{k} \) for \( k=5,\dots,14 \) and varying \( \varepsilon \); \eqref{error} and \(erry\) (bottom panels) with \( \varepsilon = 1/2^{k} \) for \( k=2,\dots,10 \) and varying \( h \).}
	\label{V_3}
\end{figure}

%%能量图
\begin{figure}[t!] 
	\centering
	\subfigure{\includegraphics[height=4cm]{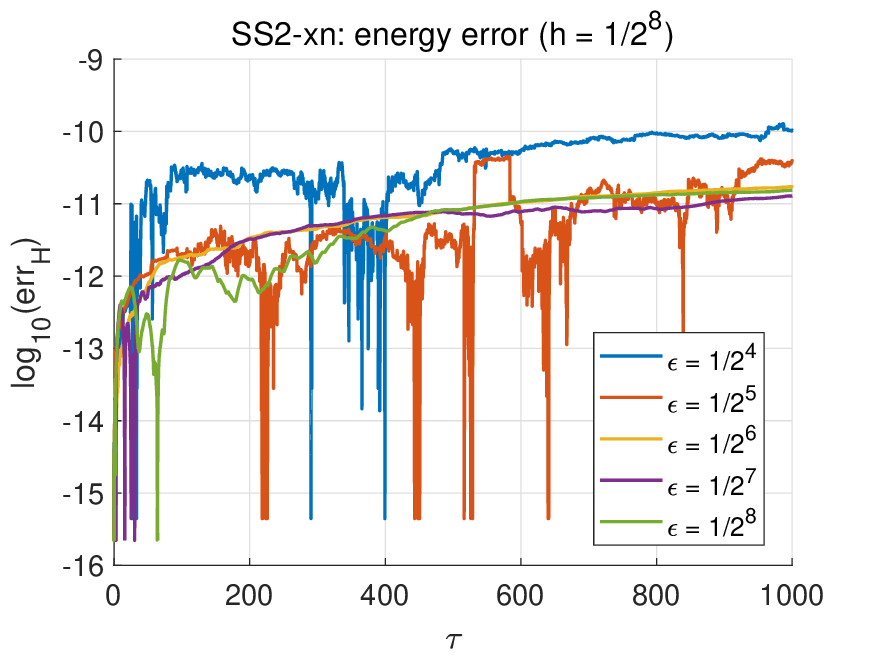}}
	\subfigure{\includegraphics[height=4cm]{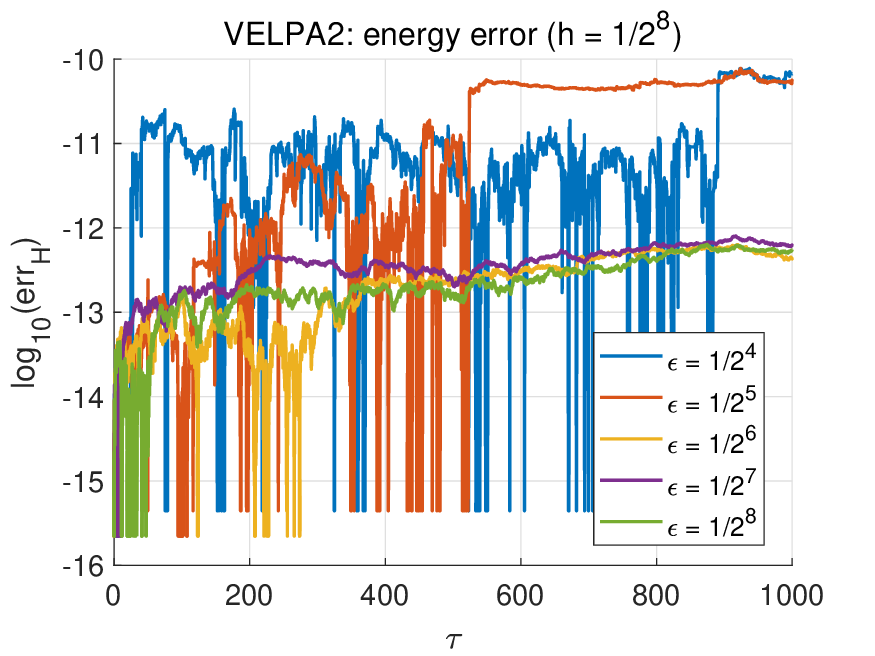}}\\
	\subfigure{\includegraphics[height=4cm]{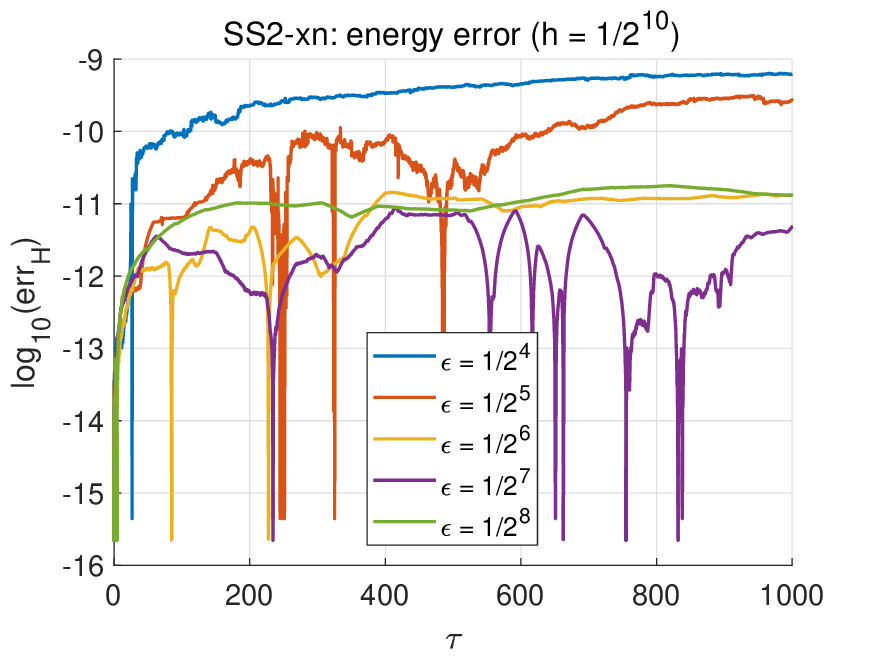}}
	\subfigure{\includegraphics[height=4cm]{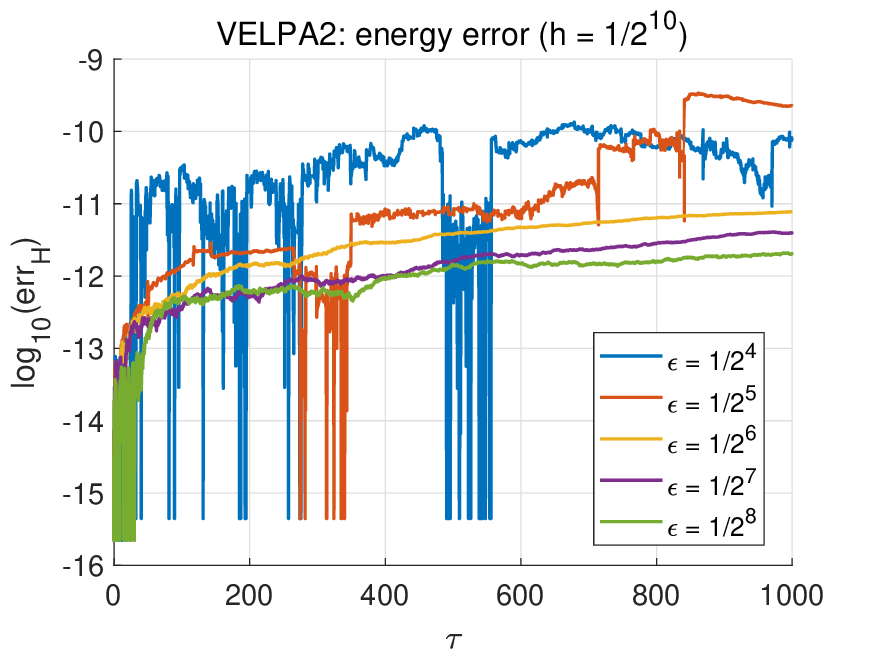}}
	\caption{Example \ref{p3}. The energy errors of SS2-xn (left panels) and VELPA2 (right panels).}
	\label{E_3}
\end{figure}

\section{Conclusion}\label{sec5}
For the four-dimensional relativistic dynamical system in the presence of strong magnetic fields, the accuracy of standard numerical methods degrades significantly as the magnetic field strength increases.
This paper develops a novel multi-physics structure-preserving algorithm with error bounds that are independent of strong magnetic fields. The proposed scheme is fully explicit and structure-preserving, achieving second-order uniform error bounds. Comparative numerical experiments demonstrate that it outperforms conventional second-order splitting schemes in terms of accuracy. Future work will focus on extending the theoretical analysis to relax the restrictions imposed by the magnetic field assumptions and step-size constraints, thereby establishing uniform accuracy under more general settings.

\appendix
\renewcommand{\thelemma}{A.\arabic{lemma}}
\section{} \label{appe}

\begin{proof}[\textbf{Proof of Lemma \ref{approximate}}]

Based on the definitions of $\zeta^n_{\by}(s)$ and $\zeta^n_{\bu}(s)$ (see \eqref{Z1}-\eqref{Z2}), we subtract the truncated system \eqref{ts} from the scaled long-time system \eqref{Rcpd2s} and derive the resulting error equations:
\begin{subequations}
	\begin{align}
		\dot{\zeta}^n_{\bx}(s) = & \,\, \varepsilon \zeta^n_{\bv}(s), \label{za} \\ \displaystyle
		\dot{\zeta}^n_{\bar{t}}(s) = & \,\, \varepsilon \zeta^n_{w}(s), \label{zb} \\ \displaystyle 
		\dot{\zeta}^n_{\bv}(s) = & \,\, \bv(\tau_{n}+s) \times \bB(\varepsilon \bx(\tau_{n}+s))
		- \widetilde{\bv}^{\, n}(s) \times \big[\bB_{mid} + \varepsilon \nabla \bB_{mid} \cdot \big( \widetilde{\bx}^{\, n}(s)-\bx_{mid} \big) \big]  \nonumber \\ 
		& + i\varepsilon \widetilde{w}^{\, n}(s) \big( \bE(\widetilde{\bx}^{\, n}(s)) -\bE(\bx(\tau_n + s)) \big)
		- i\eps \zeta^n_{w}(s) \bE(\bx(\tau_n + s)), \label{zc} \\
		\dot{\zeta}^n_{w}(s) = & \,\, i\eps \big( \bE(\bx(\tau_n + s))^{\intercal} - \bE(\widetilde{\bx}^{\, n}(s))^{\intercal} \big) \cdot \bv(\tau_n + s) + i\eps \bE(\widetilde{\bx}^{\, n}(s))^{\intercal} \cdot \zeta^n_{\bv}(s).  \label{zd}
	\end{align}
\end{subequations}
Since $\bB \in C^2(\mathbb{R}^3)$, we expand the magnetic field $\bB(\varepsilon \bx(\tau_{n}+s))$ in a Taylor series at $\bx_{mid}$ to obtain
\begin{equation}
	\bB(\varepsilon \bx(\tau_{n}+s)) = \bB_{mid} + \varepsilon \nabla \bB_{mid} \cdot \big( \bx(\tau_{n}+s)-\bx_{mid} \big) + \boldsymbol{R}_{2}(s). \label{bB}
\end{equation}
Combined with the condition $\left\| \bB(0) \right\| \lesssim \varepsilon$, we arrive at the estimates 
\begin{equation}
	\left\| \nabla \bB_{mid} \right\| = \left\| \nabla \bB(0) + \varepsilon \nabla^2 \bB(0) \cdot \bx_{mid} + \mathcal{O}(\varepsilon^2) \right\| \lesssim \varepsilon,
	\label{nablaB}
\end{equation}
and 
\begin{equation}
	\left\| \boldsymbol{R}_{2}(s) \right\| \lesssim \varepsilon^2 \left\|\bx(\tau_{n}+s)-\bx_{mid}\right\|^2 
	\lesssim \varepsilon^4 \left|s-\frac{\mathfrak{h}}{2}\right|^2. \label{R2}
\end{equation}
Furthermore, as $\bE \in C^2(\mathbb{R}^3)$, we apply the integral representation of the difference using the definition of $\zeta^n_{\bx}(s)$ and get 
\begin{equation}
	\bE(\bx(\tau_n + s)) - \bE(\widetilde{\bx}^{\, n}(s)) = 
	\int_{0}^{1} \nabla \bE\bigl(\bx(\tau_{n}+s) + (\rho-1)\zeta_{\bx}^{n}(s)\bigr) \, d\rho \cdot \zeta_{\bx}^{n}(s). \label{bE}
\end{equation}
Here $\nabla \bE$ stands for the gradient of the electric field $\bE$. Substituting \eqref{bB} and \eqref{bE} into \eqref{zc}-\eqref{zd}, we reformulate the error equations as follows:
\begin{equation}
	\begin{cases}
		\dot{\zeta}^n_{\bx}(s) = \varepsilon \zeta^n_{\bv}(s),  \displaystyle \quad 0 < s \leq \mathfrak{h}, \quad 0 \leq n < \frac{T}{\varepsilon \mathfrak{h}}, \\ \displaystyle
		\dot{\zeta}^n_{\bar{t}}(s) = \varepsilon \zeta^n_{w}(s), 
		\quad \zeta^n_{\bx}(0) = \zeta^n_{\bar{t}}(0) = 0, \\[2pt] 
		\displaystyle 
		\dot{\zeta}^n_{\bv}(s) = 
		\begin{aligned}[t]
			& \zeta^n_{\bv}(s) \times \bigl( \bB_{mid} + \Delta \bB(s) \bigr) + \widetilde{\bv}^{\, n}(s) \times \bigl( \varepsilon \nabla \bB_{mid} \cdot \zeta^n_{\bx}(s) \bigr) \displaystyle \\
			& + \bv(\tau_{n}+s) \times \boldsymbol{R}_{2}(s)
			- i\varepsilon \left[ \widetilde{w}^{\, n}(s) \big( \overline{\nabla \bE}(s) \cdot \zeta^n_{\bx}(s) \big) + \zeta^n_{w}(s) \bE(\bx(\tau_n + s)) \right], \\[2pt]
		\end{aligned} \\
		\dot{\zeta}^n_{w}(s) = i\eps \left[ \big( \overline{\nabla \bE}(s) \cdot \zeta^n_{\bx}(s) \big)^{\intercal} \bv(\tau_{n}+s) + \bE(\widetilde{\bx}^{\, n}(s))^{\intercal} \cdot \zeta^n_{\bv}(s) \right],
		\quad \zeta^n_{\bv}(0) = \zeta^n_{w}(0) = 0,\displaystyle
	\end{cases}
	\label{zeta}
\end{equation}
where $\Delta \bB(s) := \varepsilon \nabla \bB_{mid} \cdot 
\big( \bx(\tau_{n}+s)-\bx_{mid} \big) $ and 
\begin{equation*}
	\overline{\nabla \bE}(s) := \int_{0}^{1} \nabla \bE\bigl(\bx(\tau_{n}+s) + (\rho-1)\zeta_{\bx}^{n}(s)\bigr) \, d\rho.
\end{equation*}

Let $\Phi_{\bv}(\mh, s)$ denote the evolution operator associated with the homogeneous equation $\dot{\zeta}^n_{\bv}(s) = \boldsymbol{A}(s) \zeta^n_{\bv}(s)$, where $\boldsymbol{A}(s) := \widehat{\bB}_{mid} + \widehat{\Delta \bB(s)}$ is a skew-symmetric matrix. Applying Duhamel's principle to \eqref{zeta}, we arrive at
\begin{subequations}
	\begin{align}
		\zeta_{\bx}^{n}(\mh) = &\,\, \varepsilon \int_{0}^{\mh} \zeta_{\bv}^{n}(s) \, ds, \label{zea} \\
		\zeta_{\bar{t}}^{n}(\mh) = &\,\, \varepsilon \int_{0}^{\mh} \zeta_{w}^{n}(s) \, ds,  \label{zeb} \\
		\zeta_{\bv}^{n}(\mh) = & \int_{0}^{\mh} \Phi_{\bv}(\mh,s) f_{\bv}(s) \, ds, \label{zec} \\
		\zeta_{w}^{n}(\mh) = & \int_{0}^{\mh} i\eps \left[ \big( \overline{\nabla \bE}(s) \cdot \zeta^n_{\bx}(s) \big)^{\intercal} \bv(\tau_{n}+s) + \bE(\widetilde{\bx}^{\, n}(s))^{\intercal} \cdot \zeta^n_{\bv}(s) \right]
		\, ds. \label{zed}
	\end{align}
\end{subequations}
Here we take $0 < s \leq \mathfrak{h}$ and $0 \leq n < \dfrac{T}{\varepsilon \mathfrak{h}}$, and
\begin{equation*}
	f_{\bv}(s) = \widetilde{\bv}^{\, n}(s) \times \bigl( \varepsilon \nabla \bB_{mid} \cdot \zeta^n_{\bx}(s) \bigr) + \bv(\tau_{n}+s) \times \boldsymbol{R}_{2}(s) 
	- i\varepsilon \left[ \widetilde{w}^{\, n}(s) \big( \overline{\nabla \bE}(s) \cdot \zeta^n_{\bx}(s) \big) + \zeta^n_{w}(s) \bE(\bx(\tau_n + s)) \right].
\end{equation*}
Define the maximum norms $M_{\bx}=\sup_{0 \leq s \leq \mh}\left\|\zeta_{\bx}^{n}(s)\right\|$, $M_{\bar{t}}=\sup_{0 \leq s \leq \mh}\left\|\zeta_{\bar{t}}^{n}(s)\right\|$, $M_{\bv}=\sup_{0 \leq s \leq \mh}\left\|\zeta_{\bv}^{n}(s)\right\|$, and $M_{w}=\sup_{0 \leq s \leq \mh}\left\|\zeta_{w}^{n}(s)\right\|$. 
Taking norms on both sides of \eqref{zea} gives
\begin{equation*}
	\left\|\zeta_{\bx}^{n}(\mh)\right\| \lesssim \varepsilon \int_{0}^{\mh} \left\|\zeta_{\bv}^{n}(s)\right\| \, ds 
	\lesssim \eps \mh M_{\bv},
\end{equation*}
which implies
\begin{equation}
	M_{\bx} \lesssim \eps \mh M_{\bv}. \label{mx}
\end{equation}
By the boundedness of $\bE$ and $\bv$, we deduce from \eqref{zed} that
\begin{equation*}
	\left\|\zeta_{w}^{n}(\mh)\right\| \lesssim \varepsilon \int_{0}^{\mh} \big( \left\|\zeta_{\bx}^{n}(s)\right\| + \left\|\zeta_{\bv}^{n}(s)\right\| \big) \, ds 
	\lesssim \eps \mh \big(  M_{\bx} + M_{\bv} \big).
\end{equation*}
Substituting \eqref{mx} into the above inequality yields the estimate
\begin{equation}
	M_{w} \lesssim \eps \mh M_{\bv} + \eps^2 \mh^2 M_{\bv} \lesssim \eps \mh M_{\bv}. \label{mw}
\end{equation}
Similarly, we take the norm of \eqref{zeb} and combine it with \eqref{mw} to obtain
\begin{equation}
	M_{\bar{t}} \lesssim \eps \mh M_{w} \lesssim \eps^2 \mh^2 M_{\bv}. \label{mt}
\end{equation}
It remains to estimate $M_{\bv}$. We now take the norm of \eqref{zec} and apply the triangle inequality together with \eqref{R2} and \eqref{mx}-\eqref{mw}, leading to
\begin{align*}
	\left\|\zeta_{\bv}^{n}(\mh)\right\| \lesssim & \int_{0}^{\mh} \left\|f_{\bv}(s)\right\| \, ds 
	\lesssim \int_{0}^{\mh} \big( \eps^2 \left\|\zeta_{\bx}^{n}(s)\right\| 
	+ \eps \left\|\zeta_{\bx}^{n}(s)\right\| + \eps \left\|\zeta_{w}^{n}(s)\right\| + \left\| \boldsymbol{R}_{2}(s) \right\| \big) \\
	\lesssim &\,\, \eps \mh M_{\bx} + \eps \mh M_{w} + \eps^4 \int_{0}^{\mh} |s-\frac{\mathfrak{h}}{2}|^2 \, ds 
	\lesssim \eps^2 \mh^2 M_{\bv} + \eps^4 \mh^3.
\end{align*}
Consequently,
$
M_{\bv} \lesssim \varepsilon^2 \mh^2 M_{\bv} + \varepsilon^4 \mh^3.
$
For sufficiently small step size satisfying $\varepsilon^2 \mh^2 \leq 1/2$, we apply the absorption argument to derive
$
M_{\bv} \lesssim \varepsilon^4 \mh^3.
$
Substituting this bound back, we obtain the following error estimates: 
\begin{equation*}
	\left\|\zeta_{\bx}^{n}(\mh)\right\| \lesssim \eps^5 \mh^4, 
	\quad \left\|\zeta_{\bar{t}}^{n}(\mh)\right\| \lesssim \eps^6 \mh^5,
	\quad \left\|\zeta_{\bv}^{n}(\mh)\right\| \lesssim \eps^4 \mh^3,
	\quad \left\|\zeta_{w}^{n}(\mh)\right\| \lesssim \eps^5 \mh^4.
\end{equation*}
The proof of Lemma \ref{approximate} is complete.
\end{proof}

\begin{proof}[\textbf{Proof of Lemma \ref{lemmalocal}}]
	
According to the definition, we derive the local truncation error equations of the SS2-xn scheme \eqref{SS2s} for solving the truncated system as follows:

\begin{subequations}
	\begin{align}
		\widetilde{\by}^{\, n}(\mathfrak{h}) = & \,\, \by(\tau_{n})+\frac{\varepsilon \mh}{2}\varphi_{1}\biggl( \frac{\mh}{2} \bK_{\bx(\tau_{n})} \biggr)
		\Bigl(I_4 + \mathrm{e}^{\mh (\bK_{\bar{\bx}(\tau_{n})} - \bK_{\bx(\tau_{n})}) }
		\mathrm{e}^{\frac{\mh}{2}\bK_{\bx(\tau_{n})}}\Bigr) \bu(\tau_{n}) + \xi_{\by}^{n},  \label{bby} \\
		\widetilde{\bu}^{\, n}(\mathfrak{h}) = & \,\,  \mathrm{e}^{\frac{\mh}{2}\bK_{\bx(\tau_{n})}}
		\mathrm{e}^{\mh (\bK_{\bar{\bx}(\tau_{n})} - \bK_{\bx(\tau_{n})})}
		\mathrm{e}^{\frac{\mh}{2}\bK_{\bx(\tau_{n})}} \bu(\tau_{n}) + \xi_{\bu}^{n},
		\quad 0 \leq n < \frac{T}{\varepsilon \mh},  \label{bbu}
	\end{align}
\end{subequations}
where 
\begin{align*}
	\bar{\by}(\tau_{n})=\by(\tau_{n})+\frac{\varepsilon\mh}{2} \varphi_{1} \biggl( \frac{\mh}{2}\bK_{\bx(\tau_{n})} \biggr) \bu(\tau_{n}), \quad \bar{\bx}(\tau_{n}) = \bar{\by}(\tau_{n})(1:3), 
	\quad \bK_{\bx(\tau_{n})} = \bK(\bx(\tau_{n})),
	\quad \bK_{\bar{\bx}(\tau_{n})} = \bK(\bar{\bx}(\tau_{n})).
\end{align*}

Set $\delta \bB(\tau) = \varepsilon \nabla \bB_{mid} \cdot \big( \widetilde{\bx}^{\, n}(\tau)-\bx_{mid} \big)$. Then $\widetilde{\boldsymbol{A}}(\tau) = \widehat{\bB}_{mid} + \widehat{\delta \bB(\tau)}$ is a skew-symmetric matrix. Combining \eqref{nablaB} and the estimate 
\begin{align*}
	\left\| \widetilde{\bx}^{\, n}(\tau)-\bx_{mid} \right\| & =  
	\left\| \big( \widetilde{\bx}^{\, n}(\tau)-\bx(\tau_{n}+\tau) \big)
	+ \big( \bx(\tau_{n}+\tau)-\bx_{mid} \big) \right\|   
    = \left\| -\zeta_{\bx}^{n}(\tau) + \big( \bx(\tau_{n}+\tau)-\bx_{mid} \big) \right\|  \\
	& \lesssim \left\| \big( \bx_{mid} + \dot{\bx}_{mid}\cdot(\tau-\frac{\mathfrak{h}}{2})
	+ \mathcal{O}(\varepsilon (\tau-\frac{\mathfrak{h}}{2})^2 ) - \bx_{mid} \big) \right\| + \mathcal{O}(\varepsilon^5 \tau^4)   
    \lesssim \varepsilon \left|\tau-\frac{\mathfrak{h}}{2}\right| + \mathcal{O}(\varepsilon (\tau-\frac{\mathfrak{h}}{2})^2),
\end{align*}
we further deduce 
\begin{equation}
	\left\| \widehat{\delta \bB(\tau)} \right\| 
	\lesssim \varepsilon^3 \left|\tau-\frac{\mathfrak{h}}{2}\right| + \mathcal{O}(\varepsilon^3 (\tau-\frac{\mathfrak{h}}{2})^2).
	\label{deltaa}
\end{equation} 
Let $\widetilde{\Phi}_{\bv}(s,\sigma)$ denote the evolution operator of the homogeneous equation $\dot{\widetilde{\bv}} = \widetilde{\boldsymbol{A}}(\tau)\widetilde{\bv}$. By the variation of constants formula, we obtain
\begin{equation}
	\widetilde{\Phi}_{\bv}(s,\sigma) = \mathrm{e}^{(s-\sigma)\widehat{\bB}_{mid}} + \int_\sigma^s \mathrm{e}^{(s-\tau)\widehat{\bB}_{mid}} \widehat{\delta \bB(\tau)} \widetilde{\Phi}_{\bv}(\tau, \sigma) \, d\tau. 
	\label{Phi}
\end{equation}
Combining with \eqref{Phi}, applying Duhamel's principle to the truncated system \eqref{ts} yields the exact solution as
\begin{subequations}
	\begin{align}
		\widetilde{\bx}^{\, n}(\mh) = &\,\, \bx(\tau_{n}) + \eps \int_{0}^{\mh} \widetilde{\bv}^{\, n}(s) \, ds, \label{bxa} \\
		{\widetilde{\bar{t}}}^{\,\, n}(\mh)  = &\,\, \bar{t}(\tau_{n}) + \eps \int_{0}^{\mh} \widetilde{w}^{\, n}(s) \, ds, \label{btb} \\
		\widetilde{\bv}^{\, n}(\mh) = &\,\, \widetilde{\Phi}_{\bv}(\mh,0) \bv(\tau_{n}) - i\eps \int_{0}^{\mh} \widetilde{\Phi}_{\bv}(\mh,s) \widetilde{w}^{\, n}(s)\bE(\widetilde{\bx}^{\, n}(s)) \, ds  
		=  \mathrm{e}^{\mathfrak{h} \widehat{\bB}_{mid}} \bv(\tau_{n}) 
		- i\varepsilon \int_{0}^{\mathfrak{h}} \mathrm{e}^{(\mathfrak{h}-s) \widehat{\bB}_{mid}} \widetilde{w}^{\, n}(s)\bE(\widetilde{\bx}^{\, n}(s)) \, ds  \notag \\
		& + \int_{0}^{\mathfrak{h}} 
		\mathrm{e}^{(\mathfrak{h}-\tau)\widehat{\bB}_{mid}} 
		\widehat{\delta \bB(\tau)}
		\widetilde{\Phi}_{\bv}(\tau,0) \bv(\tau_{n}) \, d\tau 
		- i\varepsilon \int_{0}^{\mathfrak{h}} \int_{s}^{\mathfrak{h}}
		\mathrm{e}^{(\mathfrak{h}-\tau)\widehat{\bB}_{mid}} \widehat{\delta \bB(\tau)}
		\widetilde{\Phi}_{\bv}(\tau,s) \widetilde{w}^{\, n}(s)\bE(\widetilde{\bx}^{\, n}(s)) \, d\tau \, ds, \label{bvc} \\
		\widetilde{w}^{\, n}(\mh) = &\,\, w(\tau_{n}) + i\eps \int^{\mh}_{0}
		\bE(\widetilde{\bx}^{\, n}(s))^{\intercal} \cdot \widetilde{\bv}^{\, n}(s) \, ds, 
		\quad 0 \leq n < \frac{T}{\varepsilon \mathfrak{h}}. \label{bwd}
	\end{align}
\end{subequations}

$\bullet$   \textbf{The estimation of $\xi_{\bu}^{n}$.} 
From the boundedness of the electric field $\bE$, magnetic field $\bB$, evolution operator $\widetilde{\Phi}_{\bv}$, exact solutions $\bv$, $\widetilde{w}$, together with the bound \eqref{deltaa}, we deduce
\begin{align*}
	& \left\| \int_{0}^{\mathfrak{h}} 
	\mathrm{e}^{(\mathfrak{h}-\tau)\widehat{\bB}_{mid}} 
	\widehat{\delta \bB(\tau)}
	\widetilde{\Phi}_{\bv}(\tau,0) \bv(\tau_{n}) \, d\tau \right\|
	\lesssim  \int_{0}^{\mathfrak{h}} \left\| \widehat{\delta \bB(\tau)} \right\| \, d\tau   
	\lesssim \eps^3 \mh^3, \\
	& \left\|  - i\varepsilon \int_{0}^{\mathfrak{h}} \int_{s}^{\mathfrak{h}}
	\mathrm{e}^{(\mathfrak{h}-\tau)\widehat{\bB}_{mid}} \widehat{\delta \bB(\tau)}
	\widetilde{\Phi}_{\bv}(\tau,s) \widetilde{w}^{\, n}(s)\bE(\widetilde{\bx}^{\, n}(s)) \, d\tau \, ds \right\| 
	\lesssim \eps \int_{0}^{\mathfrak{h}} \int_{s}^{\mathfrak{h}} \left\| \widehat{\delta \bB(\tau)} \right\| \, d\tau \, ds
	\lesssim \eps^4 \mh^3.
\end{align*}
Substituting the expression of $\widetilde{w}^{\, n}(s)$ from \eqref{bwd} into \eqref{bvc} and inserting the preceding bounds yields
\begin{align*}
	\widetilde{\bv}^{\, n}(\mh) = &\,\, \mathrm{e}^{\mathfrak{h} \widehat{\bB}_{mid}} \bv(\tau_{n}) 
	- i\varepsilon \int_{0}^{\mathfrak{h}} \mathrm{e}^{(\mathfrak{h}-s) \widehat{\bB}_{mid}} w(\tau_{n}) \bE(\widetilde{\bx}^{\, n}(s)) \, ds 
	+ \eps^2 \int_{0}^{\mathfrak{h}} \mathrm{e}^{(\mathfrak{h}-s) \widehat{\bB}_{mid}} \int^{s}_{0}
	\bE(\widetilde{\bx}^{\, n}(\sigma))^{\intercal} \cdot \widetilde{\bv}^{\, n}(\sigma) \, d\sigma \, \bE(\widetilde{\bx}^{\, n}(s)) \, ds
	+ \mathcal{O}(\eps^3 \mh^3) \\
	= &\,\, \mathrm{e}^{\mathfrak{h} \widehat{\bB}_{mid}} \bv(\tau_{n}) 
	- i\varepsilon \int_{0}^{\mathfrak{h}} \mathrm{e}^{(\mathfrak{h}-s) \widehat{\bB}_{mid}} w(\tau_{n}) \bE(\widetilde{\bx}^{\, n}(s)) \, ds \\
	& + \eps^2 \int_{0}^{\mathfrak{h}} \mathrm{e}^{(\mathfrak{h}-s) \widehat{\bB}_{mid}} \int^{s}_{0}
	\bE(\widetilde{\bx}^{\, n}(\sigma))^{\intercal} \cdot \big( \mathrm{e}^{\sigma \widehat{\bB}_{mid}} \bv(\tau_{n}) \big) \, d\sigma \, \bE(\widetilde{\bx}^{\, n}(s)) \, ds
	+ \mathcal{O}(\eps^3 \mh^3). 
\end{align*}
To improve the order of accuracy for the truncation error, we reformulate the first integral appearing on the right-hand side of the above identity. Setting $s = (1-\rho)\mathfrak{h}$ and performing a change of variables in this definite integral leads to
\begin{equation*}
	- i\varepsilon \int_{0}^{\mathfrak{h}} \mathrm{e}^{(\mathfrak{h}-s) \widehat{\bB}_{mid}} w(\tau_{n}) \bE(\widetilde{\bx}^{\, n}(s)) \, ds
	= - i\varepsilon \mathfrak{h} \int_{0}^{1} \mathrm{e}^{\rho \mathfrak{h} \widehat{\bB}_{mid}} w(\tau_{n}) \bE(\widetilde{\bx}^{\, n}((1-\rho)\mathfrak{h})) \, d\rho.
\end{equation*}
Making use of the definition of the error $\zeta_{\bx}^{n}$ together with Taylor series expansion, we expand and reconstruct the electric field as
\begin{align}
	& \bE\big( \widetilde{\bx}^{\, n}((1-\rho)\mathfrak{h}) \big) = \bE\big( \bx(\tau_{n} + (1-\rho)\mathfrak{h}) - \zeta_{\bx}^{n}((1-\rho)\mathfrak{h}) \big) 
	= \bE \Big( \bx(\tau_{n} + \frac{\mathfrak{h}}{2}) + (\frac{1}{2} - \rho)\mathfrak{h} \dot{\bx}(\tau_{\rho}^{n}) - \zeta_{\bx}^{n}((1-\rho)\mathfrak{h}) \Big) \notag \\
	&\quad = \bE\Big(\bx(\tau_{n} + \frac{\mathfrak{h}}{2}) \Big) + \int_{0}^{1}\nabla \bE(s_{\sigma}) \, d\sigma \Big[ (\frac{1}{2} - \rho)\mathfrak{h} \varepsilon \bv(\tau_{\rho}^{n}) - \zeta_{\bx}^{n}((1-\rho) \mathfrak{h}) \Big],  \label{Ewx}
\end{align}
where $\tau_{\rho}^{n} \in \bigl( \min(\tau_{n}+\tfrac{\mathfrak{h}}{2},\tau_{n}+(1-\rho)\mathfrak{h}),\; \max(\tau_{n}+\tfrac{\mathfrak{h}}{2},\tau_{n}+(1-\rho)\mathfrak{h}) \bigr)$, and the intermediate argument $s_\sigma$ is defined by
\begin{equation*}
	s_{\sigma} = \bx(\tau_{n} + \frac{\mathfrak{h}}{2}) + \sigma \Big[ (\frac{1}{2} - \rho)\mathfrak{h} \varepsilon \bv(\tau_{\rho}^{n}) - \zeta_{\bx}^{n}((1-\rho) \mathfrak{h}) \Big].
\end{equation*}
Substituting the reconstructed electric-field expression back into the original integral yields
\begin{align*}
	& - i\varepsilon \int_{0}^{\mathfrak{h}} \mathrm{e}^{(\mathfrak{h}-s) \widehat{\bB}_{mid}} w(\tau_{n}) \bE(\widetilde{\bx}^{\, n}(s)) \, ds 
	= - i\varepsilon \mathfrak{h} \int_{0}^{1} \mathrm{e}^{\rho \mathfrak{h} \widehat{\bB}_{mid}} w(\tau_{n}) \bE\Big(\bx(\tau_{n} + \frac{\mathfrak{h}}{2}) \Big) \, d\rho \\
	&\quad\hspace{4mm} - i\varepsilon \mathfrak{h} \int_{0}^{1} \mathrm{e}^{\rho \mathfrak{h} \widehat{\bB}_{mid}} w(\tau_{n}) \int_{0}^{1}\nabla \bE(s_{\sigma}) \, d\sigma \Big[ (\frac{1}{2} - \rho)\mathfrak{h} \varepsilon \bv(\tau_{\rho}^{n}) - \zeta_{\bx}^{n}((1-\rho) \mathfrak{h}) \Big] \, d\rho \\
	&\quad = - i\varepsilon \mathfrak{h} \int_{0}^{1} \mathrm{e}^{\rho \mathfrak{h} \widehat{\bB}_{mid}} w(\tau_{n}) \bE\Big(\bx(\tau_{n} + \frac{\mathfrak{h}}{2}) \Big) \, d\rho + \mathcal{O}(\eps^2 \mh^3).
\end{align*}
Accordingly, we arrive at the finalized expression 
\begin{align}
	\widetilde{\bv}^{\, n}(\mh) = &\,\, \mathrm{e}^{\mathfrak{h} \widehat{\bB}_{mid}} \bv(\tau_{n}) 
	- i\varepsilon \mathfrak{h} \int_{0}^{1} \mathrm{e}^{\rho \mathfrak{h} \widehat{\bB}_{mid}} w(\tau_{n}) \bE\Big(\bx(\tau_{n} + \frac{\mathfrak{h}}{2}) \Big) \, d\rho \notag\\
	& + \eps^2 \int_{0}^{\mathfrak{h}} \mathrm{e}^{(\mathfrak{h}-s) \widehat{\bB}_{mid}} \int^{s}_{0}
	\bE(\widetilde{\bx}^{\, n}(\sigma))^{\intercal} \cdot \big( \mathrm{e}^{\sigma \widehat{\bB}_{mid}} \bv(\tau_{n}) \big) \, d\sigma \, \bE(\widetilde{\bx}^{\, n}(s)) \, ds
	+ \mathcal{O}(\eps^2 \mh^3). \label{bv}
\end{align}
Inserting \eqref{bv} into \eqref{bwd} gives
\begin{align*}
	\widetilde{w}^{\, n}(\mh) = &\,\, w(\tau_{n}) + i\eps \int^{\mh}_{0}
	\bE(\widetilde{\bx}^{\, n}(s))^{\intercal} \cdot \big( \mathrm{e}^{s \widehat{\bB}_{mid}} \bv(\tau_{n}) \big) \, ds 
	+ \eps^2 \int^{\mh}_{0}
	\bE(\widetilde{\bx}^{\, n}(s))^{\intercal} \cdot s \int_{0}^{1} \mathrm{e}^{\rho \mathfrak{h} \widehat{\bB}_{mid}} w(\tau_{n}) \bE\Big(\bx(\tau_{n} + \frac{\mathfrak{h}}{2}) \Big) \, d\rho \, ds
	+ \mathcal{O}(\eps^3 \mh^3).
\end{align*}
Likewise, we perform the change of variables $s = (1-\rho)\mathfrak{h}$ for the first integral on the right-hand side and apply identity \eqref{Ewx} to obtain
\begin{align}
	\widetilde{w}^{\, n}(\mh) = &\,\, w(\tau_{n}) + i\eps\mh \int^{1}_{0}
	\bE\Big(\bx(\tau_{n} + \frac{\mathfrak{h}}{2}) \Big)^{\intercal} \cdot \big( \mathrm{e}^{(1-\rho)\mh \widehat{\bB}_{mid}} \bv(\tau_{n}) \big) \, d\rho \notag \\
	& + \eps^2 \int^{\mh}_{0}
	\bE(\widetilde{\bx}^{\, n}(s))^{\intercal} \cdot s \int_{0}^{1} \mathrm{e}^{\rho \mathfrak{h} \widehat{\bB}_{mid}} w(\tau_{n}) \bE\Big(\bx(\tau_{n} + \frac{\mathfrak{h}}{2}) \Big) \, d\rho \, ds
	+ \mathcal{O}(\eps^2 \mh^3).  \label{bw}
\end{align}
Subtracting \eqref{bbu} from the combination of \eqref{bv} and \eqref{bw}, we derive the following representation for the local truncation error $\xi_{\bu}^{n}$:
\begin{align}
	\xi_{\bu}^{n} = & 
	\begin{pmatrix}
		\begin{aligned}
			\displaystyle 
			&\hspace{8mm} \mathrm{e}^{\mathfrak{h} \widehat{\bB}_{mid}} \bv(\tau_{n}) 
			- i\varepsilon \mathfrak{h} \int_{0}^{1} \mathrm{e}^{\rho \mathfrak{h} \widehat{\bB}_{mid}} w(\tau_{n}) \bE\Big(\bx(\tau_{n} + \frac{\mathfrak{h}}{2}) \Big) \, d\rho \\
			\displaystyle
			& + \eps^2 \int_{0}^{\mathfrak{h}} \mathrm{e}^{(\mathfrak{h}-s) \widehat{\bB}_{mid}} \int^{s}_{0}
			\bE(\widetilde{\bx}^{\, n}(\sigma))^{\intercal} \cdot \big( \mathrm{e}^{\sigma \widehat{\bB}_{mid}} \bv(\tau_{n}) \big) \, d\sigma \, \bE(\widetilde{\bx}^{\, n}(s)) \, ds  
		\end{aligned} \\[27pt]
		\begin{aligned}
			\displaystyle
			&\hspace{3mm} w(\tau_{n}) + i\eps\mh \int^{1}_{0}
			\bE\Big(\bx(\tau_{n} + \frac{\mathfrak{h}}{2}) \Big)^{\intercal} \cdot \big( \mathrm{e}^{(1-\rho)\mh \widehat{\bB}_{mid}} \bv(\tau_{n}) \big) \, d\rho \\
			\displaystyle
			& + \eps^2 \int^{\mh}_{0}
			\bE(\widetilde{\bx}^{\, n}(s))^{\intercal} \cdot s \int_{0}^{1} \mathrm{e}^{\rho \mathfrak{h} \widehat{\bB}_{mid}} w(\tau_{n}) \bE\Big(\bx(\tau_{n} + \frac{\mathfrak{h}}{2}) \Big) \, d\rho \, ds
		\end{aligned}
	\end{pmatrix}  
	\displaystyle  - \mathrm{e}^{\frac{\mh}{2}\bK_{\bx(\tau_{n})}}
	\mathrm{e}^{\mh (\bK_{\bar{\bx}(\tau_{n})} - \bK_{\bx(\tau_{n})})}
	\mathrm{e}^{\frac{\mh}{2}\bK_{\bx(\tau_{n})}} \bu(\tau_{n})
	+ \mathcal{O}(\eps^2 \mh^3). \label{xiu}
\end{align}
From its definition, the matrix 
$
\bK = \widehat{\widehat{\bB}} + \eps \widehat{\bE} 
$
is skew-symmetric. Taking the norm of both sides of \eqref{xiu} and truncating the Taylor expansion up to the \(\mh^2\) term yields 
\begin{align}
	\left\| \xi_{\bu}^{n} \right\| \lesssim &
	\left\|
	\begin{pmatrix}
		\begin{aligned}
			\displaystyle
			&\hspace{10mm} \Big[ \mh \big( \widehat{\bB}_{mid} - \widehat{\bB}_{\bar{\bx}(\tau_{n})} \big) 
			+ \mh^2 \big( \widehat{\bB}^{2}_{mid} - \widehat{\bB}^{2}_{\bar{\bx}(\tau_{n})} \big) \\
			\displaystyle
			& + \eps^2\mh^2 \big( \bE(\widetilde{\bx}^{\, n} (\sigma))^{\intercal} \bE(\widetilde{\bx}^{\, n}(s)) - \bE(\bar{\bx}(\tau_{n}))^{\intercal} \bE(\bar{\bx}(\tau_{n})) \big)  \Big] \bv(\tau_{n}) \\
			\displaystyle
			&\hspace{12mm} + \Big[ i\eps\mh \big( \bE(\bar{\bx}(\tau_{n})) - \bE\big(\bx(\tau_{n} + \frac{\mathfrak{h}}{2})\big) \big) \\
			\displaystyle
			&\hspace{3mm} + i\eps\mh^2 \big( \widehat{\bB}_{\bar{\bx}(\tau_{n})} \bE(\bar{\bx}(\tau_{n})) - \widehat{\bB}_{mid} \bE\big(\bx(\tau_{n} + \frac{\mathfrak{h}}{2})\big) \big) \Big] w(\tau_{n}) 
		\end{aligned}\\[39pt]
		\begin{aligned}
			\displaystyle
			&\hspace{16mm} \Big[ i\eps\mh \big( \bE\big(\bx(\tau_{n} + \frac{\mathfrak{h}}{2}) \big)^{\intercal} - \bE(\bar{\bx}(\tau_{n}))^{\intercal} \big) \\
			\displaystyle
			&\hspace{2mm} + i\eps\mh^2 \big( \bE\big(\bx(\tau_{n} + \frac{\mathfrak{h}}{2}) \big)^{\intercal} \widehat{\bB}_{mid} - \bE(\bar{\bx}(\tau_{n}))^{\intercal} \widehat{\bB}_{\bar{\bx}(\tau_{n})} \big)
			\Big] \bv(\tau_{n}) \\
			\displaystyle
			& + \eps^2\mh^2 \Big[   \bE(\widetilde{\bx}^{\, n}(s))^{\intercal}
			\bE\big(\bx(\tau_{n} + \frac{\mathfrak{h}}{2})\big) - \bE(\bar{\bx}(\tau_{n}))^{\intercal} \bE(\bar{\bx}(\tau_{n}))  \Big] w(\tau_{n})
		\end{aligned}
	\end{pmatrix}
	\right\| + \mathcal{O}(\eps^2 \mh^3) \notag \\
	\lesssim & 
	\left\|
	\begin{pmatrix}
		\begin{aligned}
			\displaystyle
			&\hspace{22mm} \mh \left\| \widehat{\bB}_{mid} - \widehat{\bB}_{\bar{\bx}(\tau_{n})} \right\| + \mh^2 \left\| \big( \widehat{\bB}_{mid} - \widehat{\bB}_{\bar{\bx}(\tau_{n})} \big) \widehat{\bB}_{mid} \right\| \\
			\displaystyle
			&\hspace{4mm} + \mh^2 \left\| \widehat{\bB}_{\bar{\bx}(\tau_{n})} \big( \widehat{\bB}_{mid} - \widehat{\bB}_{\bar{\bx}(\tau_{n})} \big) \right\|
			+ \eps^2\mh^2 \left\| \bE(\widetilde{\bx}^{\, n} (\sigma))^{\intercal} \big( \bE(\widetilde{\bx}^{\, n}(s)) -  \bE(\bar{\bx}(\tau_{n})) \big) \right\| \\
			\displaystyle
			& + \eps^2\mh^2 \left\| \big( \bE(\widetilde{\bx}^{\, n} (\sigma))^{\intercal} - \bE(\bar{\bx}(\tau_{n}))^{\intercal} \big)  \bE(\bar{\bx}(\tau_{n}))  \right\|
			+ \eps\mh \left\| \bE(\bar{\bx}(\tau_{n})) - \bE\big(\bx(\tau_{n} + \frac{\mathfrak{h}}{2})\big) \right\| \\
			\displaystyle
			& + \eps\mh^2 \left\| \widehat{\bB}_{\bar{\bx}(\tau_{n})} \big(
			\bE(\bar{\bx}(\tau_{n})) - \bE\big(\bx(\tau_{n} + \frac{\mathfrak{h}}{2})\big) \big) \right\| 
			+ \eps\mh^2 \left\| \big( \widehat{\bB}_{\bar{\bx}(\tau_{n})} 
			- \widehat{\bB}_{mid} \big) \bE\big(\bx(\tau_{n} + \frac{\mathfrak{h}}{2})\big) \right\| 
		\end{aligned}\\[42pt]
		\begin{aligned}
			\displaystyle
			&\hspace{10mm} \eps\mh \left\| \bE\big(\bx(\tau_{n} + \frac{\mathfrak{h}}{2}) \big)^{\intercal} - \bE(\bar{\bx}(\tau_{n}))^{\intercal} \right\|
			+ \eps\mh^2 \left\| \bE\big(\bx(\tau_{n} + \frac{\mathfrak{h}}{2}) \big)^{\intercal} \big( \widehat{\bB}_{mid} -  \widehat{\bB}_{\bar{\bx}(\tau_{n})} \big) \right\| \\
			\displaystyle
			& + \eps\mh^2 \left\| \big( \bE\big(\bx(\tau_{n} + \frac{\mathfrak{h}}{2}) \big)^{\intercal}  - \bE(\bar{\bx}(\tau_{n}))^{\intercal} \big) \widehat{\bB}_{\bar{\bx}(\tau_{n})} \right\| 
			+ \eps^2\mh^2 \left\| \big(  \bE(\widetilde{\bx}^{\, n}(s))^{\intercal}
			- \bE(\bar{\bx}(\tau_{n}))^{\intercal} \big) \bE(\bar{\bx}(\tau_{n}))  \right\| \\
			\displaystyle
			&\hspace{20mm} + \eps^2\mh^2 \left\|   \bE(\widetilde{\bx}^{\, n} (s))^{\intercal}
			\big( \bE\big(\bx(\tau_{n} + \frac{\mathfrak{h}}{2})\big) - \bE(\bar{\bx}(\tau_{n})) \big) \right\|
		\end{aligned}
	\end{pmatrix}  
	\right\|  
	+ \mathcal{O}(\eps^2 \mh^3). \label{xiu2}
\end{align}
Recalling the definitions of relevant quantities and applying Taylor expansion together with bound \eqref{nablaB}, we derive the preliminary estimates
\begin{align}
	& \left\| \widehat{\bB}_{mid} - \widehat{\bB}_{\bar{\bx}(\tau_{n})} \right\|
	\lesssim \eps^2 \left\| \bx_{mid} - \bar{\bx}(\tau_{n}) \right\|
	= \eps^2 \left\|  \bx(\tau_{n} + \frac{\mathfrak{h}}{2}) - \bar{\bx}(\tau_{n})  \right\|	\notag \\
	&\quad \lesssim \eps^2 \left\| \frac{\eps\mh^2}{8}\dot{\bv}(\tau_{n}) -  \frac{\eps\mh^2}{8} \big( \widehat{\bB}(\eps\bx(\tau_{n}))\bv(\tau_{n}) -i\eps w(\tau_{n}) \bE(\bx(\tau_{n}))  \big)  + \mathcal{O}(\eps \mh^3) \right\| 
	\lesssim \eps^3\mh^3,  \label{bB1} \\
	& \left\| \big( \bE(\widetilde{\bx}^{\, n}(s)) -  \bE(\bar{\bx}(\tau_{n})) \big) \right\| 
	\lesssim \left\| \widetilde{\bx}^{\, n}(s) - \bar{\bx}(\tau_{n}) \right\| = \left\| \bx(\tau_{n}+s) - \zeta_{\bx}^{n}(s) - \bar{\bx}(\tau_{n}) \right\|   
    \lesssim \left\| \bx(\tau_{n}+s) - \bar{\bx}(\tau_{n}) \right\|
	+ \left\| \zeta_{\bx}^{n}(s) \right\| 
	\lesssim \eps\mh,  \notag \\
	& \left\| \bE(\bar{\bx}(\tau_{n})) - \bE\big(\bx(\tau_{n} + \frac{\mathfrak{h}}{2})\big) \right\| 
	\lesssim \left\| \bar{\bx}(\tau_{n}) -\bx(\tau_{n} + \frac{\mathfrak{h}}{2}) \right\|
	\lesssim \eps \mh^3. \notag
\end{align}
Substituting all the above bounds into \eqref{xiu2}, we obtain the final local error bound
\begin{equation*}
	\left\| \xi_{\bu}^{n} \right\| \lesssim \eps^2\mh^3,
	\quad 0 \leq n < \frac{T}{\varepsilon \mathfrak{h}}. 
\end{equation*}

$\bullet$   \textbf{The estimation of $\xi_{\by}^{n}$.} 
Substituting the refined expressions of $\widetilde{\bv}^{\, n}(s)$ (see \eqref{bv}) and $\widetilde{w}^{\, n}(s)$ (see \eqref{bw}) into  \eqref{bxa}-\eqref{btb}, we get the updated approximate solutions 
\begin{align}
	\widetilde{\bx}^{\, n}(\mh) & = \bx(\tau_{n}) + \eps \int_{0}^{\mh} \mathrm{e}^{s \widehat{\bB}_{mid}} \bv(\tau_{n}) \, ds 
	- i\eps^2 \int_{0}^{\mh} s \int_{0}^{1} \mathrm{e}^{\rho s \widehat{\bB}_{mid}} w(\tau_{n}) \bE\Big(\bx(\tau_{n} + \frac{s}{2}) \Big) \, d\rho \, ds + \mathcal{O}(\eps^3 \mh^3),  \label{bx} \\
	{\widetilde{\bar{t}}}^{\,\, n}(\mh)  & = \bar{t}(\tau_{n}) + \eps \int_{0}^{\mh} w(\tau_{n}) \, ds 
	+ i\eps^2 \int_{0}^{\mh} s \int^{1}_{0}
	\bE\Big(\bx(\tau_{n} + \frac{s}{2}) \Big)^{\intercal} \cdot \big( \mathrm{e}^{(1-\rho)s \widehat{\bB}_{mid}} \bv(\tau_{n}) \big) \, d\rho \, ds + \mathcal{O}(\eps^3 \mh^3).  \label{bt}
\end{align}
Subtracting \eqref{bby} from the combination of \eqref{bx} and \eqref{bt}, we arrive at the explicit representation for the local truncation error $\xi_{\by}^{n}$:
\begin{align}
	\xi_{\by}^{n} = & 
	\begin{pmatrix}
		\displaystyle 
		\eps \int_{0}^{\mh} \mathrm{e}^{s \widehat{\bB}_{mid}} \bv(\tau_{n}) \, ds 
		- i\eps^2 \int_{0}^{\mh} s \int_{0}^{1} \mathrm{e}^{\rho s \widehat{\bB}_{mid}} w(\tau_{n}) \bE\Big(\bx(\tau_{n} + \frac{s}{2}) \Big) \, d\rho \, ds  \\[8pt]
		\displaystyle
		\eps \int_{0}^{\mh} w(\tau_{n}) \, ds
		+ i\eps^2 \int_{0}^{\mh} s \int^{1}_{0}
		\bE\Big(\bx(\tau_{n} + \frac{s}{2}) \Big)^{\intercal} \cdot \big( \mathrm{e}^{(1-\rho)s \widehat{\bB}_{mid}} \bv(\tau_{n}) \big) \, d\rho \, ds
	\end{pmatrix}  \notag \\
	& \displaystyle  - \frac{\varepsilon \mh}{2}\varphi_{1}\biggl( \frac{\mh}{2} \bK_{\bx(\tau_{n})} \biggr)
	\Bigl(I + \mathrm{e}^{\mh (\bK_{\bar{\bx}(\tau_{n})} - \bK_{\bx(\tau_{n})}) }
	\mathrm{e}^{\frac{\mh}{2}\bK_{\bx(\tau_{n})}}\Bigr) \bu(\tau_{n})
	+ \mathcal{O}(\eps^3 \mh^3).   \label{xiy}
\end{align}
In a similar fashion, we take the norm on both sides of the error identity and expand the matrix exponential to deduce
\begin{align}
	\left\| \xi_{\by}^{n} \right\| \lesssim &
	\left\|
	\begin{pmatrix}
		\begin{aligned}
			\displaystyle
			&\hspace{4mm} \Big[ \eps\mh^2 \big( \widehat{\bB}_{mid} - \widehat{\bB}_{\bar{\bx}(\tau_{n})} \big) 
			+ \eps\mh^3 \big( \widehat{\bB}^{2}_{mid} - \widehat{\bB}^{2}_{\bx(\tau_{n})} \big) 
			+ \eps\mh^3 \big( \widehat{\bB}_{\bar{\bx}(\tau_{n})} - \widehat{\bB}_{\bx(\tau_{n})} \big) \widehat{\bB}_{\bar{\bx}(\tau_{n})} \\
			\displaystyle
			&\hspace{2mm} + \eps\mh^3 \widehat{\bB}_{\bx(\tau_{n})} \big( \widehat{\bB}_{\bar{\bx}(\tau_{n})} - \widehat{\bB}_{\bx(\tau_{n})} \big) \Big] \bv(\tau_{n}) 
			+ \Big[ i\eps^2\mh^2 \big( \bE(\bar{\bx}(\tau_{n})) - \bE\big(\bx(\tau_{n} + \frac{s}{2})\big) \big) \\
			\displaystyle
			& + i\eps^2\mh^3 \big( \widehat{\bB}_{\bx(\tau_{n})} \bE(\bx(\tau_{n})) - \widehat{\bB}_{mid} \bE\big(\bx(\tau_{n} + \frac{s}{2})\big) \big)  
			+ i\eps^2\mh^3 \big( \widehat{\bB}_{\bar{\bx}(\tau_{n})} - \widehat{\bB}_{\bx(\tau_{n})} \big) \bE(\bar{\bx}(\tau_{n}))  \\
			\displaystyle
			&\hspace{18mm} + i\eps^2\mh^3 \widehat{\bB}_{\bx(\tau_{n})} \big( \bE(\bar{\bx}(\tau_{n})) - \bE(\bx(\tau_{n})) \big)
			\Big] w(\tau_{n}) 
		\end{aligned}\\[39pt]
		\begin{aligned}
			\displaystyle
			& \Big[ i\eps^2\mh^2 \big( \bE\big(\bx(\tau_{n} + \frac{s}{2}) \big)^{\intercal} - \bE(\bar{\bx}(\tau_{n}))^{\intercal} \big) 
			+ i\eps^2\mh^3 \big( \bE\big(\bx(\tau_{n} + \frac{s}{2}) \big)^{\intercal} \widehat{\bB}_{mid} - \bE(\bx(\tau_{n}))^{\intercal} \widehat{\bB}_{\bx(\tau_{n})} \big) \\ \displaystyle
			&\hspace{2mm} + i\eps^2\mh^3 \big( \bE(\bar{\bx}(\tau_{n}))^{\intercal} - \bE(\bx(\tau_{n}))^{\intercal} \big) \widehat{\bB}_{\bar{\bx}(\tau_{n})}  
			+ i\eps^2\mh^3 \bE(\bx(\tau_{n}))^{\intercal} \big( \widehat{\bB}_{\bar{\bx}(\tau_{n})} - \widehat{\bB}_{\bx(\tau_{n})} \big)
			\Big] \bv(\tau_{n})
		\end{aligned} 
	\end{pmatrix}  
	\right\| + \mathcal{O}(\eps^3 \mh^3) \notag \\
	\lesssim & 
	\left\|
	\begin{pmatrix}
		\begin{aligned}
			\displaystyle
			& \eps\mh^2 \left\| \widehat{\bB}_{mid} - \widehat{\bB}_{\bar{\bx}(\tau_{n})} \right\| 
			+ \eps\mh^3 \left\| \big( \widehat{\bB}_{mid} - \widehat{\bB}_{\bx(\tau_{n})} \big) \widehat{\bB}_{mid} \right\| 
			+ \eps\mh^3 \left\| \widehat{\bB}_{\bx(\tau_{n})} \big( \widehat{\bB}_{mid} - \widehat{\bB}_{\bx(\tau_{n})} \big) \right\|  \\[1pt]
			\displaystyle
			& + \eps\mh^3  \left\| \big( \widehat{\bB}_{\bar{\bx}(\tau_{n})} - \widehat{\bB}_{\bx(\tau_{n})} \big) \widehat{\bB}_{\bar{\bx}(\tau_{n})} \right\| 
			+ \eps\mh^3 \left\| \widehat{\bB}_{\bx(\tau_{n})} \big( \widehat{\bB}_{\bar{\bx}(\tau_{n})} - \widehat{\bB}_{\bx(\tau_{n})} \big) \right\|  
			+ \eps^2\mh^2 \left\|  \bE(\bar{\bx}(\tau_{n})) - \bE\big(\bx(\tau_{n} + \frac{s}{2})\big)  \right\| \\
			\displaystyle
			&\hspace{4mm}  + \eps^2\mh^3 \left\| \widehat{\bB}_{\bx(\tau_{n})} \big( \bE(\bx(\tau_{n})) - \bE\big(\bx(\tau_{n} + \frac{s}{2})\big) \big) \right\| 
			+ \eps^2\mh^3 \left\| \big( \widehat{\bB}_{\bx(\tau_{n})}  - \widehat{\bB}_{mid} \big) \bE\big(\bx(\tau_{n} + \frac{s}{2})\big)  \right\|  \\[1pt]
			\displaystyle
			&\hspace{8mm}  + \eps^2\mh^3 \left\| \big( \widehat{\bB}_{\bar{\bx}(\tau_{n})} - \widehat{\bB}_{\bx(\tau_{n})} \big) \bE(\bar{\bx}(\tau_{n})) \right\|
			+ \eps^2\mh^3 \left\| \widehat{\bB}_{\bx(\tau_{n})} \big( \bE(\bar{\bx}(\tau_{n})) - \bE(\bx(\tau_{n})) \big) \right\| 
		\end{aligned}\\[42pt]
		\begin{aligned}
			\displaystyle
			&\hspace{6mm} \eps^2\mh^2 \left\| \bE\big(\bx(\tau_{n} + \frac{s}{2}) \big)^{\intercal} - \bE(\bar{\bx}(\tau_{n}))^{\intercal} \right\| + \eps^2\mh^3 \left\| \bE\big(\bx(\tau_{n} + \frac{s}{2}) \big)^{\intercal} \big( \widehat{\bB}_{mid} - \widehat{\bB}_{\bx(\tau_{n})} \big) \right\| \\[1pt]
			\displaystyle
			& + \eps^2\mh^3 \left\| \big( \bE\big(\bx(\tau_{n} + \frac{s}{2}) \big)^{\intercal} - \bE(\bx(\tau_{n}))^{\intercal} \big) \widehat{\bB}_{\bx(\tau_{n})} \right\| 
			+ \eps^2\mh^3 \left\| \big( \bE(\bar{\bx}(\tau_{n}))^{\intercal} - \bE(\bx(\tau_{n}))^{\intercal} \big) \widehat{\bB}_{\bar{\bx}(\tau_{n})} \right\| \\
			\displaystyle
			&\hspace{20mm} + \eps^2\mh^3 \left\| \bE(\bx(\tau_{n}))^{\intercal} \big( \widehat{\bB}_{\bar{\bx}(\tau_{n})} - \widehat{\bB}_{\bx(\tau_{n})} \big) \right\|
		\end{aligned}
	\end{pmatrix} 
	\right\| + \mathcal{O}(\eps^3 \mh^3).  \label{xiy2}
\end{align}
We next collect the required preliminary bounds:
\begin{equation*}
	\left\| \widehat{\bB}_{mid} - \widehat{\bB}_{\bx(\tau_{n})} \right\| 
	\lesssim \eps^2 \left\| \bx_{mid} - \bx(\tau_{n}) \right\| 
	\lesssim \eps^3 \mh, 
    \quad
	\left\|  \widehat{\bB}_{\bar{\bx}(\tau_{n})} - \widehat{\bB}_{\bx(\tau_{n})}  \right\| \lesssim \eps^2 \left\| \bar{\bx}(\tau_{n}) - \bx(\tau_{n}) \right\|
	\lesssim \eps^3 \mh, 
\end{equation*}
\begin{equation*}
	\left\|  \bE(\bar{\bx}(\tau_{n})) - \bE\big(\bx(\tau_{n} + \frac{s}{2})\big)  \right\| 
	\lesssim \left\| \bar{\bx}(\tau_{n}) - \bx(\tau_{n} + \frac{s}{2}) \right\|  
	\lesssim \eps \mh,
    \quad
	\left\| \bE\big(\bx(\tau_{n} + \frac{s}{2})\big) - \bE(\bx(\tau_{n})) \right\|
	\lesssim \left\| \bx(\tau_{n} + \frac{s}{2}) - \bx(\tau_{n}) \right\|
	  \lesssim \eps \mh,
\end{equation*}
\begin{equation*}
	\left\| \bE(\bar{\bx}(\tau_{n})) - \bE(\bx(\tau_{n})) \right\|  \lesssim
	\left\| \bar{\bx}(\tau_{n}) - \bx(\tau_{n}) \right\| \lesssim \eps \mh.
\end{equation*}
Combining these estimates with \eqref{bB1} and substituting all bounds into \eqref{xiy2}, we obtain the final local truncation error bound 
\begin{equation*}
	\left\| \xi_{\by}^{n} \right\| \lesssim \eps^3\mh^3,
	\quad 0 \leq n < \frac{T}{\varepsilon \mathfrak{h}}. 
\end{equation*}	
The proof of Lemma \ref{lemmalocal} is complete. 
\end{proof}

\begin{proof}[\textbf{Proof of Lemma \ref{global}}]
	The assertion trivially holds for $n=0$ owing to the initial conditions. We proceed to prove the lemma by mathematical induction. Assume that the error estimates and boundedness of numerical solutions are valid for all $n \leq m< \frac{T}{\varepsilon \mathfrak{h}}$. It suffices to verify the desired conclusions for $n=m+1$ to complete the induction argument.
	
	By subtracting the discrete SS2-xn scheme \eqref{SS2s} from \eqref{bby}-\eqref{bbu} and incorporating the auxiliary identity \eqref{eyu}, we obtain the following recursive error equations
	\begin{subequations}
		\begin{align}
			e_{\by}^{n+1} & = e_{\by}^n + \frac{\varepsilon \mh}{2}\varphi_{1}\biggl( \frac{\mh}{2} \bK_{\bx(\tau_{n})} \biggr)
			\Bigl(I_4 + \mathrm{e}^{\mh (\bK_{\bar{\bx}(\tau_{n})} - \bK_{\bx(\tau_{n})}) }
			\mathrm{e}^{\frac{\mh}{2}\bK_{\bx(\tau_{n})}} \Bigr)  e_{\bu}^n 
			+ \xi_{\by}^n + \zeta_{\by}^n(\mathfrak{h}) + \eta_{\by}^n, \label{eya} \\
			e_{\bu}^{n+1} & = \mathrm{e}^{\frac{\mh}{2}\bK_{\bx(\tau_{n})}}
			\mathrm{e}^{\mh (\bK_{\bar{\bx}(\tau_{n})} - \bK_{\bx(\tau_{n})})}
			\mathrm{e}^{\frac{\mh}{2}\bK_{\bx(\tau_{n})}} e_{\bu}^n
			+ \xi_{\bu}^n + \zeta_{\bu}^n(\mathfrak{h}) + \eta_{\bu}^n, \quad 0 \leq n \leq m. \label{eub}
		\end{align}
	\end{subequations}
	The additional perturbation terms $\eta_{\by}^n$ and $\eta_{\bu}^n$ involved in the above error evolution formulas are explicitly defined as
	\begin{align*}
		\eta_{\by}^{n} = &\, \frac{\varepsilon \mh}{2}\varphi_{1} \biggl( \frac{\mh}{2} \bK_{\bx(\tau_{n})} \biggr) 
		\Big( I_4 + \mathrm{e}^{\mh (\bK_{\bar{\bx}(\tau_{n})} - \bK_{\bx(\tau_{n})}) }
		\mathrm{e}^{\frac{\mh}{2}\bK_{\bx(\tau_{n})}} \Big) \bu^{n}   
		- \frac{\varepsilon \mh}{2}\varphi_{1} \biggl( \frac{\mh}{2} \bK^{n} \biggr)  \Big( I_4 + \mathrm{e}^{\mh (\bK_{\bar{\bx}^{n}} - \bK^{n}) }
		\mathrm{e}^{\frac{\mh}{2}\bK^n} \Big)  \bu^{n}, \\
		\eta_{\bu}^{n} = &\, \Big( \mathrm{e}^{\frac{\mh}{2}\bK_{\bx(\tau_{n})}}
		\mathrm{e}^{\mh (\bK_{\bar{\bx}(\tau_{n})} - \bK_{\bx(\tau_{n})})}
		\mathrm{e}^{\frac{\mh}{2}\bK_{\bx(\tau_{n})}} 
		- \mathrm{e}^{\frac{\mh}{2}\bK^n}
		\mathrm{e}^{\mh (\bK_{\bar{\bx}^{n}} - \bK^{n})}
		\mathrm{e}^{\frac{\mh}{2}\bK^n} \Big) \bu^{n}.
	\end{align*}
	
	We first derive the upper bounds for the two foregoing perturbation terms. Thanks to the available estimates of 
	\begin{equation*}
		\left\|  \widehat{\bB}(\eps \bar{\bx}(\tau_n)) - \widehat{\bB}(\eps \bar{\bx}^n)  \right\|  \lesssim \eps^2 \left\| \bar{\bx}(\tau_n) - \bar{\bx}^n \right\| \lesssim \eps^2 \big( \left\| e_{\bx}^{n} \right\| + \eps\mh \left\| e_{\bv}^{n} \right\| \big)
		\lesssim \eps^2 \big( \left\| e_{\by}^{n} \right\| + \eps\mh \left\| e_{\bu}^{n} \right\| \big),
	\end{equation*}
	\begin{equation*}
		\left\| \bE(\bar{\bx}(\tau_n)) - \bE(\bar{\bx}^n)  \right\|  \lesssim 
		\left\| \bar{\bx}(\tau_n) - \bar{\bx}^n \right\| \lesssim  \left\| e_{\bx}^{n} \right\| + \eps\mh \left\| e_{\bv}^{n} \right\| 
		\lesssim \left\| e_{\by}^{n} \right\| + \eps\mh \left\| e_{\bu}^{n} \right\|,
	\end{equation*}
	the desired bounds $\eta_{\by}^n$ and $\eta_{\bu}^n$ can be consequently established as
	\begin{align}
		\left\| \eta_{\by}^n \right\| &\lesssim \eps\mh \cdot \mh \left\| \bK_{\bar{\bx}(\tau_{n})} - \bK_{\bar{\bx}^{n}} \right\|  
		\lesssim \eps\mh^2 
		\left\| 
		\begin{pmatrix}
			\widehat{\bB}(\eps \bar{\bx}(\tau_n)) - \widehat{\bB}(\eps \bar{\bx}^n) & -i\eps \big( \bE(\bar{\bx}(\tau_n)) - \bE(\bar{\bx}^n) \big) \\[2pt]
			i\eps \big( \bE(\bar{\bx}(\tau_n))^{\intercal} - \bE(\bar{\bx}^n)^{\intercal} \big) & 0
		\end{pmatrix} 
		\right\| \notag \\
		&\lesssim \eps\mh^2 
		\left\| 
		\begin{pmatrix}
			\left\| \widehat{\bB}(\eps \bar{\bx}(\tau_n)) - \widehat{\bB}(\eps \bar{\bx}^n) \right\| & \eps \left\| \bE(\bar{\bx}(\tau_n)) - \bE(\bar{\bx}^n) \right\| \\[4pt]
			\eps \left\| \bE(\bar{\bx}(\tau_n))^{\intercal} - \bE(\bar{\bx}^n)^{\intercal} \right\| & 0
		\end{pmatrix}  
		\right\|  
		\lesssim \eps^2\mh^2  \big( \left\| e_{\by}^{n} \right\| + \eps\mh \left\| e_{\bu}^{n} \right\| \big),  \label{etaby} \\
		\left\| \eta_{\bu}^n \right\| &\lesssim \mh \left\| \bK_{\bar{\bx}(\tau_{n})} - \bK_{\bar{\bx}^{n}} \right\|
		\lesssim \eps\mh \big( \left\| e_{\by}^{n} \right\| + \eps\mh \left\| e_{\bu}^{n} \right\| \big).  \label{etabu}
	\end{align}
	We proceed to the global error analysis. Taking the norm on both sides of the error recurrence relations \eqref{eya}-\eqref{eub} and exploiting the orthogonality property of matrix $\mathrm{e}^{\mh\bK}$, we arrive at the estimates 
	\begin{align*}
		\left\| e_{\by}^{n+1} \right\| & \lesssim \left\| e_{\by}^{n} \right\| + \varepsilon{\mathfrak{h}} \left\| e_{\bu}^{n} \right\| + \left\| \xi_{\by}^{n} \right\| + \left\| \zeta_{\by}^{n}({\mathfrak{h}}) \right\| + \left\|\eta_{\by}^{n} \right\|,   \\
		\left\| e_{\bu}^{n+1} \right\| & \lesssim \left\| e_{\bu}^{n} \right\|
		+ \left\| \xi_{\bu}^{n} \right\| + \left\| \zeta_{\bu}^{n}({\mathfrak{h}}) \right\| + \left\|\eta_{\bu}^{n} \right\|,
		\quad 0\leq n\leq m.
	\end{align*} 
	Adding up the above two estimates and inserting \eqref{etaby}-\eqref{etabu}, one can derive the desired estimate 
	\begin{equation*}
		\left\| e_{\by}^{n+1} \right\| + \left\| e_{\bu}^{n+1} \right\| - \left\| e_{\by}^{n} \right\| - \left\| e_{\bu}^{n} \right\|
		\lesssim \varepsilon \mathfrak{h} \big( \left\| e_{\by}^{n} \right\| + \left\| e_{\bu}^{n} \right\| \big) + \left\| \xi_{\by}^{n} \right\| + \left\| \xi_{\bu}^{n} \right\| + \left\| \zeta_{\by}^{n}(\mathfrak{h}) \right\| + \left\| \zeta_{\bu}^{n}(\mathfrak{h}) \right\|.
	\end{equation*}
	We sum the recursive error relations from step \(n=0\) to \(n=m\). Since no error occurs at the initial time level, i.e., \(e_{\by}^{0} = e_{\bu}^{0} = 0\), we arrive at
	\begin{equation*}
		\left\| e_{\by}^{m+1} \right\| + \left\| e_{\bu}^{m+1} \right\| \lesssim \varepsilon \mathfrak{h} \sum_{n=0}^{m} \big( \left\| e_{\by}^{n} \right\| + \left\| e_{\bu}^{n} \right\| \big)
		+ \sum_{n=0}^{m} \big( \left\| \xi_{\by}^{n} \right\| + \left\| \xi_{\bu}^{n} \right\| + \left\| \zeta_{\by}^{n}(\mathfrak{h}) \right\|
		+ \left\| \zeta_{\bu}^{n}(\mathfrak{h}) \right\| \big). 
	\end{equation*}
	Substituting the bounds \eqref{M} and \eqref{local}, and combining with the fact that \(m\varepsilon{\mathfrak{h}} < T \lesssim 1\), the above inequality can be simplified as
	\begin{equation*}
		\left\| e_{\by}^{m+1} \right\| + \left\| e_{\bu}^{m+1} \right\| \lesssim \varepsilon \mathfrak{h} \sum_{n=0}^{m} \big( \left\| e_{\by}^{n} \right\| + \left\| e_{\bu}^{n} \right\| \big) + \varepsilon {\mathfrak{h}}^{2}, 
		\quad 0\leq m < \frac{T}{\varepsilon{\mathfrak{h}}}. 
	\end{equation*}
	At this stage, we apply the Gronwall inequality to obtain the final global error bound
	\begin{equation*}
		\left\| e_{\by}^{m+1} \right\| + \left\| e_{\bu}^{m+1} \right\| \lesssim \varepsilon \mathfrak{h}^{2}, 
		\quad 0\leq m < \frac{T}{\varepsilon{\mathfrak{h}}}.
	\end{equation*}
	Based on the error decomposition, we further evaluate the boundedness of numerical solutions. It follows that
	$
	\left\| \by^{m+1} \right\| \leq \left\| \by(\tau_{m+1}) \right\| + \left\| e_{\by}^{m+1} \right\| 
	\lesssim \left\| \by(\tau_{m+1}) \right\| + 1, 
    $
    $
	\left\| \bu^{m+1} \right\| \leq \left\| \bu(\tau_{m+1}) \right\| + \left\| e_{\bu}^{m+1} \right\| \lesssim 
	\left\| \bu(\tau_{m+1}) \right\| + 1.
	$
	As a result, there exists a constant $\mathfrak{h}_{0}>0$, independent of $\varepsilon$ and $m$, such that whenever $0<\mathfrak{h}\le\mathfrak{h}_{0}$, the statements of Lemma \ref{global} hold for $m+1$. By mathematical induction, the Lemma is valid for all relevant indices.
\end{proof}

\section*{Acknowledgements}

\begin{itemize}
	\item \textbf{Funding}: This work was supported partially by the National Natural Science Foundation of China (Grants No. 12371403, 92470119) and Shaanxi Fundamental Science Research Project for Mathematics and Physics (Grant No. 25JSY046).

	\item \textbf{Competing interests}: We declare that we have no conflict of interest.
	%	\item \textbf{Ethics approval and consent to participate}: Not applicable.
	%	\item \textbf{Consent for publication}: Not applicable.
	%	\item \textbf{Data availability}: Not applicable.
	%	\item \textbf{Materials availability}: Not applicable.
	%	\item \textbf{Code availability}: Not applicable.
	%	\item \textbf{Author contribution}: Not applicable.
\end{itemize}

\bibliographystyle{elsarticle-num}
\bibliography{tinging}%

\end{document}